\documentclass[10 pt]{amsart}

\usepackage{lmodern}
\usepackage[usenames,dvipsnames]{color}
\usepackage[T1]{fontenc}
\usepackage{amsthm,amsfonts,amssymb,amsmath,amsxtra}
\usepackage[all]{xy}
\usepackage{leftidx}
\SelectTips{cm}{}
\usepackage{xr-hyper} 
\usepackage[colorlinks=
citecolor=Black,
linkcolor=Red,
urlcolor=Blue]{hyperref}
\usepackage{verbatim}
\usepackage{mathrsfs}

\usepackage{dynkin-diagrams}
\usepackage{enumerate}
\DeclareUnicodeCharacter{2212}{-}
\newcommand{\remind}[1]{{\bf ** #1 **}}

\RequirePackage{xspace}
\RequirePackage{etoolbox}
\RequirePackage{varwidth}
\RequirePackage{enumitem}
\RequirePackage{tensor}
\RequirePackage{mathtools}
\RequirePackage{longtable}
\RequirePackage{multirow}

\def\le{\leqslant}

\def\s{\sigma}
\def\t{\tau}

\def\k{\kappa}
\def\l{\lambda}

\def\i{^{-1}}

\def\<{\langle}
\def\>{\rangle}
\newtheorem*{introthmA}{Theorem A}
\newtheorem*{introthmB}{Theorem B}
\newcommand{\ignore}[1]{} 
\newcommand{\dbp}[1]{[\! [ #1 ]\! ] }
\newcommand{\Sh}{\ensuremath{\underline{\mathrm{Sh}}}}
\newcommand{\pfn}{\mathrm{pfn}}
\newcommand{\sA}{\ensuremath{\mathscr{A}}\xspace}

\newcommand{\sC}{\ensuremath{\mathscr{C}}\xspace}

\newcommand{\sG}{\ensuremath{\mathscr{G}}\xspace}

\newcommand{\sS}{\ensuremath{\mathscr{S}}\xspace}

\newcommand{\fka}{\ensuremath{\mathfrak{a}}\xspace}

\newcommand{\fks}{\ensuremath{\mathfrak{s}}\xspace}

\newcommand{\bT}{\ensuremath{\mathbf{T}}}

\newcommand{\fkC}{\ensuremath{\mathfrak{C}}\xspace}

\newcommand{\bE}{\mathbf E}
\newcommand{\bG}{\mathbf G}

\newcommand{\BA}{\ensuremath{\mathbb {A}}\xspace}

\newcommand{\BC}{\ensuremath{\mathbb {C}}\xspace}
\newcommand{\BD}{\ensuremath{\mathbb {D}}\xspace}

\newcommand{\BF}{\ensuremath{\mathbb {F}}\xspace}
\newcommand{{\BG}}{\ensuremath{\mathbb {G}}\xspace}
\newcommand{\BH}{\ensuremath{\mathbb {H}}\xspace}

\newcommand{{\BK}}{\ensuremath{\mathbb {K}}\xspace}
\newcommand{\BL}{\ensuremath{\mathbb {L}}\xspace}
\newcommand{\BM}{\ensuremath{\mathbb {M}}\xspace}
\newcommand{\BN}{\ensuremath{\mathbb {N}}\xspace}

\newcommand{\BQ}{\ensuremath{\mathbb {Q}}\xspace}
\newcommand{\BR}{\ensuremath{\mathbb {R}}\xspace}
\newcommand{\BS}{\ensuremath{\mathbb {S}}\xspace}
\newcommand{\BT}{\ensuremath{\mathbb {T}}\xspace}

\newcommand{\BV}{\ensuremath{\mathbb {V}}\xspace}
\newcommand{\BW}{\ensuremath{\mathbb {W}}\xspace}

\newcommand{\BZ}{\ensuremath{\mathbb {Z}}\xspace}

\newcommand{\CA}{\ensuremath{\mathcal {A}}\xspace}

\newcommand{\CG}{\ensuremath{\mathcal {G}}\xspace}

\newcommand{\CI}{\ensuremath{\mathcal {I}}\xspace}
\newcommand{\CJ}{\ensuremath{\mathcal {J}}\xspace}
\newcommand{\CK}{\ensuremath{\mathcal {K}}\xspace}

\newcommand{\CN}{\ensuremath{\mathcal {N}}\xspace}
\newcommand{\CO}{\ensuremath{\mathcal {O}}\xspace}

\newcommand{\CS}{\ensuremath{\mathcal {S}}\xspace}

\newcommand{\CX}{\ensuremath{\mathcal {X}}\xspace}

\newcommand{\CZ}{\ensuremath{\mathcal {Z}}\xspace}

\newcommand{\RH}{\ensuremath{\mathrm {H}}\xspace}
\newcommand{\RI}{\ensuremath{\mathrm {I}}\xspace}

\newcommand{\RK}{\ensuremath{\mathrm {K}}\xspace}

\newcommand{\Ad}{{\mathrm{Ad}}}
\newcommand{\ad}{{\mathrm{ad}}}

\DeclareMathOperator{\Aut}{Aut}

\DeclareMathOperator{\charac}{char}

\DeclareMathOperator{\End}{End}

\newcommand{\GDL}{\mathrm{GDL}}
\newcommand{\TIC}{\mathrm{TIC}}

\DeclareMathOperator{\Gal}{Gal}
\newcommand{\GL}{\mathrm{GL}}

\newcommand{\id}{\ensuremath{\mathrm{id}}\xspace}

\DeclareMathOperator{\Nm}{Nm}

\DeclareMathOperator{\rank}{rank}

\newcommand{\PGL}{{\mathrm{PGL}}}

\DeclareMathOperator{\Res}{Res}

\DeclareMathOperator{\Spec}{Spec}

\newcommand{\ur}{{\mathrm{ur}}}

\DeclareMathOperator{\vol}{vol}

\DeclareMathOperator{\topp}{top}

\DeclareMathOperator{\Stab}{Stab}

\def\tS{\breve {\BS}}
\def\kk{\mathbf k}
\DeclareMathOperator{\supp}{supp}

\newtheorem{theorem}[subsubsection]{Theorem}
\newtheorem{proposition}[subsubsection]{Proposition}
\newtheorem{lemma}[subsubsection]{Lemma}

\newtheorem{corollary}[subsubsection]{Corollary}

\theoremstyle{definition}
\newtheorem{definition}[subsubsection]{Definition}

\theoremstyle{remark}
\newtheorem{remark}[subsubsection]{Remark}

\numberwithin{equation}{subsubsection}

\setitemize[0]{leftmargin=*,itemsep=\the\smallskipamount}
\setenumerate[0]{leftmargin=*,itemsep=\the\smallskipamount}

\renewcommand{\to}{%
	\ifbool{@display}{\longrightarrow}{\rightarrow}%
}
\let\shortmapsto\mapsto
\renewcommand{\mapsto}{%
	\ifbool{@display}{\longmapsto}{\shortmapsto}%
}
\newlength{\olen}
\newlength{\ulen}
\newlength{\xlen}
\newcommand{\xra}[2][]{%
	\ifbool{@display}%
	{\settowidth{\olen}{$\overset{#2}{\longrightarrow}$}%
		\settowidth{\ulen}{$\underset{#1}{\longrightarrow}$}%
		\settowidth{\xlen}{$\xrightarrow[#1]{#2}$}%
		\ifdimgreater{\olen}{\xlen}%
		{\underset{#1}{\overset{#2}{\longrightarrow}}}%
		{\ifdimgreater{\ulen}{\xlen}%
			{\underset{#1}{\overset{#2}{\longrightarrow}}}
			{\xrightarrow[#1]{#2}}}}%
	{\xrightarrow[#1]{#2}}
}
\makeatother
\newcommand{\xyra}[2][]{%
	\settowidth{\xlen}{$\xrightarrow[#1]{#2}$}%
	\ifbool{@display}%
	{\settowidth{\olen}{$\overset{#2}{\longrightarrow}$}%
		\settowidth{\ulen}{$\underset{#1}{\longrightarrow}$}%
		\ifdimgreater{\olen}{\xlen}%
		{\mathrel{\xymatrix@M=.12ex@C=3.2ex{\ar[r]^-{#2}_-{#1} &}}}%
		{\ifdimgreater{\ulen}{\xlen}%
			{\mathrel{\xymatrix@M=.12ex@C=3.2ex{\ar[r]^-{#2}_-{#1} &}}}
			{\mathrel{\xymatrix@M=.12ex@C=\the\xlen{\ar[r]^-{#2}_-{#1} &}}}}}%
	{\mathrel{\xymatrix@M=.12ex@C=\the\xlen{\ar[r]^-{#2}_-{#1} &}}}%
}
\makeatletter
\newcommand{\xla}[2][]{%
	\ifbool{@display}%
	{\settowidth{\olen}{$\overset{#2}{\longleftarrow}$}%
		\settowidth{\ulen}{$\underset{#1}{\longleftarrow}$}%
		\settowidth{\xlen}{$\xleftarrow[#1]{#2}$}%
		\ifdimgreater{\olen}{\xlen}%
		{\underset{#1}{\overset{#2}{\longleftarrow}}}%
		{\ifdimgreater{\ulen}{\xlen}%
			{\underset{#1}{\overset{#2}{\longleftarrow}}}
			{\xleftarrow[#1]{#2}}}}%
	{\xleftarrow[#1]{#2}}
}
\newcommand{\isoarrow}{%
	\ifbool{@display}{\overset{\sim}{\longrightarrow}}{\xrightarrow\sim}%
}

\begin{document}
	
	\title[Stabilizers of irreducible components]{Stabilizers of irreducible components of affine Deligne--Lusztig varieties}

	\author[Xuhua He]{Xuhua He}
	\address{The Institute of Mathematical Sciences and Department of Mathematics, The Chinese University of Hong Kong, Shatin, N.T., Hong Kong SAR, China}
	\email{xuhuahe@math.cuhk.edu.hk}

	\author[Rong Zhou]{Rong Zhou}
	\address{Department of Pure Mathematics and Mathematical Statistics, Wilberforce Road, Cambridge, United Kingdom, CB3 0WB.}
	\email{rz240@dpmms.cam.ac.uk}

	\author[Yihang Zhu]{Yihang Zhu}
	\address{Department of Mathematics, University of Maryland, 4176 Campus Drive,
		College Park, MD 20742, USA}
	\email{yhzhu@umd.edu}
	\makeatletter
	\@namedef{subjclassname@2020}{%
		\textup{2020} Mathematics Subject Classification}
	\makeatother
	
	\keywords{Affine Deligne--Lusztig varieties, Shimura varieties}
	\subjclass[2020]{14G35}
	
	\date{\today}
	
	\begin{abstract}We study the $J_b(F)$-action on the set of top-dimensional irreducible components of affine Deligne--Lusztig varieties in the affine Grassmannian. We show that the stabilizer of any such component is a parahoric subgroup of $J_b(F)$ of maximal volume, verifying a conjecture of X.~Zhu. As an application, we give a description of the set of top-dimensional irreducible components in the basic locus of Shimura varieties.
	\end{abstract}
	
	\maketitle
	
	\tableofcontents

	\section{Introduction}

	\subsection{Affine Deligne--Lusztig varieties and their irreducible components} Affine Deligne--Lusztig varieties were introduced by Rapoport in \cite{RapGuide}. In the equal characteristic setting, affine Deligne--Lusztig varieties are related to the moduli space of local shtukas. In the mixed characteristic setting, they are related to the geometry of Rapoport--Zink spaces and hence to the geometry of certain distinguished loci in the special fiber of Shimura varieties via the $p$-adic uniformization. Therefore studying the geometry of affine Deligne--Lusztig varieties can give useful information on the geometry of special cycles on Shimura varieties. 
	
	This paper is concerned with studying the set of top-dimensional irreducible components of affine Deligne--Lusztig varieties. To state our main results we fix some notation. Let $F$ be a local field with ring of integers $\CO_F$, and let $\breve F$ be the completion of the maximal unramified extension of $F$. Let $G$ be a reductive group over $F$, which we assume is unramified in the introduction for simplicity. For $b\in G(\breve F)$ and $\mu$ a  cocharacter of $G$,  we have the affine Deligne--Lusztig variety $X_\mu(b)$ which is a locally closed subscheme of the affine Grassmannian. We refer to \S\ref{sssec: ADLV} for the precise definition.

	If $F$ is of equal characteristic,  $X_\mu(b)$ is locally of finite type. If $F$ is of mixed characteristic, $X_\mu(b)$ is a perfect scheme and is locally of perfectly finite type. In either case, it is known that $X_\mu(b)$ is finite dimensional. We write $\Sigma^{\topp}(X_\mu(b))$ for the set of top-dimensional irreducible components of $X_{\mu} (b)$. 
	
	The scheme $X_\mu(b)$ is equipped with an action of $J_b(F)$, the $F$-rational points of a certain reductive group $J_b$ over $F$ (the Frobenius-centralizer of $b$). This induces an action of $J_b(F)$ on $\Sigma^{\topp}(X_\mu(b))$. The goal of this paper is to understand the $J_b(F)$-set $\Sigma^{\topp}(X_\mu(b))$. This amounts to considering the following two problems.
	
	\begin{enumerate}[label=(\roman*)]
		\item Classify the $J_b(F)$-orbits in $\Sigma^{\topp}(X_\mu(b))$.
		
		\item For each $Z\in \Sigma^{\topp}(X_\mu(b))$, determine the stabilizer of $Z$ in $J_b(F)$.
	\end{enumerate}
	
	For (i), M.~Chen and X.~Zhu conjectured (see \cite[Conjecture 1.3]{HV}) that the set of  the $J_b(F)$-orbits in $\Sigma^{\topp}(X_\mu(b))$ should be in natural bijection with the Mirkovic--Vilonen basis $\BM\BV_\mu(\lambda_b)$ for a certain weight space of a representation of the dual group $\widehat G$. (See \S \ref{subsec: MV} for the definition of $\BM\BV_\mu(\lambda_b)$.) Special cases of this conjecture was proved by Xiao--Zhu \cite{XZ}, Hamacher--Viehmann \cite{HV} and Nie \cite{nie}. The conjecture was finally proved by Nie \cite{nienew}, and by the second and third authors \cite{ZZ} using different methods. 
	
	For (ii), Xiao--Zhu \cite[Theorem 4.4.14]{XZ} showed that if the element $b\in G(\breve F)$ is unramified, then the stabilizer of every $Z\in \Sigma^{\topp}(X_\mu(b))$ is a hyperspecial subgroup of $J_b(F)$ (see also \cite[Theorem 6.2.2]{ZZ}). For general $b$, it was conjectured by X.~Zhu\footnote{Private communication.} that every stabilizer should be a parahoric subgroup of $J_b(F)$ of maximal volume. \footnote{This conjecture implies that all the stabilizers have the same volume. The latter statement was also conjectured by M.~Rapoport.}  Our first main result confirms this conjecture. 
	
	\begin{introthmA}[See Theorem \ref{thm:main} and Corollary \ref{cor: main}] For each $Z\in \Sigma^{\topp}(X_\mu(b))$, the stabilizer of $Z$ in $J_b(F)$ is a \emph{very special} parahoric subgroup of $J_b(F)$. In particular, there is an isomorphism of $J_b(F)$-sets $$\Sigma^{\topp}(X_\mu(b)) \cong \coprod_{\mathbf{a}\in\BM\BV_\mu(\lambda_b)} J_b(F)/\CJ^{\mathbf{a}}.$$
		where $\CJ^{\mathbf{a}}\subset J_b(F)$ is a very special parahoric subgroup. 
	\end{introthmA}
	We refer to \S\ref{sssec: very special} for the definition of very special parahoric subgroups,  and Proposition \ref{prop: eq. v. special, max vol, max log vol} for the equivalence of this condition with that of having maximal volume. After this result was announced, S.~Nie informed us that he could also prove this result using a different method.
	
	For a reductive group over $F$ with no factors of type $C$-$BC_n$, the condition that a parahoric is very special determines the parahoric up to conjugation in the adjoint group. Thus when $J_b$ has no factors of type $C$-$BC_n$, Theorem A determines the stabilizers up to conjugation by $J_b^{\ad}(F)$. It is an interesting problem to determine the stabilizers up to $J_{b}(F)$-conjugacy. However Theorem A is already enough for some important applications explained below. 
	
	\subsection{Application to Shimura varieties} Let $(\bG,X)$ be a Shimura datum, and let $\RK\subset \bG(\BA_f)$ be a  sufficiently small compact open subgroup. Then we have the associated Shimura variety  $\mathrm{Sh}_{\RK}(\bG,X)$  which is an algebraic variety defined over a number field $\bE$. Let $p>2$ be a prime.  We assume that $(\bG,X)$ is of Hodge type, and that $\RK=\RK_p\RK^p$ where $\RK^p$ is a compact open subgroup of $ \bG(\BA_f^p)$ and $\RK_p$ is a hyperspecial subgroup of $\bG(\BQ_p)$.  Then by work of Kisin \cite{KisinIntMod}, for any prime $v|p$ of $\bE$, there is a smooth canonical integral model $\mathscr{S}_{\RK}(\bG,X)$ of $\mathrm{Sh}_{\RK}(\bG,X)$ over $\CO_{\bE_{(v)}}$. We write $\Sh_{\RK}$ for its special fiber. 
	
	Write $G$ for $G=\bG_{\BQ_p}$. There is a stratification of $\Sh_{\RK}$ indexed by the Kottwitz set $B(G,\mu)$ (cf.  \S\ref{sssec: Kottwitz set}).  We let $[b]_{\mathrm{bas}}$ denote the unique basic element of $B(G,\mu)$,   and we write  $\Sh_{\RK,\mathrm{bas}}$ for the stratum corresponding to $[b]_{\mathrm{bas}}$. This is known as the \emph{basic locus}, and is a generalization of the supersingular locus in the special fiber of a modular curve. The Rapoport--Zink uniformization (see e.g.~\cite[Corollary 7.2.16]{XZ}) implies that there is an isomorphism of perfect schemes 
	\begin{equation}\tag{$\dagger$}\Sh_{\RK,\mathrm{bas}}^{\mathrm{pfn}}\cong I(\BQ)\backslash X_\mu(b)\times \mathbf{G}(\BA_f^p)/\RK^p.
	\end{equation}
	Here  $I$ is a certain reductive group over $\BQ$ with $I\otimes_{\BQ}\BA_f^p\cong \bG\otimes_{\BQ}\BA_f^p$ and $I\otimes_{\BQ}\BQ_p\cong J_b$, and the left hand side denotes the perfection of $\Sh_{\RK,\mathrm{bas}}$. The following theorem then follows immediately from Theorem A and the above isomorphism.
	\begin{introthmB}[See Corollary \ref{thm: irred comp basic locus}]There exists a bijection between the set of top-dimensional irreducible components of $\Sh_{\RK,\mathrm{bas}}$ and the set $$ \coprod_{\mathbf{a}\in\BM\BV_\mu(\lambda_b)}I(\BQ)\backslash I(\BA_f)/\RI_{p}^{{\mathbf{a}}}\RI^p,$$
		where $\RI^p\cong \RK^p$ and $\RI_{p}^{\mathbf{a}}$ is a very special parahoric subgroup of $J_b(\BQ_p)$. Moreover the bijection is equivariant for prime-to-$p$ Hecke operators.
	\end{introthmB} 
In fact, $\Sh_{\RK,\mathrm{bas}}$ is equidimensional by \cite[Theorem 3.4]{HV}, so we have obtained a description of the set of all irreducible components in this case.
We also remark that for Theorem A, the assumption that $G$ is unramified over $F$ is not necessary, and we in fact obtain results for general quasi-split $G$ over $F$. This allows us to obtain a generalization of Theorem B. The key input for this is  a generalization of ($\dagger$) for the integral models constructed by Kisin--Pappas \cite{KP}, which we prove in \S\ref{sec: Shimura}.

Theorem B and its generalization reflect the general philosophy going back to Serre and Deuring that components of the basic locus are parameterized by class sets for an inner form of  the structure group. We refer to \cite{VW}, \cite{HP}, and \cite{LT} for some special cases of this result.

	The main contribution of this paper is the information that the compact open subgroups $\RI_p^{\mathbf{a}}\subset I(\BQ_p)$ are very special.  For many applications this is a crucial piece of information. For example, in \cite{LT}, the authors used the description of irreducible components in the supersingular locus of quaternionic Shimura varieties to prove an arithmetic level raising result on the way to proving cases of the Beilinson--Bloch--Kato conjecture. For this, they used the interpretation of functions on $\Sigma^{\topp}(\Sh_{\RK,\mathrm{bas}})$  as automorphic forms for $I$. Thus the knowledge of $\RI_{p}^{\mathbf{a}}$ is needed to determine the level of these automorphic forms. 
	In \cite{LMPT}, the authors used a formula for  the number of irreducible components in the supersingular locus of unitary Shimura varieties to prove results on the image of the Torelli map. This requires information on the volume of $\RI_p^{\mathbf{a}}$.
	
	\subsection{The proof of Theorem A}
	Our proof of Theorem A makes use of techniques from $p$-adic harmonic analysis developed in \cite{ZZ}, and the Deligne--Lusztig reduction method for affine Deligne--Lusztig varieties developed in \cite{He14}. For simplicity in the introduction, we assume that $G$ has no factors of type $A$ or $E_6$. After a series of reduction steps, we can assume that $G$ is an unramified adjoint group over $F$, that $F$ has characteristic 0, and that $b\in B(G,\mu)$ is basic. It is known that the stabilizer of every $Z \in \Sigma^{\topp}(X_\mu(b)) $ is a parahoric subgroup of $J_b(F)$, so the question is to prove that such a parahoric subgroup must have maximal volume. The proof proceeds in two steps.
	
	\begin{enumerate}
		\item Show that there exists $Z\in \Sigma^{\topp}(X_\mu(b)) $ whose stabilizer is  a parahoric subgroup of $J_b(F)$ of maximal volume.
		
		\item Show that all the stabilizers have the maximal volume.
	\end{enumerate}
	
	The Deligne--Lusztig reduction method in \cite{He14} works for the affine Deligne--Lusztig varieties in the affine flag variety. It keeps track of geometric information such as the dimension and the number of irreducible components of top dimension. To keep track of the stabilizers of top-dimensional irreducible components under the action of $J_b(F)$, we introduce a refined reduction method in the context of motivic counting. Then we use the explicit dimension formula for $X_\mu(b)$ and a certain affine Deligne--Lusztig variety $X_{w_0 t^\mu}(b)$ in the affine flag variety to obtain a $J_b(F)$-equivariant bijection $\Sigma^{\topp}(X_{w_0 t^\mu}(b)) \isoarrow \Sigma^{\topp} (X_{\mu}(b))$. We combine the explicit reduction path constructed in \cite{He14} with a refinement of the argument in \cite{HY} to obtain an element of $\Sigma^{\topp}(X_{w_0 t^\mu}(b))$ whose stabilizer in $J_b(F)$ has maximal volume. This finishes step (1).

	For step (2), consider the quantity 
	$$ Q(\mu,b) : = |J_b(F)\backslash \Sigma^{\topp}(X_\mu(b)) |^{-1} \cdot \sum _{ Z } \vol (\Stab_Z(J_b(F)))^{-1},$$ where the sum is over a set of representatives for the $J_b(F)$-orbits in $\Sigma^{\topp}(X_\mu(b))$. The results of \cite{ZZ} imply that $Q(\mu,b)$ depends only on $b$, not on $\mu$. Moreover, for the given $b $ there exists $\mu_1\in X_*(T)^+$ such that $|J_b(F)\backslash\Sigma^{\topp}(X_{\mu_1}(b)) |=1$. By step (1) applied to $(
	\mu_1, b)$, we know that  $Q(\mu_1, b)$ is equal to the inverse of the maximal volume attained by parahoric subgroups of $J_b(F)$.  Since $Q(\mu, b) = Q(\mu_1, b)$, and since $\Stab_Z(J_b(F))$ is a parahoric subgroup of $J_b(F)$ for each $Z \in \Sigma^{\topp}(X_\mu(b))$, we conclude that $\Stab_Z(J_b(F))$ must be a parahoric subgroup of maximal volume for each $Z$. This finishes step (2). 
	
	For step (2), the assumption that $F$ has characteristic $0$ is crucial. This is due to the fact that the results we use from \cite{ZZ} rely on the Base Change Fundamental Lemma, a result only known for characteristic $0$ local fields in general.
	
\subsection{Outline of the paper}In \S\ref{sec: ADLV} we introduce notations and some preliminary group theoretic results. In \S\ref{subsec very special}, we define very special parahoric subgroups and prove the equivalence of this condition with that of having maximal volume and maximal log volume. We then introduce affine Deligne--Lusztig varieties and establish the relation between components of $X_\mu(b)$ and $X_{w_0t^\mu}(b)$ in \S\ref{subsec:ADLV}. In \S\ref{sec: DL reduction an motivic counting}, we  give a reinterpretation of the Deligne--Lusztig reduction method in terms of motivic counting. We apply this in \S\ref{subsec: one comp} to show the existence of a component in $X_{w_0t^\mu}(b)$ whose stabilizer is a very special parahoric. In \S\ref{sec: main} , we prove Theorem A. In \S\ref{subsec: reduction to char 0} and \S\ref{subsec: reduction to basic}, we reduce the proof to the case where  $\mathrm{char}(F)=0$, $G$ is adjoint, unramified over $F$, and $F$-simple, and $b$ is basic. The proof then proceeds in \S \ref{subsec:numerical} and \S \ref{subsec:proof} as outlined above, with some extra work needed to handle the case of type $A$ and $E_6$, which is the content of \S \ref{subsec:special case}. Finally in \S\ref{sec: Shimura}, we apply our results to study the basic locus  of Shimura varieties and prove Theorem B. As mentioned, the key input is an analogue of the $p$-adic uniformization for the integral models of Shimura varieties constructed by Kisin--Pappas, which we prove following the method in \cite[\S7]{XZ} using results of \cite{Zhou}.
	
	\emph{Acknowledgments:}  We would like to thank S.~Nie, G.~Prasad, M.~Rapoport, and X.~Zhu for useful conversations regarding this project. This work was partially inspired by the comments of M.~Rapoport after the third-named author's talk at University of Maryland in Fall 2018. 
	
	X.H.~is partially supported by a start-up grant and by funds connected with
	Choh-Ming Chair at CUHK, and by Hong Kong RGC grant 14300220. R.Z.~was supported by NSF grant  DMS-1638352  through membership of the Institute for Advanced Study, and  by the European Research Council (ERC) under the European Union’s Horizon 2020 research and innovation programme (grant agreement No. 804176). Y.Z.~is partially supported by NSF grant DMS-1802292, and by a start-up grant at University of Maryland. 
	
	\section{Affine Deligne--Lusztig varieties}\label{sec: ADLV}

	\subsection{The Iwahori--Weyl group}\label{subsec:IW}
	\subsubsection{}\label{subsubsec:IW}Let $F$ be a non-archimedean local field with valuation ring $\CO_F$ and residue field $k_F = \BF_q$. We fix an algebraic closure $\overline{F}$ of $F$. Let $F^{\ur}$ be the maximal unramified extension of $F$ inside $\overline F$, and let $\breve F$ be the completion of $F^{\ur}$. We denote by $\CO_{\breve F}$ the valuation ring of $\breve F$, and denote by $\kk$ the residue field of $\breve F$, which is an algebraic closure of $k_F$. Fix an algebraic closure $\overline{\breve F}$ of $\breve F$, and fix an $F^{\ur}$-algebra embedding $\overline F \to \overline {\breve F}$. We write $\Gamma$ for $\Gal(\overline F/F)$ and write $\Gamma_0$ for the inertia subgroup of $\Gamma$, which is identified with $\Gal(\overline{\breve F}/\breve F)$. We let $\sigma\in\mathrm{Aut}(\breve F/F)$ denote the $q$-Frobenius.
	
	Let $G$ be a connected reductive group over $F$. We fix a maximal $ F^{\ur}$-split torus $S$ in $G$ defined over $F$, which exists by \cite[Corollaire 5.1.12]{BT2}. By \cite[Proposition 2.3.9]{rouss}, $S$ is also maximal $\breve F$-split. Let $T$ be the centralizer of $S$ in $G$. By Steinberg's theorem $G$ is quasi-split over $\breve F$,  so $T$ is a maximal torus in $G$. Let $N$ be the normalizer of $T$ in $G$, and let $$\breve W_0 : = N(\breve F)/T(\breve F). $$ In other words $\breve W_0$ is the relative Weyl group of $G_{\breve F}$.
	
	The \emph{Iwahori--Weyl group}
	is defined to be
	$$\breve W  :  = N(\breve F)/T(\breve F)_1$$  where $T(\breve F)_1$ is the kernel of the Kottwitz homomorphism $ T(\breve F) \to X_*(T)_{\Gamma_0}$. We have a natural short exact sequence 
	\begin{align}\label{eq:ses for IW}
	0 \to X_*(T)_{\Gamma_0} \to \breve W \to \breve W_0 \to 0.
	\end{align}
	For each $\lambda\in X_*(T)_{\Gamma_0}$, we write $t^\lambda$ for the corresponding element of $\breve W$. Such elements of $\breve W$ are called \emph{translation elements}.

	\subsubsection{} Let $\breve \CA$ be the apartment of $G_{\breve F}$ corresponding to $S_{\breve F}$. Thus $\breve \CA$ is an affine $\mathbb R$-space under $X_*(T)_{\Gamma_0} \otimes_{\mathbb Z} \mathbb R$. The Frobenius $\s$ and the Iwahori--Weyl group $\breve W$ act on $\breve \CA$ via affine transformations. Since $\breve \CA$ is naturally identified with the apartment of $G_{F^{\ur}}$ corresponding to $S_{F^{\ur}}$, there exists a $\s$-stable alcove in $\breve \CA$ by \cite[\S 1.10.3]{Tits} as the residue field of $F$ is finite. We fix such a $\s$-stable alcove $\breve {\mathfrak a}$. Let $\breve \CI \subset G(\breve F)$ be the Iwahori subgroup corresponding to $\breve\fka$.  Then $\breve \CI$ is $\sigma$-stable and we write $\CI$ for the corresponding Iwahori subgroup $\breve \CI ^{\s}$ of $G(F)$.
	
	As explained in \cite{HainesRapoport}, the choice of $\breve \fka$ gives rise to a subgroup $\breve W_a$ of $\breve W$ called the \emph{affine Weyl group}. This is by definition the subgroup generated by the set $\tS$ of simple reflections in the walls of $\breve \fka$. The pair $(\breve W_a,\tS)$ is a Coxeter group. \ignore{In fact, $\breve W_a$ can be identified with the Iwahori--Weyl group for the simply-connected cover of the derived group of $G$.}
	
	Let $\Omega$ be the stabilizer of $\breve \fka$ in $\breve W$. Then by \cite[Lemma 14]{HainesRapoport}, we have $$\breve W = \breve W_a \rtimes \Omega, $$ and $\Omega$ is (canonically) isomorphic to $\pi_1(G)_{\Gamma_0}$. The length function on the Coxeter group $(\breve W_a, \tS)$ extends to a function $$\breve \ell : \breve W \to \BZ_{\geq 0} $$ with respect to which $\Omega$ is the set of length-zero elements of $\breve W$. The Frobenius $\s$ naturally acts on $\breve W$, stabilizing the subset $\tS \subset \breve W$ (as $\breve \fka$ is $\s$-stable). In particular, $\s$ induces an automorphism of the Coxeter group $(\breve W_a, \tS)$. 
	
	By \cite[p.~195]{HainesRapoport}, there exists a reduced root system $\Sigma$ such that $$\breve W_a\cong Q^\vee(\Sigma)\rtimes W(\Sigma),$$
	where $Q^\vee(\Sigma)$ and $W(\Sigma)$ denote the coroot lattice and Weyl group of $\Sigma$ respectively. The roots of $\Sigma$ are proportional to the roots of the relative root system for $G_{\breve F}$. However the root systems themselves may not be isomorphic.
	
	\subsubsection{}\label{sssec: parabolic subgroup of Weyl group} Let ${\breve K}$ be a subset of $\tS$. We write $\breve W_{\breve K}\subset \breve W$ for the subgroup generated by ${\breve K}$. We let $\breve W^{\breve K}$ (resp.~$^{\breve K}\breve W$) denote the set of minimal length representatives for the cosets in $\breve W/\breve W_{\breve K}$ (resp.~$\breve W_{\breve K}\backslash \breve W$). 
	
	For each $w\in \breve W$, we choose a lift $\dot w \in N(\breve F)$ of $w$. We assume furthermore that $\s(\dot w)=\dot w$ if $\s(w)=w$. Indeed, to see that this can  always be arranged, it suffices to see that the Lang map $T(\breve F)_1 \to T(\breve F)_1, t \mapsto t \s(t)^{-1}$ is surjective. Now $T(\breve F)_1 = \mathcal T^0 (\CO_{\breve F})$ where $\mathcal T^0$ is the connected N\'eron model of $T$ over $\CO_F$, see \cite[Remark 2.2 (iii)]{RapGuide}. The desired surjectivity follows from Greenberg's theorem \cite[Proposition 3]{Greenberg} (whose proof holds regardless of the characteristic of $F$) applied to $\mathcal T^0$. 
	
	Let $\breve K$ be a subset of $\breve \BS$ such that $\breve W_{\breve K}$ is finite. In this case $\breve K$ corresponds to a standard parahoric subgroup of $G(\breve F)$ containing $\breve \CI$, which we denote by $\breve\CK$. By the Bruhat decomposition, the map $w\mapsto\dot{w}$ induces a bijection  $$\breve W_{\breve K}\backslash \breve W/\breve W_{\breve K}\isoarrow\breve \CK\backslash G(\breve F)/\breve \CK.$$ 
If furthermore ${\breve K}$ is $\s$-stable, then so is $\breve\CK$, and we write $\CK=\breve \CK^\s$ for the corresponding parahoric subgroup of $G(F)$. In what follows we will often abuse notation and write $\breve\CK$ (resp.~$\CK$) for the parahoric group scheme over $\CO_{\breve F}$ (resp. $\CO_F$) when there is no risk of confusion. The same is applied to the notations $\breve\CI$ and $\CI$.

	\subsubsection{}
	Let $A$ denote the maximal $F$-split subtorus of $S$, which is also a maximal $F$-split torus in $G$. We write $Z_A$ and $N_A$ for the centralizer and normalizer of $A$ in $G$ respectively. Since $Z_A$ is anisotropic modulo center over $F$, there is a unique parahoric subgroup $\CZ_A$ of $Z_A(F)$. The \emph{relative Iwahori--Weyl group} is defined to be $$W:=N_A(F)/\CZ_{A}.$$ It admits a natural map to the relative Weyl group $W_0:=N_A(F)/Z_A(F)$ of $G$ over $F$.

	We write $\mathscr{D}$ for the relative local Dynkin diagram of $(G,A,F)$, and write $\Delta$ for the set of vertices of $\mathscr D$. Let $\CA$ be the apartment associated to $A$, and let $\mathfrak a$ be the base alcove  in $\CA$ determined by the Iwahori subgroup $\CI$ of $G(F)$. For each $v\in \Delta$, let $\alpha_v$ be the corresponding non-divisible simple affine root on $\CA$. As explained in \cite[1.11]{Tits},  $\Delta$ is naturally identified with the set of $\sigma$-orbits $C$ in $\breve\BS$ such that $\breve W_C$ is finite. For $v\in \Delta$, we write $C_v\subset \breve \BS$ for the corresponding $\s$-orbit, and write $s_v \in W$ for the reflection in $\CA$ along $\alpha_v$. By \cite[Lemma 1.6]{Richarz}, there is a natural isomorphism $W \cong \breve W^\sigma$ induced by the inclusion map $N_A(F) \to N(\breve F)$. By \cite[A.8]{Lusztig03}, $s_v$ corresponds to the longest element of $\breve W_{C_v}$ under this isomorphism. We set $$\BS=\{s_v \mid v\in \Delta\}. $$ We also note that if $w \in W$, then the lifting $\dot w$ in $N(\breve F)$ chosen in \S \ref{sssec: parabolic subgroup of Weyl group} is contained in $N_A(F)$, which follows from our assumption that $\dot w$ is $\s$-invariant. 
	
	\subsection{Parahoric subgroups of maximal volume}\label{subsec very special}
	
	We keep the notations of \S\ref{subsec:IW}. In this subsection we give a description of the parahoric subgroups of $G(F)$ that have the maximal volume .
	
	\subsubsection{}\label{sssec: very special} 
	
	For a vertex $v\in \Delta$, we define $d(v) : =\breve\ell(s_v)$. When $G$ is simply connected and absolutely almost simple, this coincides with the integer attached to $v$ in \cite[1.8]{Tits}, cf.~\cite[Remark 1.13 (ii)]{Richarz}. We say that a special vertex $v\in \Delta$ is \emph{very special} if $d(v)$ is minimal among all special vertices $v'$ lying in the connected component of $\mathscr D$ containing $v$.
	
	Let $x\in \CA$ be a point lying in the closure $\overline{\fka}$ of $\fka$. We associate to $x$ a set of vertices $$\Delta_x:=\{v\in \Delta \mid s_v(x)\neq x\}.$$ It is easy to see that $\Delta_x$ has non-empty intersection with each connected component of $\mathscr D$. 
	
	\begin{definition}\label{defn: very special} A point $x$ lying in the closure $\overline{\fka}$  of $\fka$ is said to be \emph{very special} if $\Delta_x$ contains 	exactly one very  special vertex in each connected component of $\mathscr D$. A parahoric subgroup of $G(F)$ is said to be \emph{very special} if it is $G(F)$-conjugate to a standard parahoric subgroup associated to a very special $x\in \overline{\fka}$.\end{definition}
	
	\begin{remark}\label{rem: very special}
		When $G$ is simply connected and absolutely almost simple, our definition of a very special parahoric subgroup is the same as that in \cite[A.4]{Borel-Prasad}. There is also a notion of a very special parahoric subgroup defined in \cite[Definition 6.1]{Zhu}. When $G$ is quasi-split, it can be shown that these two notions are equivalent. However, they differ for non-quasi-split $G$ (cf.~\cite[Lemma 6.1]{Zhu}).
	\end{remark}
	
	\subsubsection{}\label{sssec: special fiber of parahoric}We now fix a choice of Haar measure on $G(F)$ such that all Iwahori subgroups of $G(F)$ have volume 1. Let $\CK$ be a parahoric subgroup of $G(F)$ and $\breve \CK$ the associated parahoric subgroup of $G(\breve F)$. We define the \emph{log-volume} of $\CK$ by
	\begin{equation}\label{eqn: def log vol}\log \vol(\CK) : =\dim\overline{\CK}/\overline{\CI},\end{equation}
	where $\overline{\CK}$ (resp.~$\overline{\CI}$) denotes the reductive quotient of the special fiber of $\breve \CK$ (resp.~the image of the special fiber of $\breve \CI$ in $\overline{\CK}$).
	If $\CK$ is a standard parahoric corresponding to a $\sigma$-stable subset $\breve K\subset \tS$, then we have
	\begin{equation}\label{eqn: log vol length}\log \vol(\CK) =\breve \ell(w_{\breve K }),\end{equation} where $w_{\breve K}$ is the longest element of $\breve W_{\breve K}$.
	
	We have the Bruhat decompositions $$\breve \CK=\coprod_{w \in \breve W_{\breve K}} \breve \CI \dot w \breve \CI$$ and $$\CK=\coprod_{w\in \breve W_{\breve K}^\sigma} \CI \dot w \CI. $$ By \cite[Proposition 1.11]{Richarz}, we have \begin{equation}\label{eqn: parahoric volume Weyl group}\vol(\CK)=\sum_{w\in \breve W_{\breve K}^\sigma}q^{\breve\ell(w)}. \end{equation}

	\begin{proposition}\label{prop: eq. v. special, max vol, max log vol}
		Let $\CK$ be a parahoric subgroup of $G(F)$. Then the following are equivalent:
		
		\begin{enumerate}
			\item $\CK $ is a very special parahoric;
			
			\item $\CK$ is of maximal volume among all the parahoric subgroups of $G(F)$;

			\item $\CK$ has maximal log-volume.
		\end{enumerate}
	\end{proposition}
	
	\begin{remark}
		When $G$ is simply connected and absolutely almost simple, the equivalence between (1) and (2) is \cite[Proposition A.5]{Borel-Prasad}. The equivalence between (3) and the other two conditions will be used in the proof of Corollary \ref{cor:red1} below, especially when we alter the local field. \end{remark}
	
	\subsubsection{}  To prove Proposition \ref{prop: eq. v. special, max vol, max log vol} we follow the method in \cite[A.4]{Borel-Prasad}. We begin with some  preparation.
	Assume that $G$ is almost simple over $F$ and let $\Phi$ be the relative root system $\Phi(G,A)$. We let $\Phi^{\mathrm{nd}}$ denote the system of non-divisible roots in $\Phi$ and we write $\BW$ for the Weyl group of $\Phi^{\mathrm{nd}}$, which is identified with the relative Weyl group $W_0$  of $G$.
	
	For an element $v\in \Delta$, we define $K(v):= \BS\setminus\{s_v\}\subset  \BS$. We let $W_{K(v)}$ denote the subgroup of $W\cong \breve W^\sigma$ generated by $K(v)$. Then the natural map $\mathrm{Aff}(\CA) \to \GL(X_*(A)\otimes \BR)$ (i.e., taking the linear part) induces an identification between $W_{K(v)}$ and a subgroup of $\BW$, which we denote by $\BW_v$. We denote the inverse isomorphism by $\iota_v: \BW_v \isoarrow W_{K(v)}$. For $w\in \BW_v$, we set $$d(w,v): =  \breve\ell(\iota_v(w)), $$ where we consider $W_{K(v)}$ as a subgroup of $\breve W$. For each $v'\in \Delta \setminus \{ v \}$, we write $\overline \alpha_{v'}$ for the unique proportion of the vector part of $\alpha_{v'}$ that lies in $\Phi^{\mathrm{nd}}$. We let $\Phi_v$ denote the sub-root system of $\Phi^{\mathrm{nd}}$ generated by $\overline \alpha_{v'}$ with $v' \in   \Delta  \setminus \{ v \}$. 
	
	We define an ordering on $\Phi_v$ by specifying the positive simple  roots to be given by $\overline{\alpha}_{v'}$ with $v'\in\Delta\setminus \{v\} $, and we write $\Phi_v^+$  (resp.~$\Phi_{v}^-$) for the subset of positive (resp.~negative) roots. Note that the ordering on $\Phi_v$ depends on $v$; it is possible that there exist $v_1, v_2\in \Delta$ such that $\Phi_{v_1}=\Phi_{v_2}$ but $\Phi_{v_1}^+\neq \Phi_{v_2}^+$.
	
	For $\overline{\alpha}\in \Phi_v$, we define an integer $d(\overline{\alpha},v)$ as follows. If $\overline{\alpha}=\overline\alpha_{v'}$ for some $v'\in\Delta\setminus \{v\}$, then we define $d(\overline{\alpha},v)=d(v')$. In general, we define $d(\overline{\alpha},v)$ by specifying that its dependence on $\overline \alpha$ is $\BW_v$-invariant.  This is well-defined since if  $v_1,v_2,\in \Delta\setminus\{v\}$ are such that $\overline{\alpha}_{v_1}$ and $\overline{\alpha}_{v_2}$ are $\BW_v$-conjugate, then $d(v_1)=d(v_2)$; cf. \cite[A.4]{Borel-Prasad}.

	\begin{lemma}\label{lem: integer d and roots}For each $w\in \BW_v$, we have
		\begin{equation}
		d(w,v)=\sum_{\overline{\alpha}\in \Phi_{v}^+,w\overline{\alpha}\in \Phi_v^-}{d(\overline{\alpha},v)}.
		\end{equation}
	\end{lemma}
	\begin{proof}
		Let $s_1 \cdots s_n$ be a reduced word decomposition for $w\in \BW_v$, where $s_{i}$ is the simple reflection corresponding to $\overline{\alpha}_{v_i}$ for $v_i\in \Delta\setminus \{v\}$. For $i=1,\ldots,n$, set $w_i=s_{i+1} \cdots s_n$. Then the association $s_i\mapsto w^{-1}_i\overline{\alpha}_{v_i}$ defines a bijection $$\{s_1,\dotsc,s_n\}\xrightarrow{\sim} \{\overline{\alpha}\in \Phi_v^+\mid w\overline{\alpha}\in\Phi_v^-\}.$$
		By $\BW_v$-invariance, we have for $i=1,\dotsc,n$,  $$d(w_i^{-1}\overline{\alpha}_{v_i},v)=d(\overline{\alpha}_{v_i},v)=\breve\ell(s_{v_i}).$$ By \cite[Sublemma 1.12]{Richarz} and induction, we have $$d(w,v)=\sum_{i=1}^n\breve \ell(s_{v_i})$$ and the result follows.
	\end{proof}
	\begin{lemma}\label{lem: comparision of d for different vertices}
		Assume $G$ is almost simple over $F$ and let $v,v_0\in \Delta$ with $v_0$ a very special vertex. Then for all $\overline{\alpha}\in \Phi_{v}$, we have $d(\overline{ \alpha},v)\leq d(\overline{\alpha},v_0)$.
	\end{lemma}
	
	\begin{proof}Since $d(\overline{\alpha},v)$ and $d(\overline{\alpha},v_0)$ only depend on the $\BW_v$-orbit of $\overline{\alpha}$, it suffices to prove this in the case that $\overline{\alpha}\in \Phi_{v}^+$ is a simple root, i.e. $\overline\alpha=\overline\alpha_{v'}$ with $v'\in \Delta\setminus \{v\}$. If $v'\neq v_0$, then $\overline{\alpha}$ is also a simple root for $\Phi_{v_0}^+$ and we have $d(\overline{\alpha},v)=d(\overline{\alpha},v_0)=d(v')$.
		
		If $v'=v_0$, by inspection of Tits' table \cite[\S 4]{Tits}, we find that $$d(\overline\alpha,v):=d(v')=\min_{v''\in \Delta}d(v'')$$
		unless $G$ is of type ${^2A}_{2m-1}''$, ${^2D}_n'$, ${^2D}_{2m}''$, ${^4D}_{2m+1}$ or  ${^3E}_6$. In these cases, one computes explicitly that $d(\overline{\alpha},v)\leq d(\overline{\alpha},v_0)$.
	\end{proof}

	\begin{proof}[Proof of Proposition \ref{prop: eq. v. special, max vol, max log vol}]
		It suffices to prove the result for $\CK$ a standard parahoric.
		We first consider the case where $G$ is adjoint and simple over $F$. Let $\breve K_0,\breve K\subset \breve\BS$ be $\sigma$-stable subsets  with corresponding parahoric subgroups $\CK_0$ and $\CK$ of $G(F)$, and corresponding subsets $K_0,K_0\subset \BS$. Assume that $\CK_0$ is a very special parahoric. Then we need to show that
		\begin{align*}
\vol(\CK) & \leq \vol(\CK_0) , \\ 
\log\vol(\CK)& \leq\log\vol(\CK_0)
		\end{align*}
		and that strict inequality holds in each case if $\CK$ is not very special.
		
		Since $\CK_0$ is very special, we have $K_0=K(v_0)$ for $v_0\in \Delta$  a very special vertex.  Moreover, since $\CK$ is contained inside a parahoric corresponding to some $v\in \Delta$, we may assume $ K=K(v)$.
		
		Since $v_0$ is a very special vertex, $\BW_{v_0}=\BW$ and we have $\Phi^{\mathrm{nd}}=\Phi_{v_0}$. Let $u\in \BW_v$ be the unique element such that $u(\Phi_v^+)\subset \Phi_{v_0}^+$. Then $u(\Phi_{v}^-)\subset \Phi_{v_0}^-$. By Lemma \ref{lem: integer d and roots}, for $w\in \BW_v$ we have \begin{align*}d(w,v)&= \sum_{\overline{\alpha}\in u( \Phi_v^+), uwu^{-1}\overline{\alpha}\in \Phi_{v_0}^-}d(\overline{\alpha},v)\leq \sum_{\overline{\alpha}\in u( \Phi_v^+), uwu^{-1}\overline{\alpha}\in \Phi_{v_0}^-}d(\overline{\alpha},v_0)\\& \leq \sum_{\overline{\alpha}\in \Phi_{v_0}^+,uwu^{-1}\overline{\alpha}\in \Phi_{v_0}^-}d(\overline{\alpha},v_0)= d(uwu^{-1},v_0),
		\end{align*} where the  first inequality follows from Lemma \ref{lem: comparision of d for different vertices}. Thus $$\vol (\CK)=\sum_{w\in \BW_v}q^{d(w,v)}\leq \sum_{w\in \BW_v}q^{d(w,v_0)}\leq\sum_{w\in \BW}q^{d(w,v_0)}=\vol(\CK_0).$$
		If $\CK$ is not special, then the second inequality is strict. If $\CK$ is special but not very special, then the first inequality is strict. We thus obtain the equivalence $(1)\Leftrightarrow (2)$.
		
		Similarly, if we let  $w_{v}\in \BW_v$ (resp. $w_{v_0}\in \BW_{v_0}$) denote the image of $w_{\breve K}$ (resp. $w_{\breve K_0}$), then we have
		\begin{align*}\log\vol (\breve W_{\breve K}) &=d(w_{v},v)= \sum_{\overline{\alpha}\in u(\Phi_{v}^+)}d(\overline{\alpha},v)\leq \sum_{\overline{\alpha}\in \Phi_{v_0}^+}d(\overline{\alpha},v_0)\\ &
		= d(w_{v_0},v_0)=\log\vol (\breve W_{\breve K_0}).
		\end{align*} 
		If $\CK$ is not very special, then the inequality is strict. Thus we obtain $(1)\Leftrightarrow (3)$.
		
	The case with general $G$ is reduced to the above special case by considering the direct product decomposition of $G_{\mathrm{ad}}$ into $F$-simple factors. In fact, by (\ref{eqn: def log vol}) (resp.~ (\ref{eqn: parahoric volume Weyl group})), we know that the  log-volume (resp.~volume) of a parahoric subgroup of $G(F)$ is equal to the product of the log-volumes (resp.~volumes) of corresponding parahoric subgroups of the $F$-simple factors of $G_{\mathrm{ad}}$.   
		\ignore{		
				Now we consider the case of general $G$.	Let $\breve W_{\ad}$ denote the Iwahori--Weyl group for the adjoint group $G_{\ad}$. Then $\breve \BS$ determines a set of simple reflection $\breve\BS_{\ad}$ for $\breve  W_{\ad}$. We assume $\CK$ corresponds to a $\sigma$-stable subset $\breve K\subset \tS$. We let $\breve K_{\ad}\subset \breve \BS_{\ad}$ denote the image of $\breve K$ and we let $\CK_{\ad}$ denote the corresponding parahoric of $G_{\ad}(F)$.  Let $G_{\ad}\cong \prod_{i=1}^r G_i$ be the decomposition of $G_{\ad}$ into $F$-simple factors. Correspondingly we obtain decompositions $$\tS_{\ad}\cong\coprod_{i=1}^r\tS_i,\ \ \breve K_{\ad}\cong \coprod_{i=1}^r\breve K_i,\ \ \CK_{\ad}\cong \prod_{i=1}^r\CK_i.$$
		For $i=1,\dotsc,r$ we fix a Haar measure $\mu_{G_i}$ on $G_i(F)$ such that the volume of any Iwahori is 1, and we let $\mu_{G_{ad}}$ denote the product measure on $G_{\ad}$. Then by equation (\ref{eqn: parahoric volume Weyl group}), we have $$\vol(\CK)=\vol(\CK_{\ad})=\prod_{i=1}^r\vol(\CK_i).$$
		
		Similarly, by equation (\ref{eqn: def log vol}) we have $$\log\vol(\CK)=\log\vol(\CK_{\ad})=\prod_{i=i}^r\log\vol(\CK_i).$$
		
		Moreover, by definition we have that $\CK$ is very special if and only if every $\CK_i$ is a very special parahoric of $G_i(F)$. Thus the Proposition follows from the case when $G$ is adjoint and $F$-simple.}
	\end{proof}

	\subsection{$\s$-conjugacy classes}\label{subsec:B(G)}
	We keep the setting of \S \ref{subsec:IW}, and assume in addition that $G$ is quasi-split over $F$. 
	\subsubsection{}\label{subsubsec:(B,T)} Under the assumption that $G$ is quasi-split over $F$, we can fix a $\s$-stable special point $\breve\fks$ lying in the closure of $\breve \fka$ (cf.~\cite[Lemma 6.1]{Zhuram}). For an abelian group $X$ and a $\BZ$-algebra $R$, we write $X_R$ for $X \otimes _{\BZ} R$. The choice of $\breve \fks$ gives rise to a $\s$-equivariant isomorphism \begin{align}\label{eqn: id apts} X_*(T)_{\Gamma_0, \BR} \cong \breve \CA,
	\end{align} which sends $0$ to $\breve \fks$. We let $\tS_0\subset\tS$ denote the subset of simple reflections fixing $\breve \fks$.  Then $\tS_0$ is preserved by the action of $\s$. The identification (\ref{eqn: id apts}) determines a chamber ${X_*(T)_{\Gamma_0,\BR}}^+$ in $X_*(T)_{\Gamma_0,\BR} \cong X_*(S)_{\BR}$ (with respect to the relative roots of $(G_{\breve F},S_{\breve F})$), namely the one whose image under (\ref{eqn: id apts}) contains the alcove $\breve \fka$. We let ${X_*(T)_{\Gamma_0}}^+$ (resp. ${X_*(T)_{\Gamma_0,\BQ}}^+$) denote the preimage of ${X_*(T)_{\Gamma_0,\BR}}^+$ under the map $X_*(T)_{\Gamma_0} \to X_*(T)_{\Gamma_0, \BR}$ (resp. $X_*(T)_{\Gamma_0,\BQ} \to X_*(T)_{\Gamma_0, \BR}$).  
	
	Note that ${X_*(T)_{\Gamma_0,\BR}}^+$ gives rise to an ordering of the relative roots of $(G_{\breve F}, S_{\breve F})$. Since $G$ is quasi-split over $\breve F$, this uniquely determines an ordering of the absolute roots in $X^*(T)$, and determines a Borel subgroup of $G_{\breve F}$ containing $T_{\breve F}$. Since $C$ is $\s$-stable, this Borel subgroup comes from a Borel subgroup $B$ of $G$ containing $T$. 
	
	\subsubsection{}\label{subsubsec:two maps} For $b\in G(\breve F)$, we let $[b]$ denote the $\s$-conjugacy class of $b$, namely 
	$$ [ b]   = \{h^{-1}b\sigma(h) \mid h\in G(\breve F)\}. $$ We shall sometimes write $[b]_G$ if we want to specify $G$. Let $B(G)$ be the set of $\s$-conjugacy classes in $G(\breve F)$.

	The elements of $B(G)$ have been classified by Kottwitz in \cite{kottwitzisocrystal2}.  For $b\in G(\breve F)$, we write $\overline{\nu}_b\in ({X_*(T)_{\Gamma_0,\BQ}}^+)^\sigma$ for its dominant Newton point. (Note that $({X_*(T)_{\Gamma_0,\BQ}}^+)^\sigma$ is canonically identified with $({X_*(T)_{\BQ}}^+)^{\Gamma}$, where ${X_*(T)_{\BQ}}^+$ consists of the $B$-dominant elements of $X_*(T)_{\BQ}$.) The map $b\mapsto \overline \nu _b$ induces a map $\overline{\nu}:B(G)\rightarrow ({X_*(T)_{\Gamma_0,\BQ}}^+)^\sigma$.
	
	We let $\tilde{\kappa}:G(\breve F)\to \pi_1(G)_{\Gamma_0}$ denote the Kottwitz homomorphism and we write $$\kappa:G(\breve F)\to  \pi_1(G)_{\Gamma}$$ for the composition of $\tilde{\kappa}$ with the natural projection $\pi_1(G)_{\Gamma_0}\rightarrow \pi_1(G)_{\Gamma}$. This factors through a map $B(G) \to \pi_1(G)_{\Gamma}$, which we still denote by $\kappa$. 
	
	By \cite[\S 4.13]{kottwitzisocrystal2}, the map $$(\overline{\nu},\kappa):B(  G) \to    ({X_*(T)_{\Gamma_0,\BQ}}^+)^\sigma \times \pi_1(G)_\Gamma$$ is injective. We sometimes write $\overline{\nu}^G$ and $\kappa_{G}$ for $\overline {\nu}$ and $\kappa$ if we want to specify $G$. 
	
	An element $b\in G(\breve F)$ is said to be \emph{basic} if $\overline \nu_b$ is central. Similarly we define \emph{basic} elements of $B(G)$. 
	
	\subsubsection{} \label{subsubsec:J_b} For $b \in G(\breve F)$, let  $J_b$ denote the $\s$-centralizer group of $b$. It is a reductive group over $F$ such that  $$J_b(R)=\{g \in G(\breve F\otimes_F R)\mid g \i b \s(g)=b\}$$ for any $F$-algebra $R$. Let $M$ be the centralizer of $\overline{\nu}_b$, where we consider $\overline{\nu}_b$ as an element of $({X_*(T)_{\BQ}}^+)^{\Gamma} \subset (X_*(T)^{\Gamma})_{\BQ}$ as explained in \S \ref{subsubsec:two maps}. Then $M$ is a Levi subgroup of $G$ defined over $F$ and $J_b$ is an inner form of $M$ over $F$. 
	\subsubsection{}The maps $\overline \nu$ and $\kappa$ on $B(G)$ can be described in a more explicit way as follows. Let $B(\breve W, \s)$ be the set of $\s$-conjugacy classes in $\breve W$. The map $\breve W \to G(\breve F)$, $w \mapsto \dot w$ defined in \S \ref{subsec:IW} induces a well-defined map
	\begin{align}\label{eq:Psi}
 B(\breve W, \s)  \to B(G).
	\end{align}
	
	For each $w\in \breve W$, there exists a positive  integer $n$ such that $\s^n$ acts trivially on $\breve W$ and such that $w\sigma(w)\cdots\sigma^{n-1}(w)=t^\lambda$ for some $\lambda\in X_*(T)_{\Gamma_0}$. We set $\nu_w:=\frac{\lambda}{n}\in X_*(T)_{\Gamma_0,\BQ}$ and we let $\overline{ \nu}_w$ denote the unique $\breve W_0$-conjugate of $\nu_w$ that lies in ${X_*(T)_{\Gamma_0,\mathbb{Q}}}^+$. Then $\overline{ \nu}_w$  is necessarily fixed by $\s$. We let $\tilde\kappa(w)\in \pi_1(G)_{\Gamma_0}$ denote the image of $w$ under the quotient map $\breve W\rightarrow \breve W/\breve W_a\cong \pi_1(G)_{\Gamma_0}$, and we let $\kappa(w)$ be the image of $\tilde \kappa(w)$ in $\pi_1(G)_\Gamma$. By \cite{He14}, we have a commutative diagram:
	
	\[\xymatrix{B( \breve W,\s)\ar[rr]^{(\ref{eq:Psi})} \ar[rd]_{(\overline\nu,\kappa)}&& B(G)\ar[ld]^{(\overline\nu,\kappa)}\\
		& ({X_*(T)_{\Gamma_0,\mathbb{Q}}}^+)^\sigma\times \pi_1(G)_{\Gamma}&}.\]

	\subsection{Affine Deligne--Lusztig varieties}\label{subsec:ADLV}
	We keep the setting and notation of \S\ref{subsec:B(G)}. We assume in addition that $G$ splits over a tamely ramified extension of $F$ and that $\mathrm{char} ( F)$ is either zero or coprime to the order of $\pi_1(G_{\mathrm{ad}})$. 
	\subsubsection{}\label{sssec: ADLV}	
	Let ${\breve K}\subset \breve \BS$ be a $\s$-stable subset that corresponds to a parahoric subgroup $\breve \CK \subset G(\breve F)$. For $w \in \breve W_{\breve K}\backslash \breve W/\breve W_{\breve K}$ and $b \in G(\breve F)$, we set
	\begin{align}\label{eq:ADLV}
	X_{{\breve K},w}(b)(\kk)=\{g \breve \CK \in G(\breve F)/\breve \CK\mid g \i b \s(g) \in \breve \CK \dot{w} \breve \CK\}.
	\end{align}
	If $\mathrm{char}(F)>0$, then $X_{{\breve K}, w}(b)(\kk)$ is the set of $\kk$-points of a locally closed sub-scheme $X_{{\breve K},w}(b)$ of the partial affine flag variety $\mathrm{Gr}_{\breve \CK}$. In this case $X_{{\breve K},w}(b)$ is locally of finite type over $\kk$ (cf.~\cite{PR}). If $\mathrm{char}(F)=0$, then $X_{{\breve K}, w}(b)(\kk)$ is the set of $\kk$-points of a locally closed sub-scheme $X_{{\breve K}, w}(b)$ of the Witt vector partial affine flag variety  $\mathrm{Gr}_{\breve\CK}$ constructed by X.~Zhu \cite{Zhu} and Bhatt--Scholze \cite{BS}. In this case $X_{{\breve K}, w}(b)$ is locally of perfectly finite type over $\kk$ (see \cite[Theorem 1.1]{HVfinite}). 
	In both cases, we call $X_{\breve K,w}(b)$ the \emph{affine Deligne--Lusztig variety} associated to $b$, $w$, and $\breve K$. 
	
	The group $J_b(F)$ (see \S \ref{subsubsec:J_b}) acts on $X_{\breve K, w}(b)$ via $\kk$-scheme automorphisms.
	By \cite[Theorem 1.1]{HVfinite}, the induced $J_b(F)$-action on the set of irreducible components of $X_{{\breve K},w}(b)$ has finitely many orbits. The results in \cite{HVfinite} also have the following easy consequence. 
	
	\begin{lemma}\label{lem:quasi-compact}
		Every irreducible component of $X_{\breve K, w}(b)$ is quasi-compact. 
	\end{lemma} 
\begin{proof} Let $Z$ be an irreducible component of $X_{\breve K, w}(b)$. By \cite[Proposition 5.4]{HVfinite}, there is a dense open subset $U \subset Z$ which is contained in a finite union $\bigcup_i S_i$ of Schubert varieties in $\mathrm{Gr}_{\breve \CK}$. Since the Schubert varieties are closed in $\mathrm{Gr}_{\breve \CK}$, we have $Z \subset \bigcup_i S_i$. Moreover, since $Z$ is closed in $X_{\breve K, w}(b)$, it is locally closed in $\bigcup_i S_i$. Now the Schubert varieties are of finite type over $\kk$ when $\mathrm{char}(F) >0$ and of perfectly finite type over $\kk$ when $\mathrm{char}(F) = 0$ (cf.~\cite[\S 4]{HVfinite}), so the underlying topological space of $\bigcup_i S_i$ is Noetherian. It follows that $Z$ is quasi-compact. 
\end{proof}
 
	\subsubsection{}\label{subsubsec:cases we consider}
	We are mainly interested in $X_{\breve K, w}(b)$ in the following two cases: 
	
	\begin{itemize}
		\item (Iwahori level.) We have ${\breve K} = \emptyset$, i.e., $\breve \CK=\breve \CI$.
		\item (Maximal special level.) We have ${\breve K}= \tS_0$, i.e., $\breve \CK$ is the maximal special parahoric subgroup corresponding to the special point $\breve \fks$. 
	\end{itemize} 
	
	When ${\breve K}=\emptyset$, we simply write $X_w(b)$ for $X_{\emptyset,w}(b)$. When ${\breve K}=\tS_0$, the restriction of the natural map $\breve W \to \breve W_0$ to $\breve W_{\breve K} \subset \breve W$ induces an isomorphism $\breve W_{{\breve K}} \isoarrow \breve W_0$. In other words, our choice of $\breve \fks$ determines a splitting of the exact sequence (\ref{eq:ses for IW}). In this case we shall identify $\breve W_0$ with $\breve W_{\breve K}$, viewed as a subgroup of $\breve W$. We have natural bijections  $$ {X_*(T)_{\Gamma_0}}^+ \cong X_*(T)_{\Gamma_0}/ \breve W_0 \cong    \breve W_{0}\backslash \breve W/\breve W_0 ,$$ where the second map is induced by the inclusion $X_*(T)_{\Gamma_0} \hookrightarrow \breve W, \mu \mapsto t^{\mu}$ (see (\ref{eq:ses for IW})). For $\mu\in {X_*(T)_{\Gamma_0}}^+$, we write $X_\mu(b)$ for $X_{\tS_0,t^\mu}(b)$. We sometimes write $X_\mu^G(b)$ for $X_\mu(b)$ if we need to specify the group $G$. If $G$ is unramified over $F$, then every cocharacter $\mu'$ of $G_{\overline F}$ is conjugate to a unique 
	element $\mu\in {X_*(T)_{\Gamma_0}}^+ = \ {X_*(T)_{\Gamma_0}}$. In this case we also write $X_{\mu'}(b)$ for $X_\mu (b)$. 
	
	\subsubsection{}\label{sssec: Kottwitz set}For $\lambda,\lambda'\in X_*(T)_{\Gamma_0,\BQ}\cong X_*(S)_{\BQ}$, we write $\lambda\leq\lambda'$ if $\lambda'-\lambda$ is a non-negative rational linear combination of the positive coroots in $X_*(S)$ (with respect to $(G_{\breve F}, S_{\breve F})$ and the ordering defined in \S \ref{subsubsec:(B,T)}).
	
	For $\mu\in {X_*(T)_{\Gamma_0}}^+$,  we define $$B(G,\mu):=\{[b]\in B(G)\mid\overline{\nu}_b\leq \mu^{\diamond}, \kappa(b)=\mu^\natural\}.$$
	Here $\mu^\natural$ is the image of $\mu$ in $\pi_1(G)_\Gamma$, and $\mu^{\diamond} \in {X_*(T)_{\Gamma_0, \BQ}}^+$ denotes the average of the $\sigma$-orbit of the image of $\mu$ in ${X_*(T)_{\Gamma_0, \BQ}}^+$. The set $B(G, \mu)$ has a unique basic element, which is also the unique minimal element with respect to the natural partial order on $B(G, \mu)$ (see \cite[\S 2]{HR}). 
	
	The following criterion for the non-emptiness of $X_\mu(b)$, originally conjectured by Kottwitz and Rapoport, was proved by Gashi \cite{Ga} for unramified groups and by the first-named author \cite[Theorem 7.1]{He14} in general.
	
	\begin{theorem}\label{thm:non-empty} For $\mu \in {X_*(T)_{\Gamma_0}}^{+}$, we have 
		$X_\mu(b)\neq\emptyset$ if and only if $[b]\in B(G,\mu)$. \qed 
	\end{theorem}

\subsubsection{} Now we recall the dimension formula for  $X_{\mu}(b)$. For $b\in G(\breve F)$, the \textit{defect} of $b$ is defined as $$\mathrm{def}_G(b):=\rank_F G-\rank_F J_b.$$
We let $\rho$ denote the half sum of positive roots in the root system $\Sigma$ (see \S \ref{subsec:IW}). The following theorem was proved by G\"ortz--Haines--Kottwitz--Reumann \cite{GHKR} and \cite{viehmanndim} for split $G$, and  by  Hamacher \cite{Ham} and X.~Zhu \cite{Zhu} independently for unramified groups. The result in general was proved by the first-named author \cite[Theorem 2.29]{He-CDM}.

\begin{theorem}\label{thm:dim-adlv}
	Assume $[b]\in B(G,\mu)$. Then we have $$\dim X_\mu(b)=\langle\mu-\overline{\nu}_ b,\rho\rangle-\frac{1}{2}\mathrm{def}_G(b).$$ \qed 
\end{theorem}

	\begin{definition}\label{defn:Sigma}
	For a scheme $X$ of finite Krull dimension and each non-negative integer $d$, we write $\Sigma^d(X)$ for the set of irreducible components of $X$ of dimension $d$ (which is allowed to be empty). We write $\Sigma^{\topp}(X)$ for $\Sigma^{\dim(X)}(X)$. We write $\Sigma(X)$ for the set of all irreducible components of $X$. 
	\end{definition}
	\subsubsection{} 
 The main object of interest in this paper is the set $\Sigma^{\topp}(X_{\mu}(b)) $. To study this set it will be useful to relate $X_{\mu}(b)$ to a certain affine Deligne--Lusztig variety with Iwahori level.
	
	We have a natural projection map $$\pi:\mathrm{Gr}_{\breve \CI}\to \mathrm{Gr}_{\breve \CK}$$ between the partial affine flag varieties, which exhibits $\mathrm{Gr}_{\breve \CI}$ as an \'etale fibration over $\mathrm{Gr}_{\breve \CK}$ with fibers isomorphic to  $\overline{\CK}/\overline{\CI}$ when $\mathrm{char}(F)>0$ (resp. ~the perfection of $\overline{\CK}/\overline{\CI}$ when $\mathrm{char}(F)=0$). See \S  \ref{sssec: special fiber of parahoric} for $\overline{\CK}/\overline{\CI}$. 
	
	As in \S \ref{subsubsec:cases we consider}, we identify $\breve W_0$ with the subgroup $\breve W_{\tS_0}$ of $\breve W$. For $\mu\in {X_*(T)_{\Gamma_0}}^+$, the map $\pi$ induces a $J_b(F)$-equivariant map \begin{equation}\label{eqn: fibration ADLV}
	\bigcup_{w\in \breve W_0t^\mu \breve W_0}X_w(b)\to X_\mu(b).
	\end{equation}
	In fact, the left hand side is equal to $\pi^{-1}(X_\mu(b))$.

	\begin{proposition}\label{prop: relation between Iw and Sp ADLV} Let $w_0$ denote the longest element of $\breve W_0$. The map $X_{w_0 t^\mu}(b) \to  X_\mu(b)$ induces a $J_b(F)$-equivariant bijection 
		$$\Sigma^{\topp}(X_{w_0 t^\mu}(b)) \isoarrow \Sigma^{\topp}(X_\mu(b)).$$
	\end{proposition}
	\begin{proof} Write $Y$ for the left hand side of (\ref{eqn: fibration ADLV}). Since the map (\ref{eqn: fibration ADLV}) is a fibration, it induces a $J_b(F)$-equivariant bijection $$\Sigma^{\topp}(Y)\isoarrow \Sigma^{\topp}(X_\mu(b)).$$
		Note that the $J_b(F)$-action on $Y$ stabilizes $X_w(b)$ for each $w \in \breve W_0 t^\mu \breve W_0$. Moreover, each $X_w(b)$ is locally closed in $Y$. By \cite[Theorem 9.1]{He14}, for $w\in \breve W_0t^\mu \breve W_0$, we have $$\dim X_w(b)\leq \dim X_{w_0t^\mu} (b)$$ with equality if and only if $w=w_0t^\mu$. Thus the inclusion map $X_{w_0 t^\mu}(b) \hookrightarrow Y$ induces a $J_b(F)$-equivariant bijection $$\Sigma^{\topp}(X_{w_0 t^\mu}(b)) \xrightarrow{\sim} \Sigma^{\topp}(Y).$$ 
		The statement is proved.  
	\end{proof}
	
	\section{Deligne--Lusztig reduction method and motivic counting}\label{sec: DL reduction an motivic counting}

	\subsection{The Grothendieck--Deligne--Lusztig monoid}\label{subsec:Groth monoid}
	Recall that $\kk$ is a fixed algebraic closure of $\BF_q$. Let $H$ be an abstract group. We retain the notations introduced in Definition \ref{defn:Sigma}. 	
	\begin{definition}\label{defn:S^H}
		Let $\CS^H$ be the category of perfect $\kk$-schemes $V$ that are equipped with an $H$-action and satisfy the following conditions: 
		\begin{enumerate}
			\item The scheme $V$ is locally of  perfectly finite type over $\kk$. 
			\item Each irreducible component of $V$ is quasi-compact. 
			\item The $H$-action on $\Sigma(V)$ has finitely many orbits.
		\end{enumerate}
	We define morphisms in $\CS^H$ to be the $\kk$-scheme morphisms that are $H$-equivariant.
	\end{definition} 
	\subsubsection{}
	It is a simple exercise to check that the category $\CS^H$ is essentially small. Thus the isomorphism classes in $\CS^H$ form a set. Let $\BN[\CS^H]$ be the free commutative monoid generated by this set. For any object $V$ in $\CS^H$, we denote by $[V]$ the  element of $\BN[\CS^H]$ given by the isomorphism class of $V$. 
	
	For any $\kk$-scheme $Q$, we write $Q^{\pfn}$ for the perfection of $Q$, which is a perfect $\kk$-scheme. 
	We write $\BA^1$ for $\BA^1_{\kk}$, and write $\BG_m$ for $\BA^1_{\kk} - \{ 0 \}$. Then $\BG_m^{\pfn}$ equipped with the trivial $H$-action is an object in $\CS^H$. Moreover, if $V$ is in $\CS^H$, then $V \times_{\kk}  \BG_m^{\pfn} $ equipped with the product $H$-action is also in $\CS^H$. 
	We thus define an endomorphism $\BT$ of $\BN[\CS^H]$ by $$[V] \longmapsto [V \times_{\kk} \BG_m^{\pfn}], \quad  \text{for any object } V \text{ in } \CS^H.$$ 

	\begin{lemma}\label{lemma: complement in S^H}
		Let $V$ be an object in $\CS^H$, and let $U$ be an $H$-stable open subscheme of $V$. Then $U$ equipped with the induced $H$-action is an object in $\CS^H$. 
	\end{lemma}
	
	\begin{proof}Clearly $U$ satisfies condition (1) in Definition \ref{defn:S^H}. We verify the other two conditions. For each $Z \in \Sigma(U)$, the closure $\overline Z$ of $Z$ in $V$ is an element of $\Sigma(V)$. Conversely, for each $Z' \in \Sigma(V)$, either $Z' \cap U = \emptyset$, or $Z' \cap U$ is an element of $\Sigma(U)$. Hence we have a bijection 
		$$ \Sigma(U) \isoarrow \{ Z' \in \Sigma(V) \mid Z' \cap U \neq \emptyset  \}, \quad Z \mapsto \overline Z. $$ The right hand side is an $H$-stable subset of $\Sigma(V)$, and the bijection is $H$-equivariant. Since  $V$ satisfies condition (3), so does $U$. 
		
	Since $V$ satisfies conditions (1) and (2), each $Z' \in \Sigma(V)$ is Noetherian as a topological space. For an arbitrary $Z \in \Sigma(V)$, we know that $Z$ is open in $\overline Z$ (since $Z = \overline Z \cap U$), and that $\overline Z$ is noetherian (since $\overline Z \in \Sigma(V)$). Hence $Z$ is quasi-compact. Thus $U$ satisfies condition (2).  
	\end{proof}
	
	\begin{definition}\label{defn:GDL}
		Let $\sim$ be the minimal equivalence relation on $\BN[\CS^H]$ generated by the following rules. 
		\begin{enumerate}
					\item If there is a morphism $ V_1 \to V_2$ in $\CS^H$ such that forgetting the $H$-actions this is a Zariski-locally trivial $\BG_m^{\pfn}$-bundle, then $[V_1] \sim  \BT [V_2]$.
			\item If there is a morphism $ V_1 \to V_2$ in $\CS^H$ such that forgetting the $H$-actions this is a Zariski-locally trivial $\BA^{1, \pfn}$-bundle, then $[V_1] \sim  \BT [V_2] + [V_2]$.
			\item Suppose there is a morphism $V' \to V$ in $\CS^H$ that is a closed embedding. By Lemma \ref{lemma: complement in S^H}, the open subscheme $V \setminus V'$ of $V$ is an object in $\CS^H$. We require that $[V] \sim  [V']+[V \setminus V']$.
		\end{enumerate}
	\end{definition} 
	
	\subsubsection{}We recall the general notion of a quotient monoid. Let $(M,+)$ be a commutative monoid. An equivalence relation $\equiv$ on $M$ is called a \emph{congruence}, if for all $x,x',y,y' \in M$ such that $x\equiv x'$ and $y \equiv y'$, we have $x+y \equiv x' + y'$. If $\equiv$ is a congruence, then the quotient set $M/{\equiv}$ inherits from $M$ the structure of a commutative monoid. This is called the \emph{quotient monoid} of $M$ by $\equiv$. Starting with an arbitrary equivalence relation $\sim$ on $M$, we obtain a congruence $\equiv$ on $M$ by  declaring $x \equiv y$ if and only if we can write $x = \sum_{i=1}^n x_i$ and $y= \sum_{i=1}^ny_i$ for some $x_i, y_i \in M$ such that $x_i \sim y_i$ for each $i$.
	\begin{definition}
		Let $\equiv$ be congruence on $\BN[\CS^H]$ associated to $\sim$, and let $\GDL^H$ be the quotient monoid $\BN[\CS^H]/{\equiv}$. We call $\GDL^H$ the \emph{Grothendieck--Deligne--Lusztig monoid}. For any object $V$ in $\CS^H$, we denote the image of $[V]$ under $\BN[\CS^H] \to \GDL^H$  by $\dbp{V}$. 
	\end{definition}
	
	\subsubsection{}\label{subsubsec:defn BT}
	One easily checks that the endomorphism $\BT$ of $\BN[\CS^H]$ descends to an endomorphism of $\GDL^H$, which we still denote by $\BT$. We write $\BL$ for $\BT +1 \in \End(\GDL^H)$.

	\subsection{Calculus of top  irreducible components}\label{subsec:calculus} 
	\subsubsection{}\label{subsubsec:TIC}Let $H$ be an abstract group as before. One can formally calculate ``top-dimensional irreducible components'' of elements of $\GDL^H$. To this end we first introduce a commutative monoid $\TIC^H$ which is much simpler than $\GDL^H$ and serves to record information about top-dimensional irreducible components. Let $\CS et^H_f$ be the category of $H$-sets which contain only finitely many $H$-orbits. This is an essentially small category. We let $\TIC^H$ be the set of pairs $(\Sigma, d)$, where $\Sigma$ is an isomorphism class in $\CS et^H_f$, and $d \in \BZ_{\geq 0}$. Given two elements $(\Sigma_1, d_1), (\Sigma_2, d_2) \in \TIC^H$, we define their sum to be 
	$$(\Sigma_1, d_1) +  (\Sigma_2, d_2):=  \begin{cases}
	(\Sigma_1, d_1) , & \text{if } d_1>d_2 , \\
	(\Sigma_2, d_2) , & \text{if } d_2>d_1 , \\
	(\Sigma_1 \sqcup \Sigma_2, d_1) , & \text{if } d_1=d_2 .
	\end{cases}  $$
	This makes $\TIC^H$ a commutative monoid. In the above definition of the sum, if $d_1 \geq d_2$, then we say that $(\Sigma_1, d_1)$ \emph{makes non-trivial contribution to the sum}. 
	
	Define an endomorphism $\BT$ of $\TIC^H$ by
	$$\BT : (\Sigma, d) \longmapsto (\Sigma, d+1).$$ We write $\BL$ for $\BT+1 \in \End(\TIC^H)$; it is easy to see that in fact $\BL=\BT$ in $\End(\TIC^H)$.
	
	Note that every object $V$ in $\CS^H$ has finite Krull dimension. The sets $\Sigma(V)$ and $\Sigma^d(V)$ for all $d \in \BZ_{\geq 0}$ (Definition \ref{defn:Sigma}) equipped with the natural $H$-actions are all objects in $\CS et^H_f$.
	\begin{definition}\label{defn:Sigma top}
		For any $X = \sum_{i=1}^n [V_i] \in \BN[\CS^H]$, we define $$\dim X : = \max_{1\leq i \leq n} \dim V_i \in \BZ_{\geq 0}, $$ and define $$\Sigma^{\topp}(X) : = \coprod_{1\leq i \leq n} \Sigma^{\dim X}(V_i),$$ which is an object in $\CS et^H_f$. The pair consisting of the isomorphism class of $\Sigma^{\topp}(X)$ and the integer $\dim X$ is thus an element of $\TIC^H$, which we denote by $\tilde \fkC(X) \in \TIC^H$. 
	\end{definition}

	\begin{lemma}\label{lem:monoid homo} The map $\tilde \fkC: \BN[\CS^H] \to \TIC^H$ is a monoid homomorphism, and descends to a monoid homomorphism $\fkC: \GDL^H \to \TIC^H$. Moreover, $\fkC$ is equivariant with respect to the endomorphisms $\BT$ on $\GDL^H$ and $\BT$ on $\TIC^H$ (see \S \ref{subsubsec:defn BT} and \S \ref{subsubsec:TIC}).   
	\end{lemma} 
	\begin{proof} It follows from the definitions that $\tilde \fkC$ is a monoid homomorphism. To show that $\tilde \fkC$ descends to $\GDL^H$, it suffices to check that any $X , X' \in \BN[\CS^{H}]$ with $X \sim X'$ satisfies $\tilde \fkC(X) = \tilde \fkC(X')$. For this, we only need to analyze the three situations in Definition \ref{defn:GDL}. Namely, we may assume that $X$ and $X'$ are the two sides of $\sim$ in those situations. 
		
 In situation (1), we have $X = [V_1]$ and $X' = \BT [V_2]$. We have $\dim V_1 = \dim V_2 +1$, and taking the inverse image along $V_1 \to V_2$ induces an $H$-equivariant bijection $\Sigma^{\topp}(V_2) \isoarrow \Sigma^{\topp}(V_1)$. (In fact we have an $H$-equivariant bijection $\Sigma^{d}(V_2)  \isoarrow \Sigma^{d+1}(V_1)$ for arbitrary $d$.) Thus we have $\tilde \fkC([V_1]) = \BT \tilde \fkC([V_2])$. For the same reason, we also have $\tilde \fkC(\BT[V_2]) = \BT \tilde \fkC([V_2])$. Thus we have $\tilde \fkC([V_1]) = \tilde \fkC(\BT [V_2])$ as desired. 
 
 One treats situation (2) similarly, noting that $\tilde \fkC(\BT [V_2] +[V_2]) = \tilde \fkC(\BT[V_2])$. 
 
 Now consider situation (3). We have $X = [V]$ and $X' = [V
 '] + [V\setminus V']$. Observe that for each $Z \in \Sigma(V)$, precisely one of the following two statements holds:
		\begin{itemize}
			\item We have $Z \subset V'$, and $Z \in \Sigma (V')$. 
			\item The intersection $Z_1: = Z \cap (V\backslash V')$ is dense in $Z$. Moreover, $Z_1 \in \Sigma(V\backslash V')$, and $\dim Z_1 = \dim Z$ (cf.~the proof of Lemma \ref{lemma: complement in S^H}). 
		\end{itemize}
		It follows that for each $d \in \BZ_{\geq 0}$ we have an $H$-equivariant bijection 
		\begin{align*}
		\Sigma^d(V') \sqcup \Sigma^d(V\backslash V') & \isoarrow \Sigma^d (V) \\
		Z & \longmapsto \bar Z.
		\end{align*}
		Therefore we have $\tilde \fkC([V]) = \tilde \fkC([V'] + [ V \backslash V'])$ , as desired. We have proved that $\tilde \fkC$ descends to $\GDL^H $. 
		
For any $V$ in $\CS^H$, we have $\dim(V \times_{\kk} \mathbb G_m^{\pfn}) = \dim (V ) +1$, and we have a natural $H$-equivariant bijection $\Sigma^{\topp}(V \times_{\kk} \mathbb G_m^{\pfn}) \isoarrow \Sigma^{\topp} (V)$. It follows that $\tilde \fkC$ is equivariant with respect to $\BT$ on the two sides. Since $\fkC$ is induced by $\tilde \fkC$, it is also equivariant with respect to $\BT$ on the two sides. 
	\end{proof}
	
	\subsection{Class polynomials and motivic counting}\label{subsec:DLreduction} We assume that $G$ is as in \S\ref{subsec:ADLV}, i.e., $G$ is quasi-split, tamely ramified, and $\mathrm{char}(F)\nmid |\pi_1(G_{\mathrm{ad}})|$ if $\mathrm{char}(F)>0$. Then we have the  affine Deligne--Lusztig variety $X_w(b)$ associated to $w\in \breve W$ and $b\in G(\breve F)$. The motivation behind the definition of the Grothendieck--Deligne--Lusztig monoid is that it gives a natural setting to apply the Deligne--Lusztig reduction method for $X_w(b)$. We recall the reduction method in the proposition below. 
	\begin{proposition}\label{DLReduction1}
		Let $w\in \breve W$, $s\in \tS$, and $b \in G(\breve F)$. If $\charac (F) >0$, then the following two statements hold.
		\begin{enumerate}
			\item	If $\breve \ell(sw\s(s)) = \breve \ell(w)$, then there exists a $J_b(F)$-equivariant morphism $X_w(b) \to X_{sw\s(s)}(b)$ which is a universal homeomorphism.
			\item
			If $\breve \ell(sw\s(s)) = \breve \ell(w)-2$, then $X_w(b)$ has a $J_b(F)$-stable closed subscheme $X_1$ satisfying the following conditions:
			\begin{itemize}
		\item There exist a $\kk$-scheme $Y_1$ with a $J_b(F)$-action, and $J_b(F)$-equivariant morphisms $f_1: X_1\to Y_1$ and $g_1: Y_1 \to X_{sx\s(s)}(b)$, where $f_1$ is a Zariski-locally trivial $\BA^1$-bundle and $g_1$ is a universal homeomorphism.
				\item Let $X_2$ be the open subscheme of $X_w(b)$ complement to $X_1$, which is $J_b(F)$-stable. There exist a $\kk$-scheme $Y_2$ with a $J_b(F)$-action, and $J_b(F)$-equivariant morphisms $f_2: X_2\to Y_2$ and $g_2: Y_2 \to X_{sx}(b)$, where $f_2$ is a Zariski-locally trivial $\BG_m$-bundle and $g_2$ is a universal homeomorphism.
			\end{itemize} 
		\end{enumerate}
		If $\charac (F) =0$, then the above two statements still hold, but with ``$\BA^1$-bundle'' and ``$\BG_m$-bundle'' replaced by ``$\BA^{1,\pfn}$-bundle'' and ``$\BG_m^{\pfn}$-bundle'' respectively. 
	\end{proposition}
	\begin{proof} The equal characteristic case is proved in \cite[\S 2.5]{GoHe}. 
		The mixed characteristic case follows from the same proof.
	\end{proof}
	\subsubsection{} Let $w\in \breve W$ and $b \in G(\breve F)$. By the discussion in \S \ref{sssec: ADLV} and Lemma \ref{lem:quasi-compact}, we know that the perfection $X_w(b)^{\pfn}$ of $X_w(b)$ is an object in $\CS^{J_b(F)}$. (Of course $X_w(b) = X_w(b)^{\pfn}$ if $\charac(F) = 0$.) To simplify the notation, we write $\dbp{X_w(b)}$ for the element $\dbp{X_w(b)^{\pfn}} \in \GDL^{J_b(F)}$. 
	
	Using the formalism in \S \ref{subsec:Groth monoid}, we can reformulate Proposition \ref{DLReduction1} in the following proposition (which is weaker, but more convenient for applications). 

	\begin{proposition}\label{DL-red2} 	Let $w\in \breve W $, $s\in \tS$, and $b \in G(\breve F)$. The following statements hold. 
		\begin{enumerate}
			\item If $\breve \ell(sw\s(s)) = \breve \ell(x)$, then $$\dbp{X_{w}(b)} = \dbp{X_{s w \s(s)}(b)} \in \GDL^{J_b(F)}.$$
			\item If $\breve \ell(sw\s(s)) = \breve \ell(x)-2$, then $$\dbp {X_w(b)} =  (\mathbb L -1) \dbp{X_{s w}(b)}+  \mathbb L \dbp{X_{s w\s(s)}(b)} \in \GDL^{J_b(F)}.$$
		\end{enumerate}
	\end{proposition}
	\begin{proof}
		This follows from Proposition \ref{DLReduction1} and the following three observations. Firstly, if a morphism of $\kk$-schemes is universally homeomorphic, then the perfection of this morphism is an isomorphism, by \cite[Lemma 3.8]{BS}. Secondly, if a morphism of $\kk$-schemes is a Zariski-locally trivial $\BA^1$-bundle (resp.~ $\BG_m$-bundle), then the perfection of this morphism is a Zariski-locally trivial $\BA^{1, \pfn}$-bundle (resp.~$\BG_m^{ \pfn}$-bundle). Thirdly, the perfections of the $\kk$-schemes $X_1, X_2, Y_1, Y_2$ in Proposition \ref{DLReduction1} (2), equipped with the natural $J_b(F)$-actions, are all objects in $\CS^{J_b(F)}$. Indeed, the assertion for $X_2$ follows from the fact that $X_w(b)^{\pfn}$ is in $\CS^{J_b(F)}$ and Lemma \ref{lemma: complement in S^H}. The assertion for $Y_2$ follows from the fact that $X_{sx}(b)^{\pfn}$ is in $\CS^{J_b(F)}$, and the fact that the perfection of $g_2$ is a $J_b(F)$-equivariant isomorphism. The assertion for $Y_1$ follows from the fact that $X_{sx\sigma(x)}(b)^{\pfn}$ is in $\CS^{J_b(F)}$, and the fact that the perfection of $g_1$ is a $J_b(F)$-equivariant isomorphism. The assertion for $X_1$ follows from the assertion for $Y_1$, the fact that the perfection of $f_1$ is locally of perfectly finite type, and the fact that pulling back along the perfection of $f_1$ induces a $J_b(F)$-equivariant bijection $\Sigma(Y_1^{\pfn}) \isoarrow \Sigma (X_1^{\pfn})$. 
	\end{proof}
	\subsubsection{}In order to effectively use Proposition \ref{DL-red2} to study the $J_b(F)$-action on $\Sigma^{\topp}(X_w(b))$, we need a refined version of the class polynomials for affine Hecke algebras. We first recall the definition of the usual class polynomials. Here we use the convention of \cite[\S 2.8.2]{He-CDM}, which differs from that in \cite{He14}. 
	
	Let $\mathbf q$ be an indeterminate, and let $\BZ[\mathbf q^{\pm 1}]$ be the Laurent polynomial ring. Let $\BH$ be the affine Hecke algebra over $\BZ[\mathbf q^{\pm 1}]$ attached to $\breve W$. Thus $\BH$ is the associative $\BZ[\mathbf q^{\pm 1}]$-algebra generated by symbols $\{T_w\mid w \in \breve W\}$ subject to the following relations: 
	\begin{itemize}
		\item $T_{w} T_{w'}=T_{w w'}$ if $\breve \ell(w w')=\breve \ell(w)+\breve \ell(w')$; 
		
		\item $(T_s+1)(T_s- \mathbf q)=0$ for all $s \in \tS$. 
	\end{itemize}
	The action of $\s$ on $\breve W$ induces an automorphism $\s$ of $\BH$ characterized by $\s (T_w) =  T_{\s(w)}$ for all $w\in \breve W$. Define $[\BH, \BH]_{\sigma}$ to be the $\BZ[\mathbf q^{\pm 1}]$-submodule of $\BH$ generated by $h \sigma(h') - h' \sigma(h), $ where $h$ and $h'$ run over elements of $\BH$. Define the \emph{$\sigma$-cocenter}
	(or simply \emph{cocenter}) to be the quotient module $\bar \BH_\s : =\BH/[\BH, \BH]_\s$. 
	
	For any $\CO \in B(\breve W,\sigma)$, let $\CO_{\min}$ be the set of minimal length elements of $\CO$. By \cite[Theorem 5.3, Theorem 6.7]{HN}, the cocenter $\bar \BH_\s$ is a free $\BZ[\mathbf q ^{\pm 1}]$-module with a basis given by $\{T_\CO\mid \CO \in B(\breve W, \s)\}$. Here $T_\CO$ is the image of $T_w$ in $\bar \BH_{\s}$ for some (or equivalently, any) $w \in \CO_{\min}$. Moreover, for any $w \in \breve W$, we have $$T_w \equiv \sum_{\CO \in B(\breve W,\sigma)} F_{w, \CO} T_\CO \mod [\BH, \BH]_\s,$$ where $F_{w, \CO} \in \BZ[\mathbf q]$ is the \emph{class polynomial}, uniquely determined by the above identity. 
	
	\subsubsection{} As indicated above, we need a refinement of the polynomials $F_{w,\CO}$ where $(w,\CO) \in  \breve W  \times B(\breve W,\s)$. The refined polynomials will be indexed by pairs $(w, C) \in \breve W \times \sC(\breve W)$, where $\sC(\breve W)$ is a set more refined than $B(\breve W,\sigma)$. We now define $\sC(\breve W)$. For $w, w' \in  \breve W$ and $s \in \tS$, we write $$w \xrightarrow{s}_\s w' $$ if $w'=s w \s(s)$ and $\breve \ell(w') \le \breve  \ell(w)$. We write $$w \rightarrow_\s w' $$ if there is a sequence $w=w_1, w_2, \dotsc, w_n=w'$ in $\breve W$ such that for each $2 \leq k \leq n $ we have $w_{k-1} \xrightarrow{s_k}_\s w_k$ for some $s_k \in \tS$. We write $$w \approx_\sigma w'$$ if $w \to_\sigma w'$ and $w' \to_\sigma w$. We write $w\ \tilde{\approx}_\sigma w'$ if there exists $\tau\in \Omega$ such that $w \approx_\sigma \tau w'\sigma(\tau)^{-1}$. The following theorem is proved in \cite[Theorem 2.9]{HN}.
	\begin{theorem}\label{HN-min}
		Let $\CO$ be a $\s$-conjugacy class in $\breve W$. Then for each $w \in \CO$, there exists $w'\in \CO_{\min}$ such that $w \to_\s w'$. \qed 
	\end{theorem}
	
	\begin{definition}
		Let $\breve W_{\s, \min}$ be the set of $w\in \breve W$ such that $w$ has minimal length in its own $\s$-conjugacy class. We write $\sC(\breve W)$ for the set $\breve W_{\s, \min}/\tilde \approx_\s$, and we view each element of $\sC(\breve W)$ as a subset of $\breve W$. We denote by $\pi$ the natural map $\mathscr C(\breve W) \to B(\breve W , \s)$ sending $C \in \mathscr C(\breve W)$ to the unique $\s$-conjugacy class in $\breve W$ containing $C$. We denote the composition of the map (\ref{eq:Psi}) with $\pi$ by $\Psi: \mathscr C(\breve W) \to B(G) $. 
	\end{definition}

	\subsubsection{} 
	For any $C \in  \sC(\breve W)$ and $b \in G(\breve F)$, we write $\dbp{X_C(b)}$ for $\dbp{X_w(b)} \in \GDL^{J_b(F)}$ for arbitrary $w \in C$. By Proposition \ref{DL-red2} (1), the definition of $\dbp{X_C(b)}$ is independent of the choice of $w$. 
	
	We now construct the refined polynomials in the  following theorem. Let $\BN[\mathbf q-1]$ denote the set of polynomials in the variable $\mathbf q -1 $ with positive integral coefficients. The second statement in the theorem can be viewed as a ``motivic counting'' result. 
	
	\begin{theorem}\label{motivic}
		Fix $w \in \breve W$. There exists a map $$
		\sC (\breve W) \to \BN[\mathbf q-1] , \qquad C \mapsto F_{w,C} (\mathbf q -1) $$
		satisfying the following conditions. 
		\begin{enumerate}
			\item For each $\CO \in  B(\breve W,\s)$, we have $$F_{w, \CO} (\mathbf q)=\sum_{C \in \sC(\breve W) , \pi(C) =  \CO} F_{w, C}(\mathbf q -1) \in \BZ[\mathbf q]. $$ 
			In particular, we have $F_{w,\CO}(\mathbf q) \in \BN[\mathbf q-1]$.
			
			\item For each $b \in G(\breve F)$, we have
			\begin{align}\label{eq:expand X_w(b)}
			\dbp{X_w(b)}=\sum_{C \in \sC(\breve W), \Psi(C)=[b]} F_{w, C} (\BL -1) \cdot  \dbp{X_C(b)} \in \GDL^{J_b(F)}.
			\end{align}	
		\end{enumerate}
	\end{theorem}

	\begin{proof}
		We prove the statement by induction on $\ell(w)$.  
		
		If $w \in \breve W_{\s, \min}$, then by \cite[\S 2.8.2]{He-CDM}, for any $\CO \in B(\breve W, \s)$, we have $$F_{w, \CO}=\begin{cases} 1, & \text{ if } w \in \CO; \\ 0, & \text{ otherwise}. \end{cases}$$
		On the other hand, for $C \in \sC(\breve W)$, we set $$F_{w, C} : =\begin{cases} 1, & \text{ if } w \in C; \\ 0, & \text{ otherwise}. \end{cases}$$
		In this case, the map $C \mapsto F_{w,C}$ satisfies conditions (1) and (2). 
		
		Now assume that $w \notin \breve W_{\s, \min}$. Then by Theorem \ref{HN-min}, there exists $w' \in \breve W$ and $s \in \tS$ such that $w \tilde \approx_\s w'$ and $s w' \s(s)<w'$. By \cite[\S 2.8.2]{He-CDM}, for any $\CO \in B(\breve W, \s)$, we have $$F_{w, \CO}(\mathbf q)=(\mathbf q-1) F_{s w', \CO} (\mathbf q)+ \mathbf q F_{s w' \s(s), \CO} (\mathbf q).$$ For $C \in \sC(\breve W)$, we set $$F_{w, C} (\mathbf q -1):  =(\mathbf q-1) F_{s w', C} (\mathbf q-1) + \mathbf q F_{s w' \s(s), C} (\mathbf q-1), $$ where $F_{s w', C} (\mathbf q-1)$ and $F_{s w' \s(s), C} (\mathbf q-1)$ are defined by the induction hypothesis. Since condition (1) holds for $s w'$ and $s w' \s(s)$, it also holds for $w$. 
		
		By Proposition \ref{DL-red2} (2), for any $b\in G(\breve F)$ we have  $$\dbp{X_w(b)}=\dbp{X_{w'}(b)}=(\BL-1) \dbp{X_{s w'}(b)}+\BL \dbp{X_{s w' \s(s)}(b)} . $$ By the induction hypothesis, we have the following identities in  $\GDL^{J_b(F)}$: \begin{gather*}
		\dbp{X_{sw'}(b)}=\sum_{C \in \sC (\breve W) , \Psi(C)=[b]} F_{sw', C} (\BL - 1) \cdot \dbp{X_C(b)}, \\ 
		\dbp{X_{s w' \s(s)}(b)}=\sum_{C \in \sC(\breve W), \Psi(C)=[b]} F_{s w' \s(s), C}(\BL -1) \cdot \dbp{X_C(b)}.
		\end{gather*}
		
		Then 		
	\begin{align*}		
		\dbp{X_w(b)} &=(\BL-1) \dbp{X_{s w'}(b)}+\BL \dbp{X_{s w' \s(s)}(b)} \\ &=\sum_{C \in \sC(\breve W), \Psi(C)=[b]} \bigg((\BL-1) \cdot F_{sw', C}(\BL-1)+\BL \cdot F_{s w' \s(s), C}(\BL-1) \bigg) \cdot  \dbp{X_C(b)} \\ &=\sum_{C \in \sC(\breve W), \Psi(C)=[b]} F_{w, C} (\BL -1) \cdot \dbp{X_C(b)} \in \GDL^{J_b(F)}.
		\end{align*}
		Thus (2) holds for $w$.
	\end{proof}

	\begin{remark} 
		\begin{enumerate}
			\item The polynomials $F_{w, C}$ are not uniquely characterized by condition (1) in Theorem \ref{motivic}. This is because the cocenter of the affine Hecke algebra over $\BZ[\mathbf q]$ has a torsion part, cf.~\cite[\S 5.2]{He-zero}. (In contrast, as we have mentioned above, the cocenter of the affine Hecke algebra over $\BZ[\mathbf q^{\pm 1}]$ is free.)
			\item Fix $b\in G(\breve F)$, and let ${K}_0^{J_b(F)}$ be the Grothendieck group of the monoid $\GDL^{J_b(F)}$. The endomorphism $\BL$ of $\GDL^{J_b(F)}$ gives rise to a $\BZ[\mathbf q]$-module structure on $K_0^{J_b(F)}$ via the specialization $\mathbf q\mapsto \BL$. The $\BZ[\mathbf q]$-submodule of $K_0^{J_b(F)}$ generated by $\{\dbp{X_w(b)}\mid w\in \breve W\}$ is not necessarily torsion-free as a $\BZ[\mathbf q]$-module. It would be interesting to compare the torsion phenomenon here with the cocenter of the affine Hecke algebra over $\BZ[\mathbf q]$. 
			\item As we have seen in the proof of Theorem \ref{motivic}, the construction of $F_{w, C}$ depends on $G$ only via the triple $(\breve W, \breve \ell : \breve W \to \BZ_{\geq 0}, \s \in \Aut(\breve W))$. This will allow us to reduce the study of general $G$ to unramified groups. 
		\end{enumerate}
	\end{remark}

	\ignore{
		\begin{corollary}\label{cor:finiteness}
			For $w\in W$ and $b\in \breve G$, we have $X_w(b) \in S^{J_b,f}_{\kk}$. 
		\end{corollary}
		\begin{proof}
			By Theorem \ref{j-min}, we have $\dbp{X_C(b)} \in \GDL^{J_b,f}_{\kk}$ for any $C \in \sC$. Then by Lemma \ref{lem:T preserves finiteness}, the right hand side of (\ref{eq:expand X_w(b)}) lies in $\GDL^{J_b(F),f}_{\kk}$. Hence by (\ref{eq:expand X_w(b)}), we have $\dbp{X_w(b)} \in \GDL^{J_b(F),f}_{\kk}$. It then follows from Lemma \ref{lem:finiteness} that $X_w(b) \in S^{J_b(F),f}_{\kk}$.  
		\end{proof}

		\begin{proof}
			We argue by induction on $\ell(w)$. If $w$ is a minimal length element in its $\s$-conjugacy class, then by Theorem \ref{j-min}, $\sharp J_b \backslash \Sigma X_w(b) \le 1$. 
			
			If $w$ is not a minimal length element in its $\s$-conjugacy class, then by Theorem \ref{HN-min}, there exists $w' \in \breve \breve W$ and $s \in \tS$ such that $w \to_\s w'$, $\ell(w')=\ell(w)$ and $s w' \s(s)<w'$. Then by Proposition \ref{DLReduction1}, $X_w(b)$ and $X_{w'}(b)$ are universally $J_b$-homeomorphic. And there is an injection of $J_b$-sets $$\Sigma X_{w'}(b) \hookrightarrow \Sigma X_{s w'}(b) \sqcup \Sigma X_{s w' \s(s)}(b).$$ By inductive hypothesis, $J_b \backslash X_{s w'}(b)$ and $J_b \backslash X_{s w' \s(s)}(b)$ are finite. Hence $J_b \backslash X_{w}(b) \cong J_b \backslash X_{w'}(b)$ are still finite. 
		\end{proof}
	}
	
	\begin{corollary}\label{cor:motivic count irred comp} Let $w\in \breve W$ and $b \in G(\breve F)$. For each $C \in \sC(\breve W)$, choose an element $w_C \in C$. The isomorphism class of the $J_b(F)$-set $\Sigma^{\topp} (X_w(b))$ (resp.~the integer $\dim X_w(b)$) is given by the first (resp.~second) coordinate of the element $$ \sum_{C \in \sC(\breve W), \Psi(C) = [b]} F_{w,C} (\BL -1) \cdot \big (\Sigma^{\topp}(X_{w_C} (b)), \dim X_{w_C}(b)  \big ) \in  \TIC^{J_b(F)}.$$
	\end{corollary}
	\begin{proof} Note that the isomorphism class of the $J_b(F)$-set $\Sigma^{\topp} (X_w(b))$ and the integer $\dim X_w(b)$ do not change if we replace $X_w(b)$ by its perfection. The corollary then follows from applying the $\BT$-equivariant homomorphism $\fkC$ in Lemma \ref{lem:monoid homo} to the two sides of (\ref{eq:expand X_w(b)}). 
	\end{proof}
	
	\begin{remark}\label{rem:useful}
		Fix $b \in G(\breve F)$. For $x\in \breve W_{\s, \min}$, by \cite[Theorem 4.8]{He14} we know that $X_x(b) \neq \emptyset$ if and only if $\Psi(x) = [b]$, that $X_x(b)$ is equidimensional, and that the $J_b(F)$-action on $\Sigma^{\topp}(X_x(b))$ is transitive. Moreover, when $X_x(b) \neq \emptyset$, we have an explicit formula for $\dim X_x(b)$ (see \cite[Theorem 4.8]{He14}), and we know that the stabilizer of each irreducible component of $X_x(b)$ in $J_b(F)$ is a parahoric subgroup of $J_b(F)$ with an explicit description (see the proof of \cite[Proposition 3.1.4]{ZZ}). The upshot is that we explicitly understand the elements $(\Sigma^{\topp}(X_{w_C}(b)), \dim X_{w_C}(b) ) \in  \TIC^{J_b(F)}$ for all $C \in \sC(\breve W)$ and $w_C \in C$. Thus by Corollary \ref{cor:motivic count irred comp}, the determination of the $J_b(F)$-set $\Sigma^{\topp} (X_w(b))$ and $\dim X_w(b)$ for general $w \in \breve W$ reduces to the computation of the polynomials $F_{w,C}$. It also follows that for general $w$, the stabilizer of each element of $\Sigma^{\topp} (X_w(b))$ in $J_b(F)$ is a parahoric subgroup, cf.~\cite[Proposition 3.1.4]{ZZ}. 
	\end{remark}

	\subsection{Stabilizer of one irreducible component}\label{subsec: one comp}We keep the setting and notation of \S\ref{subsec:DLreduction}. 	In this subsection we assume in addition that $G$ is $F$-simple and adjoint.   We will apply the results in \S\ref{subsec:DLreduction} to study the stabilizers for the $J_b(F)$-action on $\Sigma^{\topp}(X_{w_0t^\mu}(b))$.
	
	\subsubsection{}
		Recall that for $\delta$ an automorphism of $(\breve W_a,\breve \BS)$ and ${\breve K}\subset \breve\BS$ a $\delta$-stable subset, a \emph{$\delta$-twisted Coxeter element} of $\breve W_{\breve K}$ is an element which can be written as $s_1\cdots s_n$, where $s_1,\dotsc, s_n \in \breve W_{\breve K}$ are distinct and form a set of representatives of the  $\delta$-orbits in ${\breve K}$. For $w\in \breve W_a$ we write $\supp_\delta(w)$ for the smallest $\delta$-stable subset ${\breve K}$ of $\breve \BS$ such that $w\in \breve W_{\breve K}$. As explained in \S \ref{subsubsec:cases we consider}, we identify $\breve W_0$ with the subgroup $\breve W_{\breve \BS_0}$ of $\breve W$. Note that every $w \in \breve W$ can be written in a unique way as $w=x t^\mu y$, where $\mu\in {X_*(T)_{\Gamma_0}}^+$, $x, y\in \breve W_0$, and $t^\mu y\in {^{\breve{\BS}_0}\breve W}$. Moreover, ${\breve \ell}(w)={\breve \ell}(x)+{\breve \ell}(t^\mu)-{\breve \ell}(y)$. 
	
	The following result gives a refinement of \cite[Proposition 11.6]{He14}. 
	
	\begin{proposition}\label{prop: He 11.6} Assume $G$ is $F$-simple and adjoint. Let ${\breve K}$ be a $\s$-stable subset of $\breve{\mathbb S}_0$. Let $w=xt^\mu y \in \breve W$, with $\mu\in X_*(T)_{\Gamma_0}^+$, $x, y\in \breve W_0$, and $t^\mu y\in {^{\breve{\BS}_0}\breve W}$. Assume that $\mu\neq 0$, that $\mathrm{supp}_\sigma (x)={\breve K}$, and that $y$ is a $\sigma$-twisted Coxeter element of $\breve W_{\breve{\BS}_0\setminus {\breve K}}$. Then there exists a $\sigma$-twisted Coxeter element $c$ of $\breve W_0$ with $t^\mu c \in {^{\breve{\BS}_0}\breve W}$ such that for each $b \in G(\breve F)$, we have  $$\dbp{X_w(b)}  = (\BL-1)^{{\breve \ell}(x)}\dbp{X_{t^\mu c}(b)} +P  \in  \GDL^{J_b(F)}$$ for some $P \in \GDL^{J_b(F)}$. 
	\end{proposition}
	
	\begin{proof} 
		We follow the method in \cite[Proposition 11.6]{He14}.

		We proceed by induction on $|{\breve K}|$. The case $|{\breve K}|=0$ is clear, as we can take $c=y$. We thus assume that the result is true for all ${\breve K}'\subsetneq {\breve K}$. We may also assume that the result is true for all $x'\in \breve W_{\breve K}$ with $\supp_\sigma(x')={\breve K}$ and ${\breve \ell}(x')<{\breve \ell}(x)$.
		We set ${\breve K}_1:=\{s\in {\breve K}\mid t^\mu y s\notin{^{\breve\BS_0}}\breve W\}$. Then as in \cite[Proposition 11.6]{He14}, ${\breve K}_1$ is a proper subset of ${\breve K}$, and every $s\in {\breve K}_1$ commutes with $y$ and with $t^\mu y$.
		
		We write $x=ux'$ where $u\in \breve W_{\sigma^{-1}({\breve K}_1)}$ and $x'\in{^{\sigma^{-1}({\breve K}_1)}}\breve W$. We let $u=s_1\cdots s_n$ be a reduced word decomposition for $u$. We  write $u_i=s_1
		\cdots  s_i$ and set $x_i=u_i^{-1}x\sigma(u_i)$ for $i=1,\dotsc,n$. Then there are two possibilities:
		
		Case (i): There exists $k$ such that ${\breve \ell}(x)={\breve \ell}(x_i)$ for $i=1,\dotsc, k-1$ and ${\breve \ell}(x_{k})<{\breve \ell}(x)$. 
		
		Case (ii): ${\breve \ell}(x_i)={\breve \ell}(x)$ for all $i=1,\dotsc,n$.
		
		In Case (i), we have $$\dbp{X_w(b)}=\dbp{X_{x_{k-1}t^\mu y}(b)}=(\BL-1)\dbp{X_{s_kx_{k-1}t^\mu y}(b)}+\BL \dbp{X_{x_k t^\mu y}(b)}$$ by Proposition \ref{DL-red2}. 
		Since ${\breve \ell}(s_kx_{k-1}\s(s_{k}))<{\breve \ell}(x_{k-1})$, we have $\supp_\s(s_kx_{k-1})=\supp_\s(x)={\breve K}$. Thus by induction hypothesis, we have $$\dbp{X_{s_kx_{k-1}t^\mu y}(b)}=(\BL-1)^{{\breve \ell}(s_kx_{k-1})}\dbp{X_{t^\mu c}(b)}+P'$$ for some $\sigma$-twisted Coxeter element $c$ of $\breve W_0$ with $t^\mu c \in {^{\breve{\BS}_0}\breve W}$ and $P'\in \GDL^{J_b(F)}$. Since ${\breve \ell}(s_kx_{k-1})={\breve \ell}(x)-1$, we have $$\dbp{X_w(b)}=( \BL-1)^{{\breve \ell}(x)}\dbp{X_{t^\mu c}(b)}+ P$$ with $P\in \GDL^{J_b(F)}$ as desired.
		
		In Case (ii), we have $x\approx_{\sigma}x_n= x'\sigma (u)$. It follows that $w\approx_{\sigma}x_nt^\mu y$ and hence  $$\dbp{X_{w}(b)}=\dbp{X_{x_nt^\mu y}(b)}\in \GDL^{J_b(F)}$$ by Proposition \ref{DL-red2}. 
		We first consider the case where $x\notin \breve W_{\sigma^{-1}({\breve K}_1)}$. Then $x'\neq 1$ and there exists $s\in \sigma^{-1}({\breve K})\setminus \sigma^{-1}({\breve K}_1)$ such that $sx_n< x_n$. Moreover  ${\breve \ell}(y\sigma(s))={\breve \ell}(y)+1$ and $t^\mu y\sigma(s)\in {^{\breve\BS_0}\breve W}$ and we have $${\breve \ell}(sx_nt^\mu y \s(s))={\breve \ell}(sx_n)+{\breve \ell}(t^\mu)-{\breve \ell}(y\sigma(s))={\breve \ell}(x_n t^\mu y)-2.$$
		It follows that 
		\begin{align*}\dbp{X_{w}(b)}&=\dbp{X_{x_nt^\mu y}(b)}\\&=(\mathbb L-1)\dbp{X_{sx_n t^\mu y}(b)}+\mathbb L \dbp{X_{sx_nt^\mu y \sigma(s)}(b)}
		\\&=(\mathbb L-1)\dbp{X_{sx_n t^\mu y}(b)}+(\mathbb L-1)\dbp{X_{sx_nt^\mu y \sigma(s)}(b)}+\dbp{X_{sx_nt^\mu y \sigma(s)}(b)}.
		\end{align*}
		If $\supp_\s(sx_n)={\breve K}$, the  induction hypothesis applied to $X_{sx_nt^\mu y}(b)$ gives $$\dbp{X_{sx_nt^\mu y}(b)}=(\mathbb L-1)^{{\breve \ell}(sx_n)}\dbp{X_{t^\mu c}(b)}+P'$$
		for some $\sigma$-twisted Coxeter element $c$ of $\breve W_{0}$ and $P'\in\GDL^{J_b(F)}$. It follows that $$\dbp{X_{w}(b)}=(\mathbb L-1)^{{\breve \ell}(x)}\dbp{X_{t^\mu c}(b)}+P$$ with $P\in \GDL^{J_b(F)}$.
		
		Similarly, if $\supp_\sigma(sx_n)\neq {\breve K}$, the induction hypothesis applied to $X_{sx_nt^\mu y\s(s)}(b)$ gives $$\dbp{X_{sx_nt^\mu y\s(s)}(b)}=(\mathbb L-1)^{{\breve \ell}(sx_n)}\dbp{X_{t^\mu c}(b)}+P'$$ for some $P'\in \GDL^{J_b(F)}$, and hence
		$$\dbp{X_{w}(b)}=(\mathbb L-1)^{{\breve \ell}(x)}\dbp{X_{t^\mu c}(b)}+P$$ with $P\in \GDL^{J_b(F)}$.
		
		Finally we consider the case $x\in \breve W_{\sigma^{-1}(\breve K_1)}$.	Since ${\breve K}_1$ is a proper subset of ${\breve K}$ and $\supp_\s(x)={\breve K}$, there exists $m \in \BN$ such that $x, \s(x), \dotsc, \s^{m-1}(x) \in \breve W_{\sigma^{-1}({\breve K}_1)}$ and $\s^m(x) \notin \breve W_{\sigma^{-1}({\breve K}_1)}$. We have $$\dbp{X_{xt^\mu y}(b)}=\dbp{X_{\sigma (x)t^\mu y}(b)}=\dotsc=\dbp{X_{\sigma^{m-1} (x)t^\mu y}(b)}.$$ The argument above applied to $\sigma^m(x)$ shows that $$\dbp{X_{\sigma(m)t^\mu y}(b)}=(\mathbb L-1)^{{\breve \ell}(x)}\dbp{X_{t^\mu c}(b)}+P,$$ for some $\sigma$-twisted Coxeter element $c\in \breve W_0$ and $P\in \mathbb \GDL^{J_b(F)}$ as desired.
	
	\end{proof}
	
	\subsubsection{} For an element $\tau\in \Omega$, the Iwahori--Weyl group and affine Weyl group of $J_{\dot{\tau}}$ are isomorphic to $\breve W$ and $\breve W_a$ respectively, and the  Frobenius actions are both given by $\Ad(\tau)\circ \sigma$.
	
	We need the following result which is proved in \cite{HY}. Set $V:=X_*(T)_{\Gamma_0}\otimes_{\BZ}\BR$. 
	
	\begin{proposition}\label{H-Y}
		Let $p: \breve W \subset \mathrm{Aff}(V) \to \GL(V)$ be the natural map. Consider the $\s$-twisted conjugation action of $\breve W_a$ on $\breve W$. Let $\CO$ be a $\breve W_a$-orbit in $\breve W$ with $\CO\subset \breve W_a\tau$ for some $\tau\in\Omega$. If $p(\CO)\subset \breve W_0$ contains a $\sigma$-twisted Coxeter element of $\breve W_0 $, then there exists a unique $\Ad(\t) \circ \s$-stable subset ${\breve K}$ of $\breve \BS$ such that $W_{\breve K}$ is finite and the set  $\CO_{\min}$ of minimal length elements of $\CO$ is precisely the set of $\mathrm{Ad}(\tau)\circ\sigma$-twisted Coxeter elements of $\breve W_{\breve K}$. Moreover, the standard parahoric subgroup of $J_{\dot{\tau}}(F)$ corresponding to $\breve K$ is very special.
	\end{proposition}

	\begin{remark} In Proposition \ref{H-Y}, the unique ${\breve K}$ is explicitly computed in each case in \cite{HY}. The ``moreover'' part of the proposition immediately follows from the explicit description. 
	\end{remark}
	
	The main result of this subsection is the following proposition.
	
	\begin{proposition}\label{prop:one maximal}
		Assume that $G$ is $F$-simple and adjoint. Let $[b]\in B(G,\mu)$ be the unique basic element.  Then there exists $Z \in \Sigma^{\topp} (X_{w_0 t^\mu}(b))$ such that $\Stab_Z (J_b(F))$ is a very special parahoric subgroup of $J_b(F)$. 
	\end{proposition}
	
	\begin{proof}
		Since $\mu$ is dominant, $t^\mu\in {}^{\breve \BS_0}\breve W$. If $\mu=0$, then we may take $b=1$. In this case, $J_b(F)=G(F)$ and $X_\mu(b)=G(F)/\CK$ is discrete; here $\CK\subset G(F)$ is the parahoric subgroup corresponding to $\tS_0$ which is very special (cf.~Remark \ref{rem: very special}). For any $Z \in X_\mu(b)$, the stabilizer $\Stab_Z (J_b(F))$ is conjugate to $\CK$ and thus is a very special parahoric subgroup of $G(F)$. Now the statement on $X_{w_0 t^\mu}(b)$ follows from Proposition \ref{prop: relation between Iw and Sp ADLV}. 
		
		Now assume that $\mu \neq 0$. By Proposition \ref{prop: He 11.6} applied to ${\breve K}=\breve\BS_0$ and $w=w_0t^\mu$, there exists a $\s$-twisted Coxeter element $c$ of $\breve W_0$ such  that  $$\dbp{X_{w_0t^\mu}(b)}=(\mathbb L-1)^{{\breve \ell}(w_0)}\dbp{X_{t^\mu c}(b)}+P',$$ where $P'\in\mathbb \GDL^{J_b(F)}$.
		
		Let $\tau\in \Omega$ be the unique element such that $\kappa(\tau)=\mu^\natural\in\pi_1(G)_{\Gamma_0}$. Upon replacing $b$ by another representative in $[b]$, we may assume $b=\dot{\tau}$. By Proposition \ref{H-Y} and Theorem \ref{HN-min}, there exists an $\mathrm{Ad}(\tau)\circ\s$-stable subset $\breve K\subset \breve \BS$ and an $\mathrm{Ad}(\tau)\circ\s$-twisted Coxeter element $c'$ of $\breve W_{{\breve K}}$ such that the associated parahoric $\CJ$ of $J_{\dot{\tau}}(F)$ is very special, $c'\tau$ is of minimal length in its $\s$-conjugacy class, and $t^\mu c\rightarrow_{\sigma}c'\tau$.  
		
		By Proposition \ref{DL-red2}, we have 
		\begin{align*}\dbp{X_{t^\mu c}(\dot{\tau})}&=(\mathbb L-1)^{\frac{{\breve \ell}(t^\mu c)-{\breve \ell}(c')}{2}}\dbp{X_{c'\tau}(\dot{\tau})}+Q
		\end{align*}for some $Q\in\GDL^{J_{\dot \tau}(F)}$ 
		and hence \begin{align*}\dbp{X_{w_0t^\mu}(\dot{\tau})}= (\mathbb L-1)^{{\breve \ell}(w_0)+\frac{{\breve \ell}(t^\mu c)-{\breve \ell}(c')}{2}}\dbp{X_{c'\tau}(\dot{\tau})}+P\end{align*}
		for some $P\in \GDL^{J_{\dot{\tau}}(F)}$. By Lemma \ref{lem:monoid homo}, the above equality implies that 
		\begin{align}\label{eq:two contri} \fkC(\dbp{X_{w_0t^\mu}(\dot{\tau})})
			=\BT^{{\breve \ell}(w_0)+\frac{{\breve \ell}(t^\mu c)-{\breve \ell}(c')}{2}}\fkC(\dbp{X_{c'\tau}(\dot{\tau})})+H\end{align}
		for some $H\in \TIC^{J_{\dot{\tau}}(F)}$. Here on the right side, the addition is in the monoid $\TIC^{J_{\dot{\tau}}(F)}$.
		
		By \cite[\S 1.9]{Ko06}, ${\breve \ell}(c)-{\breve \ell} (c')=\mathrm{def}_G(\dot\tau)$. By Theorem \ref{thm:dim-adlv} and \cite[Theorem 10.1]{He14},  
		\begin{align*}
		\dim X_{w_0t^\mu}(\dot{\tau})&={\breve \ell}(w_0)+\dim X_\mu(\dot{\tau})\\
		&={\breve \ell}(w_0)+\langle \mu,\rho\rangle+\frac{{\breve \ell}(c')-{\breve \ell}(c)}{2}\\
		&={\breve \ell}(w_0)+\frac{{\breve \ell}(t^\mu)-{\breve \ell}(c)+{\breve \ell}(c')}{2}\\
		&={\breve \ell}(c')+ {\breve \ell}(w_0)+\frac{{\breve \ell}(t^\mu c)-{\breve \ell} (c')}{2}\\
		&=\dim( X_{c' \t}(\dot{\tau}))+{\breve \ell}(w_0)+\frac{{\breve \ell}(t^\mu c)-{\breve \ell}(c')}{2},
		\end{align*}
		where the fourth equality follows from the fact that $t^\mu c\in {^{\tS_0}\breve W}$.

		By the above computation, the first term in the sum $$\BT^{{\breve \ell}(w_0)+\frac{{\breve \ell}(t^\mu c)-{\breve \ell}(c')}{2}}\fkC(\dbp{X_{c'\tau}(\dot{\tau})})+H$$ makes a non-trivial contribution to the sum in the sense of \S \ref{subsubsec:TIC}. Thus we have a $J_{\dot \tau}(F)$-equivariant embedding $\Sigma^{\topp}(X_{c'\tau}(\dot{\tau})) \to \Sigma^{\topp}(X_{w_0t^\mu}(\dot{\tau}))$. It remains to find an element of $\Sigma^{\topp} (X_{c'\tau}(\dot{\tau}))$ whose stabilizer in $J_{\dot{\tau}}(F)$ is a very special parahoric subgroup.

		By \cite[Theorem 4.8]{He14}, $X_{c' \t}(\dot\t) \cong J_{\dot\t}(F) \times^{\CJ} X^{\breve \CK}_{c' \t}(\dot\t)$, where $X^{\breve \CK}_{c' \t}(\dot\t)$  is a classical Deligne--Lusztig variety (resp.~perfection of a classical Deligne--Lusztig variety) if $\mathrm{char}(F)>0$ (resp.~$\mathrm{char}(F)=0$) defined by \begin{align*} X^{\breve \CK}_{c' \t}(\dot\t) \cong \{g \overline{\CI}\in \overline{ \CK}/\overline{\CI}\mid g \i \s'(g) \in \overline{\CI} c' \overline{\CI}\}.\end{align*} 

		Here  $\s'=\Ad(\t) \circ \s$, and note that $\overline{\CI}$ is a $\s'$-stable Borel subgroup of $\overline{\CK}$.
		
		Since $c'$ is a $\s'$-Coxeter element of $\breve W_{\breve K}$, $X^{\breve \CK}_{c' \t}(\dot\t)$ is irreducible. Hence $$\Sigma^{\topp}( X_{c' \t}(\dot\t)) \cong J_{\dot\t}(F)/{\CJ}$$ as $J_{\dot\t}(F)$-sets and the stabilizer of the elements are isomorphic to $\CJ$. 
	\end{proof}

	\section{Component stabilizers for $X_\mu(b)$}\label{sec: main}

	\subsection{The main theorem and some consequences} \label{subsec: MV}\subsubsection{}
	We keep the notation and assumptions of \S\ref{subsec:ADLV}. In particular, $G$ is a quasi-split tamely ramified reductive group over $F$, and $\mathrm{char}(F) \nmid|\pi_1(G_{\mathrm{ad}})|$ if $\mathrm{char}(F)>0$.
	
	 We now state our main theorem, which confirms conjectures made by X.~Zhu and Rapoport. 
	\begin{theorem}\label{thm:main} Let $\mu \in X_*(T)_{\Gamma_0}^+$ and $[b]\in B(G,\mu)$. Each stabilizer for the $J_b(F)$-action on $\Sigma^{\topp} (X_{\mu}(b))$ is a very special parahoric subgroup of $J_b(F)$.
	\end{theorem}
By  \cite[Proposition 3.1.4]{ZZ}, we already know that each stabilizer for the $J_b(F)$-action on $\Sigma^{\topp} (X_{\mu}(b))$ is a parahoric subgroup of $J_b(F)$. In the proof of Theorem \ref{thm:main} below we shall freely use this fact. 

We now deduce an immediate consequence of Theorem \ref{thm:main}. 
\begin{definition}\label{defn:N(mu,b)}
	Fix $\mu$ and $b$ as in Theorem \ref{thm:main}. We write $\mathscr N(\mu,b)$ for the number of $J_b(F)$-orbits in $\Sigma^{\topp}(X_\mu (b))$. 
\end{definition}	
	
	\begin{corollary}\label{cor: main} There is an identification of $J_b(F)$-sets $$\Sigma^{\topp} (X_\mu (b))\cong \coprod_{i=1}^{\mathscr N(\mu,b)} J_b(F)/\CJ_i,$$ where $\CJ_i\subset J_b(F)$ is a very special parahoric subgroup for each $i$. \qed
	\end{corollary}
	 
	\subsubsection{}\label{subsec: pinning dual group} When $G$ is unramified, an explicit formula for $\mathscr N(\mu,b)$ was conjectured by M.~Chen and X.~Zhu, and was proved independently by the second and third named authors in \cite{ZZ} and by S.~Nie in \cite{nienew}. In the appendix of \cite{ZZ}, a generalization of this formula for ramified $G$ is given. We now recall this formula when $G$ is unramified, as this will be needed in \S \ref{subsec: reduction to basic} below.  
	  
	Consider the dual group  $\widehat{G}$ of $G$ over $\BC$. We fix a pinning $(\widehat{B}, \widehat{T}, \widehat{\mathbb X}_+)$ of $\widehat {G}$, and fix an isomorphism between the based root datum of $(\widehat G, \widehat B,\widehat T)$ and the dual of the based root datum of $(G,B,T)$. (See \S \ref{subsubsec:(B,T)} for $B$.) We then have a unique $\Gamma$-action on $\widehat G$ via automorphisms preserving $(\widehat{B}, \widehat{T}, \widehat{\mathbb X}_+)$ such that the induced $\Gamma$-action on the based root datum of $(\widehat G, \widehat B,\widehat T)$ is compatible with the natural $\Gamma$-action on the based root datum of $(G,B,T)$, see for instance \cite[\S 5.1]{ZZ}. Now $\Gamma_0$ acts trivially on $X_*(T)$, so the element $\mu\in X_*(T)_{\Gamma_0}^+$ can be viewed as a $\widehat{B}$-dominant character of $\widehat{T}$. Let $V_\mu$ be the irreducible representation of $\widehat{G}$ of highest weight $\mu$. Let $\widehat{\mathcal S}$ be the identity component of the $\Gamma$-fixed points of $\widehat T$. Then $X^*(\widehat {\CS})$ is identified with the maximal torsion-free quotient  of $X_*(T)_{\Gamma} = X_*(T)_{\sigma}$.  As in \cite[Definition 2.6.4]{ZZ}, $b$ determines an element  $\l_b \in X^*(\widehat {\CS})$. We omit the explicit definition of $\lambda_b$ here. Let $V_{\mu}(\lambda_b)$ be the weight space in the $\widehat \CS$-representation $V_{\mu}$ of weight $\lambda_b$. The geometric Satake provides us with a canonical basis $\BM\BV_\mu(\lambda_b)$ of  $V_\mu(\lambda_b)$.
	
	In the theorem below, the numerical identity is proved independently by the second and third named authors \cite[Theorem A]{ZZ} and Nie \cite[Theorem 0.5]{nienew}. The second statement is due to Nie \cite[Theorem 0.5]{nienew}. 
	\begin{theorem}\label{thm:ZZA} Keep the assumptions in \S \ref{subsec:ADLV}, and assume that $G$ is unramified over $F$. 
		We have $$\mathscr N(\mu,b)=\dim V_\mu(\lambda_b).$$ Moreover, there is a natural bijection between $J_b(F)\backslash \Sigma^{\topp}(X_\mu(b))$ and $\BM\BV_\mu(\lambda_b)$. \qed
	\end{theorem}

	\subsection{Reduction to adjoint unramified $F$-simple groups in characteristic $0$}\label{subsec: reduction to char 0}
	In this subsection, we show that to prove Theorem \ref{thm:main},  it suffices to prove it in the case where $\charac(F)=0$, and $G$ is an adjoint $F$-simple unramified group over $F$.
	\subsubsection{}
	
	Let $w \in \breve W$ and $[b] \in B(G)$. We first construct some combinatorial data involving only the affine Weyl group $\breve W_a$ together with the length function $\breve \ell$ and the action of $\s$ on $\breve W_a$, but not the reductive group $G$. This allows us to connect different reductive groups over different local fields. 
	
	Let $\Aut^0(\breve W_a)$ be the group of length-preserving automorphisms of $\breve W_a$. We may regard $\s$ as an element of $\Aut^0(\breve W_a)$. Let $\widehat  W_a=\breve W_a \rtimes \Aut^0(\breve W_a)$. We have a natural group homomorphism $$i: \breve W \to \widehat W_a, \qquad w \t \mapsto (w, \Ad(\t)) \text{ for } w \in \breve W_a, \t \in \Omega.$$ Moreover, the map $i$ is compatible with the actions of $\s$. (Here the action of $\s$ on $\widehat W_a$ is given by $(w,f) \mapsto (\s (w), \s \circ f \circ \s^{-1})$.) 
	
	\subsubsection{}\label{subsubsec:association}By \cite[Theorem 3.7]{He14}, the set $B(G)$ is in natural bijection with a certain subset of $\s$-conjugacy classes in $\breve W$. By composing with the map $i$, we may associate to any $[b] \in B(G)$ a $\s$-conjugacy class $C_{[b]}$ in $\widehat W_a$. Let $G'$ be a connected reductive group over a (possibly different) local field $F'$, let $b' \in G'(\breve {F'})$, and let $w'$ be an element of the Iwahori--Weyl group $\breve W'$ of $G_{\breve F'}$. Note that any length-preserving isomorphism of $\breve W_a$ to $\breve W'_a$ extends in a unique way to a group isomorphism $\widehat W_a \to \widehat W'_a$. Write $\s'$ for the Frobenius in $\Aut(\breve F'/F)$, and write $[b']$ for the $\s'$-conjugacy class of $b'$ in $G'(\breve F')$. Then $[b']$ determines a $\s'$-conjugacy class $C_{[b']}$ in $\widehat W'_a$. We say that the triples $(G, b, w)$ and $(G', b', w')$ are {\it associated} if the following conditions are satisfied:
	\begin{itemize}
		\item We have $\k_G(w)=\k_G(b)$ and $\k_{G'}(w')=\k_{G'}(b')$. 
		\item There exists a length-preserving isomorphism $f: \breve W_a \isoarrow \breve W'_a$ such that the diagram
		\[
		\xymatrix{
			\breve W_a \ar[r]^f \ar[d]_\s & \breve W'_a \ar[d]^{\s'} \\
			\breve W_a \ar[r]^f & \breve W'_a
		}
		\] commutes, and we have $f(C_{[b]})=C_{[b']}$ and $f(i(w))=i'(w')$. Here $i':\breve W'\rightarrow\widehat{W}'_a$ is the natural homomorphism analogous to $i$.  
	\end{itemize}  
In this case, $f$ induces an isomorphism from the  affine Weyl group of $J_b$ to the affine Weyl group of $J_{b'}$. We thus obtain a bijection between the standard parahoric subgroups of $J_b(F)$ and those of $J_{b'}(F)$, cf.~\cite[Lemma 3.2.2]{ZZ}. Let $\CJ\subset J_b(F)$ and $\CJ'\subset J_{b'}(F')$ be parahoric subgroups. We say that $\CJ$ and $\CJ'$ are \textit{associated with respect to $f$}, if there exist  $j\in J_b(F)$ and $j'\in J_{b'}(F')$ such that $j\CJ j^{-1}$ and $j'\CJ'j'^{-1}$ are standard parahoric subgroups which correspond to each other under the above-mentioned bijection.
	
	\begin{proposition}\label{prop: assoc parahorics}
		Suppose that $(G, b, w)$ and $(G', b', w')$ are associated, and fix $f: \breve W_a \isoarrow \breve W'_a$ as in \S \ref{subsubsec:association}. Then there is a bijection $$\Theta:J_b(F)\backslash\Sigma^{\topp}( X_w(b))\isoarrow J_{b'}(F')\backslash\Sigma^{\topp }(X_{w'}(b'))$$ satisfying the following condition:
		
		 For $Z\in \Sigma^{\topp}(X_w(b))$ and $Z'\in \Sigma^{\topp}(X_{w'}(b'))$ such that $\Theta(J_b(F)Z)=J_{b'}(F)Z'$, the parahoric subgroups $\Stab_Z(J_b(F))$ and $\Stab_{Z'} (J_{b'}(F'))$ are associated with respect to $f$.
	\end{proposition}
	
	\begin{proof}
		By Corollary \ref{cor:motivic count irred comp}, the isomorphism class of the $J_b(F)$-set $\Sigma^{\topp} (X_{w}(b)) $ depends only on $F_{w,C}$ and the $J_b(F)$-sets $\Sigma^{\topp}(X_u(b))$ for $u\in \breve W_{\s,\min}$. Similarly, the isomorphism class of the $J_{b'}(F')$-set $\Sigma^{\topp} (X_{w'}(b')) $  depends only on $F_{w',C'}$  and $\Sigma^{\topp}(X_{u'}(b'))$ for $u'\in \breve W'_{\s',\min}$. Here $C'$ runs over $\mathscr C(\breve W')$, the set of $\tilde{\approx}_{\s'}$-equivalence classes in $\breve W'_{\s',\min}$. 
		
		Note that in the proof of Theorem \ref{motivic}, the construction of the polynomials $F_{w, C}$ only involves $\s$-conjugation of $\breve W_a$ on $\breve W$, and the construction remains the same if we replace $\breve W$ by $\widehat W_a$. The identification $f: \breve W_a \isoarrow \breve W
		_a'$ induces an identification $\mathscr C (\breve W)\isoarrow \mathscr C(\breve W')$, and we choose the polynomials $F_{w',C'}$ such that $F_{w,C}=F_{w',C'}$ if $C$ and $C'$ correspond to one another under $\mathscr C (\breve W)\isoarrow \mathscr C(\breve W')$. It follows that it suffices to prove the proposition for $w\in \breve W_{\s,\min}$.
		
		Assume that $w\in \breve W_{\s,\min}$. Since the maps $i: \breve W \to \widehat  W_a$ and $i':\breve W' \to \widehat  W'_a$ are compatible with the conjugation actions and preserves the length functions, we have $w' \in \breve W'_{\s',\min}$. By \cite[Theorem 4.8]{He14}, $J_b(F)$ (resp.~$J_{b'}(F')$) acts transitively on $\Sigma^{\topp} (X_w(b))$ (resp.~$\Sigma^{\topp }(X_{w'}(b'))$) and hence we obtain the desired bijection $\Theta$. The ``moreover'' part follows from the explicit description of $X_w(b)$ in terms of finite Deligne--Lusztig varieties given in the proof \cite[Theorem 4.8]{He14}, cf.~the proof of \cite[Proposition 3.1.4]{ZZ}.
	\end{proof}
	\begin{corollary}\label{cor:red1}
		To prove Theorem \ref{thm:main}, it suffices to prove it when $\mathrm{char}(F)=0$ and $G$ is an adjoint $F$-simple unramified group over $F$.
	\end{corollary}
	
	\begin{proof} We first assume Theorem \ref{thm:main} is true for unramified adjoint groups over local fields of characteristic $0$. Let $G$ be an arbitrary (i.e., quasi-split, tamely ramified, reductive) group over an arbitrary local field $F$. By Proposition \ref{prop: relation between Iw and Sp ADLV}, it suffices to show that the stabilizer in $J_b(F)$ of every element of $\Sigma^{\topp} (X_{w_0t^\mu}(b))$ is a very special parahoric of $J_b(F)$.
		
		Since $X_\mu(b) \neq \emptyset$, we have $X_{w_0 t^\mu}(b) \neq \emptyset$ and thus $\k_G(w_0 t^\mu)=\k_G(b)$. By \cite[Theorem 3.7]{He14}, there exists $w \in \breve W$ such that $[b]=[\dot w]$. In particular, we have $\k_G(w_0 t^\mu)=\k_G(w)$. By replacing $w$ by a suitable element in the $\s$-orbit of $w$, we may assume furthermore that $w_0 t^\mu \breve W_a=w \breve W_a$. 
		
		We choose $F'$ a local field of characteristic $0$ and $G'$ an adjoint unramified group over $F'$ such that there is a length-preserving isomorphism $f: \breve W_a \to \breve W'_a$ such that $f \circ \s=\s' \circ f$; see \cite[\S\S 6.1, 6.2]{He14} for the construction of such a group. Since $G'$ is adjoint, we have $f(i(\breve W)) \subset i'(\breve W')$. Let $\mu' \in X_*(T')$ with $f(i(t^\mu))=i'(t^{\mu'})$. Then $f(i(w_0 t^\mu))=i'(w'_0 t^{\mu'})$. Let $w' \in \breve W'$ with $f(i(w))=i'(w')$ and $[b'] \in B(G')$ with $[b']=[\dot w']$. Since $w_0 t^\mu \breve W_a=w \breve W_a$, we have $w'_0 t^{\mu'} \breve W'_a=w' \breve W'_a$. Therefore $\k_{G'}(b')=\k_{G'}(w')=\k_{G'}(w'_0 t^{\mu'})$. 
		
		Hence $(G, b, w_0 t^\mu)$ and $(G', b', w'_0 t^{\mu'})$ are associated. 
		
		By Proposition \ref{prop: assoc parahorics}, for any $Z\in\Sigma^{\topp}(X_{w_0t^\mu}(b))$, its stabilizer $\Stab_{Z}(J_b(F))$ is associated to $\Stab_{Z'}(J_{b'}(F'))$ for some $Z'\in\Sigma^{\topp}(X_{w_0't^{\mu'}}(b'))$. By assumption, $\Stab_{Z'}(J_{b'}(F'))$ is a very special parahoric subgroup of $J_{b'}(F')$. By the equivalence (1) $ \Leftrightarrow $ (3) in Proposition \ref{prop: eq. v. special, max vol, max log vol} and by the formula (\ref{eqn: log vol length}) for the log-volume, we know that $\Stab_{Z}(J_b(F))$ is a very special parahoric subgroup of $J_b(F)$. 
		
		Now the reduction from the adjoint  unramified case to the adjoint unramified $F$-simple case follows from the fact that any adjoint unramified group over $F$ is a direct product of adjoint unramified $F$-simple groups. 	\end{proof}

	\subsection{Reduction to the basic case}\label{subsec: reduction to basic} 
		We assume that $\mathrm{char}(F)=0$ and that $G$ is an adjoint $F$-simple  unramified group over $F$. By Corollary \ref{cor:red1}, we can reduce the proof of Theorem \ref{thm:main} to this case. In this subsection we show that we can further reduce the proof to the case where $b$ is basic. We follow the strategy of \cite[\S5]{HV}.
	\subsubsection{}\label{subsubsec:notation for M}Let $\breve K = \breve \BS_0$, and let $\CK$ and $\breve \CK$ be the corresponding parahoric subgroups of $G(F)$ and $G(\breve F)$ respectively, as in \S \ref{subsubsec:cases we consider}. In our current setting, $\CK$ is in fact a hyperspecial subgroup of $G(F)$. 
	
	Let $M\subset G$ denote the standard Levi subgroup of $G$ given by the centralizer of $\overline\nu_b^G$. We view $\breve \CA$ as an apartment for $M$ and let $\breve \fka_M \subset \breve \CA$ be the (unique) alcove with respect to $M$ such that $\breve \fka \subset \breve \fka_M$. We denote by $\breve W_M$ the Iwahori--Weyl group for $M$ and denote by $\Omega_M$ the subgroup of length zero elements determined by $\breve\fka_M$. Upon replacing $b$ by an element of its $\s$-conjugacy class in $G(\breve F)$, we may assume that $b\in M(\breve F)$ and that $\overline \nu_b^M=\overline \nu_b^G$ (see e.g.~\cite[Lemma 2.5.1]{CKV}). Then $b$ is basic in $M$. Upon further  replacing $b$ by an element of its $\s$-conjugacy class in $M(\breve F)$, we may assume that $b=\dot{\tau}$ for some $\tau\in\Omega_M$.
	
	Let $P$ be the standard parabolic subgroup of $G$ with Levi subgroup $M$. Let $N$ be the unipotent radical of $P$. Let $\breve \CK_M$ (resp.~$\breve \CK_P$) denote the intersection $M(\breve F)\cap\breve \CK$ (resp $P(\breve F)\cap \breve \CK$). These arise from group schemes $\CK_M$ and $\CK_P$ defined over $\CO_F$, and $\CK_M(\CO_F)$ is a hyperspecial subgroup of $M(F)$. As in \cite[\S5]{HV}, we define 
	\begin{align*} 
	X^{M\subset G}_\mu(b)(\kk)& :=\{g\in M(\breve F)/\breve\CK_M\mid g^{-1}b\sigma(g)\in \breve \CK\dot{t}^\mu\breve\CK\}, \\
	X^{P\subset G}_\mu(b)(\kk)& :=\{g\in P(\breve F)/\breve \CK_P\mid g^{-1}b\sigma(g)\in \breve \CK\dot{t}^\mu\breve\CK\}.
	\end{align*} 

	These can be identified with the sets of $\kk$-points of perfect subschemes $X^{M\subset G}_\mu(b)$ and $X^{P\subset G}_\mu(b)$ of $\mathrm{Gr}_{\breve\CK_M}$ and $\mathrm{Gr}_{\breve\CK_P}$ respectively. 
	
	The natural maps $M\leftarrow P\rightarrow G$ induce maps $$\xymatrix{X^{M\subset G}_\mu(b)& \ar[l]_p X^{P\subset G}_\mu(b) \ar[r]^q & X_\mu(b)},$$ which are easily seen to be $J_b(F)$-equivariant. The same argument as \cite[Lemma 2.2]{Ham} and the paragraph preceding it shows that the map $q$ is a decomposition of $X_\mu(b)$ into locally closed subschemes (cf.~\cite[Lemma 5.2]{HV}) and hence we obtain a $J_b(F)$-equivariant bijection \begin{equation}\label{eqn: bij irred comp ADLV parabolic}
	\Sigma^{\topp}(X^{P\subset G}_\mu(b))\isoarrow \Sigma^{\topp}(X_\mu(b)).
	\end{equation}
	
	\subsubsection{}Let $X$ and $\CX$ be smooth finite-type affine group schemes over $\breve F$ and $\CO_{\breve F}$ respectively. The loop group $LX$ and the positive loop group $L\CX$ are defined to be the functors on perfect $\kk$-algebras $R$ given by $$LX(R)=X(W(R) \otimes_{W(\kk)} \breve F ),\quad  \text{ and } \quad  L^+\CX(R)=\CX(W(R) \otimes_{W(\kk)}  \CO_{\breve F} ). $$ Then $LX$ is representable by an ind-perfect ind-group-scheme, and   $L^+\CX$ is representable by the perfection of an affine group scheme over $\kk$. We also define the $n^{\text{th}}$ jet-group $L^n \CX$ to be the functor on perfect $\kk$-algebras $R$ given by $$L^n\CX (R) =\CX (W (R)\otimes_{W(\kk)}\CO_{\breve F}/(\pi^n)),$$
	where $\pi$ is a uniformizer in $\breve F$. Then $L^n\CX$ is representable by the perfection of an algebraic group over $\kk$. 
		
	\begin{lemma}\label{lem: Lang map iso}
		The map $$f_b: LN \to  LN, \qquad n\mapsto n^{-1}b\sigma(n) b^{-1}$$ is an isomorphism.
	\end{lemma}
	
	\begin{proof}Recall we have assumed $b=\dot{\tau}$ for $ \tau \in \Omega_M$. Choose $s$ sufficiently divisible such that $\tau\sigma(\tau)\dotsc\sigma^{s-1}(\tau)=t^{\lambda_s}$ where $\lambda_s:=s\overline{\nu}_b\in X_*(T)^+$. (Note that since we have assumed $G$ is unramified, $\Gamma_0$ acts trivially on $X_*(T)$.) We set $b_s:=b\sigma(b) \cdots \sigma^{s-1}\sigma(b)$. Then we have $b_s\in \dot{t}^{\lambda_s}T(\breve F)_1$ and it suffices to show that the map $$f_b^s=f_b\circ\dotsc\circ f_b:n\mapsto n^{-1}b_s\sigma^s(n)b_s^{-1}$$ is an isomorphism $LN \to LN$.
		
		For $r\geq 0$, we define $N_r:=N(F)\cap\CI_r$ where $\CI_r$ is the $r^{\mathrm{th}}$-subgroup in the Moy--Prasad filtration of $\CI$. Then $N_r = \CN_r(\CO_F)$ for an $\CO_{ F}$-group scheme $\CN_r$ and we have $$L^+\CN_r=\ker(L^+\CN_0\rightarrow L^r\CN_0).$$

		Since $\lambda_s\in X_*(T)^+$, we have $\dot{t}^{\lambda_s}\sigma^s(N_r)\dot{t}^{-\lambda_s}\subset N_r$ for all $r$. It follows that $f^s_b$ induces a morphism $$f^s_{b,r}:L^r\CN_0 \to  L^r\CN_0$$
		for each $r$.
		In fact $f^s_{b,r}$ is naturally defined before taking perfections and is an \'etale morphism since it induces multiplication by $-1$ on tangent spaces. It follows that $f^s_{b,r}$ is an \'etale covering.
		
		Let $F_s$ be the degree $s$ unramified extension of $F$ and let $J^{(s)}_{b_s}(F_s)$ denote the $\sigma$-centralizer group $J^{(s)}_{b_s}(F_s):=\{g\in G(\breve F):g^{-1}b_s\sigma^s(g)=b_s\}$; then $J^{(s)}_{b_s}(F_s)\subset M(\breve F)$.  Since the fibers of $f^s_b:LN\rightarrow LN$ are torsors for $J^{(s)}_{b_s}(F_s)\cap N(\breve F)=\{1\}$, 
		it follows that the  ``pro-\'etale'' covering  $f^s_{b}|_{L^+\CN_0}:L^+\CN_0\rightarrow L^+\CN_0$ obtained by taking the inverse limit of the $f_{b,r}^s$ is trivial and hence $f^s_{b}|_{L^+\CN_0}$ is an isomorphism (cf. \cite[Lemma 4.3.4]{XZ}).
		
		Now fix an element $\chi\in X_*(T)^{+,\sigma}\cap X_*(Z_M)$, where $Z_M$ is the center of $M$. Using the fact that $\mathrm{Ad}\dot t^{\chi}\circ f^s_b=f^s_b\circ\mathrm{Ad}\dot{t}^\chi$,  we find that $f^s_b:\dot t^{-\chi}L^+\CN_0\dot t^\chi \rightarrow\dot t^{-\chi}L^+\CN_0\dot t^\chi$ is an isomorphism. Taking an inductive limit over $\chi$, we find that  $f^s_b:LN\rightarrow LN$ is an isomorphism. 
	\end{proof}

	\subsubsection{}We identify $\mathrm{Gr}_{\breve\CK}$ with the fpqc quotient $LG/L^+\breve\CK$.
	For $\lambda\in X_*(T)$, recall the semi-infinite orbit $$S_{N,\lambda}:=LN\dot{t}^\lambda L^+\breve\CK/L^+\breve\CK\subset\mathrm{Gr}_{\breve\CK}.$$
	We let $\mathrm{Gr}_{\breve\CK,\mu}$ denote the Schubert cell $L^+\breve{\CK}\dot{t}^\mu L^+\breve\CK/L^+\breve\CK$ and $\mathrm{Gr}_{\breve\CK,\preccurlyeq\mu}$ the corresponding Schubert variety which is defined to be the closure of $\mathrm{Gr}_{\breve\CK,\mu}$ inside $\mathrm{Gr}_{\breve\CK}$. 
	
	Let $\widehat{M}\subset \widehat{G}$ denote the Levi subgroup determined by $M$ and the fixed pinning from \S\ref{subsec: pinning dual group}.
	For an $M$-dominant element $\lambda\in X_*(T)$, we may consider $\lambda$ as element of $X^*(\widehat{T})$ which is $\widehat{M}$-dominant with respect to the ordering determined by $\widehat{B}\cap\widehat{M}$. We write $V_\lambda^{\widehat{M}}$ for the irreducible representation of $\widehat{M}$ of highest weight $\lambda$.

	We let $a_{\lambda,\mu}$ denote the multiplicity of $V_\lambda^{\widehat{M}}$ appearing in the $\widehat{M}$-representation $V_\mu|_{\widehat{M}}$, and we write $\rho_M$ (resp. $\rho_N$) for the half sum of positive roots in $M$ (resp.~roots in $N$). The same argument as 
	\cite[Proposition 5.4.2]{GHKR} shows that  $$\dim S_{N,\lambda}\cap\mathrm{Gr}_{\breve\CK,\mu}=\langle \mu+\lambda,\rho\rangle -2\langle \lambda,\rho_M\rangle, $$ and that we have $|\Sigma^{\topp}(S_{N,\lambda}\cap\mathrm{Gr}_{\breve{\CK},\mu})|=a_{\lambda,\mu}$.
	
	\begin{lemma}\label{lem: semi-infinite components}
		Let $k_M\in L^+\breve \CK_M(\kk)$ be an element such that $\dot{t}^{-\lambda}k_M \dot{t}^\lambda\in  L^+\breve\CK_M$.  Then left multiplication by $k_M$ induces an automorphism of $S_{N,\lambda}\cap\mathrm{Gr}_{\breve\CK,\mu}$, and we have $k_M(Z)=Z$ for all $Z\in \Sigma^{\topp}(S_{N,\lambda}\cap\mathrm{Gr}_{\breve\CK,\mu})$.
	\end{lemma}
	\begin{proof}
		Let $n\in LN(R)$ where $R$ is a $\kk$-algebra. Then $$k_M n\dot{t}_\lambda=(k_M nk_M^{-1})k_M \dot{t}^\lambda=(k_M nk_M^{-1})\dot{t}^{\lambda} (\dot{t}^{-\lambda}k_M\dot{t}^\lambda)\in LN\dot{t}^\lambda L^+\breve \CK.$$
		It follows that multiplication by $k_M$ induces an automorphism of $S_{N,\lambda}$ with inverse given by multiplication by $k_M^{-1}$, and hence an automorphism of $S_{N,\lambda}\cap\mathrm{Gr}_{\breve\CK,\mu}$. 
		
		The group $\dot{t}^{-\lambda} \breve \CK_M\dot{t}^\lambda\cap \breve\CK_M$ arises as the $\CO_{\breve F}$-points of  a smooth connected $\CO_{\breve  F}$-scheme $\breve \CK_{\lambda}$. 
		Then as above, left multiplication induces a map $$L^+\breve\CK_\lambda\times  (S_{N,\lambda}\cap\mathrm{Gr}_{\breve\CK,\mu})\rightarrow  S_{N,\lambda}\cap\mathrm{Gr}_{\breve\CK,\mu}.$$ Let $Z\in \Sigma^{\topp}(S_{N,\lambda}\cap\mathrm{Gr}_{\breve\CK,\mu})$. Then $k_M(Z)\in \Sigma^{\topp}(S_{N,\lambda}\cap\mathrm{Gr}_{\breve\CK,\mu})$ is contained in the image of  $L^+\breve \CK_\lambda\times Z\rightarrow  S_{N,\lambda}\cap\mathrm{Gr}_{\breve\CK,\mu}$. The image of this map is an irreducible subscheme of $S_{N,\lambda}\cap\mathrm{Gr}_{\breve\CK,\mu}$ containing $Z$, hence is equal to $Z$. It follows that $k_M(Z)=Z$.
	\end{proof}

	\subsubsection{}We define the sets \begin{align*}
I_{\mu,M} & :=\{\lambda\in X_*(T)\mid\lambda\in X_*(T) \text{ is $M$-dominant}, S_{N,\lambda}\cap\mathrm{Gr}_{\breve\CK,\mu} \neq\emptyset\}, \\ 
I_{\mu,b,M} & :=\{\lambda\in I_{\mu,M}\mid\kappa_M(b)=\lambda^\natural\in\pi_1(M)_{\Gamma}\}.
	\end{align*} 
	Then there is a decomposition \begin{equation}\label{eqn: decomp ADLV for M}X^{M\subset G}_\mu(b)=\coprod_{\lambda\in I_{\mu,b,M}}X_{\lambda}^M(b),
	\end{equation} where each $X_{\lambda}^M(b)$ is locally closed inside $X_{\mu}^{M\subset G}(b)$. 

	\begin{proposition}\label{lem: dimension Levi ADLV}
		\begin{enumerate}\item Let $\lambda\in I_{\mu,b,M}$ and $Z\in\Sigma^{\topp}(X^M_\lambda(b))$.
			Then $$\dim p^{-1}(Z)\leq\dim X_\mu(b)$$ with equality  if and only if $a_{\lambda,\mu}\neq 0$.

			\item Let  $U\in \Sigma^{\topp}(X_\mu(b))$ and $U^P\in \Sigma^{\topp}(X_\mu^{P\subset G}(b))$ the corresponding element. Then there exists $\lambda\in I_{\mu,b,M}$ with $a_{\lambda,\mu}\neq 0$ and $Z\in\Sigma^{\topp}(X^{M}_\lambda(b))$ such that $Z\cap p(U_P)$ is open dense in $p(U^P)$.
			
		\end{enumerate}
	\end{proposition}
	
	\begin{proof} (1) By \cite[Lemma 2.8, Proposition 2.9 (3)]{Ham}, which also holds in the mixed  characteristic setting, we have \begin{align*}\dim p^{-1}(Z)&\leq\dim X_\lambda^M(b)+\langle \mu+\lambda,\rho\rangle-2\langle\lambda,\rho_M\rangle -2\langle\overline{\nu}_b,\rho_N\rangle\\
		&=\langle \lambda,\rho_M\rangle-\frac{1}{2}\mathrm{def}_M(b)+\langle \mu+\lambda,\rho\rangle-2\langle\lambda,\rho_M\rangle -2\langle\overline{\nu}_b,\rho_N\rangle\\
		&=\langle\mu-\overline\nu_b,\rho\rangle-\frac{1}{2}\mathrm{def}_G(b)\\
		&=\dim X_\mu(b).
		\end{align*}
		
		The first and third equalities follow Theorem \ref{thm:dim-adlv}
		and the second equality follows from the identities $\mathrm{def}_G(b)=\mathrm{def}_M(b)$ and $\langle \overline\nu_b,\rho_N\rangle=\langle \lambda,\rho_N\rangle$. By \cite[Proposition 5.4.2]{GHKR}, which again holds in mixed characteristic, the first inequality is an equality if and only if $a_{\lambda,\mu}\neq 0$.

		(2) By \cite[Proposition 2.9 (2)]{Ham} and a similar calculation as in (1),  for any  $\lambda\in I_{\mu,b,M}$ and $x\in X^{M}_\lambda(b)$, we have $$\dim p^{-1}(x)\leq\dim X_\mu(b)-\dim X^M _\lambda(b)$$ with equality if and only if $a_{\lambda,\mu}\neq 0$. 
		
		Since the $X_{\lambda}^M(b)$ are locally closed inside $X_\mu^{M\subset G}(b)$, there exists a unique $\lambda\in I_{\mu,b,M}$ such that $p(U^P)\cap X^M_{\lambda}(b)$ is open dense in $p(U^P)$. Since $p(U^P)$ is irreducible, we can further find a $Z\in \Sigma^{\topp}(X_\lambda^M(b))$ such that $p(U^P)\cap Z$ is open dense in $p(U^P)$. Then we have $$\dim  p(U^P)\geq \dim X_{\lambda}^M(b).$$ It follows that these quantities are equal  and we have $a_{\lambda,\mu}\neq 0$. 
	\end{proof}
	
	\begin{proposition}\label{prop: J_b action on fibers}  Let $\lambda\in I_{\mu,b,M}$ with $a_{\lambda,\mu}\neq 0$  and let $Z\in\Sigma^{\topp}(X_\lambda^M(b))$. Then $\Stab_Z(J_b(F))$ acts trivially on $\Sigma^{\topp}(p^{-1}(Z))$. 
	\end{proposition}
	\begin{proof}

		Let $Y$ be an open subscheme of $Z$ and let $Y'$ be an \'etale cover of $Y$  such that the inclusion map $Y\to  X_\lambda^M(b)$ lifts to a map $\iota: Y' \rightarrow LM$. The existence of $Y'$ follows from the same argument as \cite[Theorem 1.4]{PR}. Upon replacing $Y'$ with an irreducible component, we may assume that $Y'$ is also irreducible.
		
		We write $p^{-1}(Y')$ for the fiber product 
		\[\xymatrix{p^{-1}(Y')\ar[r]\ar[d]& X^{P\subset G}_\mu(b)\ar[d]\\
			Y'\ar[r]& X^{M\subset P}_\mu(b).}\]
		The natural map $p^{-1}(Y')\rightarrow p^{-1}(Z)$ induces a bijection $\Sigma^{\topp}(p^{-1}(Y'))\cong \Sigma^{\topp}(p^{-1}(Z))$.

		As in \cite[Proposition 5.6]{HV}, we set $$\Phi:=\{(m,n)\in \iota(Y')\times LN \mid mnL^+\breve\CK_P\in X_\mu^{P\subset G} \}\subset LM\times LN.$$
		Then the natural map $\gamma:\Phi\rightarrow p^{-1}(Y')$ is a $\breve \CK_N$-torsor. Moreover upon shrinking $Y$ and replacing $\iota$ if necessary (cf. \cite[Proof of Proposition 5.6]{HV}) we may assume $m^{-1}b\sigma(m)\in \dot{t}^\lambda\breve\CK_M$ for any $m\in \iota(Y')$. We then define $$\mathscr{E}:=\iota(Y')\times (LN \cap L^+\breve\CK\dot{t}^\mu L^+\breve\CK\dot{t}^{-\lambda}) \subset LM\times LN.$$
		
		We write $\mathrm{Ad}_M:LM\times LN\rightarrow LM\times LN$ for the map $(m,n)\mapsto(m,mnm^{-1})$. This is easily seen to be an isomorphism with inverse given by $\mathrm{Ad}_M^{-1}:(m,n)\mapsto (m,m^{-1}nm)$. Define $\tilde{f}_b=\mathrm{Ad}_M \i \circ (\id, f_b) \circ \mathrm{Ad}_M: LM\times LN\rightarrow LM\times LN$. By Lemma \ref{lem: Lang map iso}, $\tilde{f}_b$ is an isomorphism. The restriction of $\tilde{f}_b$ to $\Phi$ gives an isomorphism $\tilde{f}_b: \Phi\rightarrow \mathscr E$ and we have a Cartesian diagram:
		\[\xymatrix{ \Phi\ar[r]^{\tilde{f}_b} \ar[d]& \mathscr E\ar[d] \\
			\iota(Y')\times LN\ar[r]^{\tilde f_b }	& 
			\iota(Y')\times LN}\]
		
		We consider the projection $$\mathrm{pr}_\lambda:LN\rightarrow LN\dot{t}^\lambda L^+\breve{\CK}/L^+\breve \CK$$given by $n\mapsto n\dot{t}^{\lambda}\breve\CK$.
		Then $LN \cap L^+\breve\CK\dot{t}^\mu L^+\breve\CK\dot{t}^{-\lambda}$ is the preimage of $S_{N,\lambda}\cap\mathrm{Gr}_{\breve\CK,\mu}$ under $\mathrm{pr}_{\lambda}$. 
		
		We write $\mathrm{pr}:\mathscr{E}\rightarrow S_{N,\lambda}\cap\mathrm{Gr}_{\breve\CK,\mu}$ for the composition of projection onto the second component $\mathrm{pr}_2:\mathscr{E}\rightarrow LN \cap L^+\breve\CK\dot{t}^\mu L^+\breve\CK\dot{t}^{-\lambda}$ followed by $\mathrm{pr}_\lambda$. Let $Z'\in \Sigma^{\topp}(S_{N,\lambda}\cap\mathrm{Gr}_{\breve \CK,\mu})$. The same argument as \cite[Proposition 2.9]{Ham} shows that for  $Z'\in \Sigma(S_{N,\lambda}\cap\mathrm{Gr}_{\breve \CK,\mu})$, we have $$\dim\gamma((\mathrm{pr}\circ \tilde f_b)^{-1}(Z'))\leq\dim X_\mu(b)=\dim p^{-1}(Y')$$ with equality if and only if $Z'\in \Sigma^{\topp}(S_{N,\lambda}\cap\mathrm{Gr}_{\breve \CK,\mu})$. Thus the association $Z'\mapsto \gamma((\mathrm{pr}\circ \tilde f_b)^{-1}(Z'))$ induces a map $\theta:\Sigma^{\topp}(S_{N,\lambda}\cap\mathrm{Gr}_{\breve \CK,\mu})\rightarrow\Sigma^{\topp}(p^{-1}(Y'))$.  Since the maps $\gamma$, $\tilde f_b$, and $\mathrm{pr}$ all induce bijections on irreducible components, it follows that $\theta$ is a  bijection.
		
		Let $U\in\Sigma^{\topp}(p^{-1}(Z))$ and let $U_1,U_2\subset \Phi$ denote the preimages of $U$ and $jU$ respectively.  For $i=1,2$, let $Z_i'\in \Sigma^{\topp}(S_{N,\lambda}\cap\mathrm{Gr}_{\breve \CK,\mu})$ be  the unique component containing 
		$\mathrm{pr}\circ \tilde f_b(U_i)$. Then it suffices to show that $Z'_1=Z'_2$.
		
		Let $x=(m,n)\in U_1(\kk)$ such that the image $y=m\breve \CK_M$ of $x$ in $Z$ lies in $j^{-1}Y$. Note that the set of such $x$ is dense in $U_1$. Then $jy$ lies in $Y$ and we let $y'\in Y'(\kk)$ be a lift of $jy$ to $Y'(\kk)$. Then the element  $\iota(y')\in LM(\kk)$  is of the form $jmk_M$ for some $k_M \in L^+\breve \CK_M(\kk)$. 
		
		Consider the element $z=(jmk_M,k_M^{-1}nk_M)\in \Phi$. Then we have $z\in U_2(\kk)$, and one computes that $$\mathrm{pr_2}(\tilde{f}_b(z))=k_M^{-1}n^{-1}b_m\sigma(n)b_m^{-1}k_M=k_M^{-1}\mathrm{pr_2}(\tilde{f}_b(x))k_M,$$ where $b_m=m^{-1}b\sigma(m)$. By the  assumption on $\iota$, we have $b_m, k_M^{-1}b_mk_M\in \dot{t}^\lambda \breve L^+\CK_M$, and hence $\dot{t}^{-\lambda} k_M\dot{t}^\lambda\in L^+\breve\CK_M$. 
		Then by Lemma \ref{lem: semi-infinite components}, we have $\mathrm{pr}\circ \tilde f_b(x)\in Z'_2$. Since this is true for a dense set of $x$ in $U_1$, it  follows that $\mathrm{pr}\circ \tilde f_b(U_1)\subset Z_2'$, and hence $Z_1'=Z_2'$.
		
	\end{proof}
	
	\begin{corollary}\label{cor: reduction to basic}
		Let $U\in\Sigma^{\topp}(X_\mu(b))$. Then there exists $\lambda\in I_{\mu,b,M}$ and  $Z\in \Sigma^{\topp}(X^M_\lambda(b))$ such that $$\Stab_U(J_b(F))=\Stab_{Z}(J_b(F)).$$
	\end{corollary}
	\begin{proof}Let $U^P\in \Sigma^{\topp}(X_\mu^{P\subset G}(b))$ be the component corresponding to $U$ and let $Z:=p(U^P)\subset X_\mu^{M\subset G}(b)$.  By Lemma \ref{lem: dimension Levi ADLV}, we have  $Z\in \Sigma^{\topp} (X_\lambda^{M}(b))$  for some $\lambda\in I_{\mu,b,M}$ with $a_{\lambda,\mu}\neq 0$. 
		
		By the $J_b(F)$-equivariance of $p$, we have $$\Stab_{U^P}(J_b(F))\subset \Stab_{Z}(J_b(F)).$$ Since $U^P\in \Sigma^{\topp}(p^{-1}(Z))$, Proposition \ref{prop: J_b action on fibers} implies $$\Stab_U(J_b(F))=\Stab_{U^P}(J_b(F))=\Stab_{Z}(J_b(F)). $$ The statement is proved.
	\end{proof}

	\begin{proposition}\label{prop: final reduction}
		In order to prove Theorem \ref{thm:main}, it suffices to prove it when $\mathrm{char}(F)=0$, $G$ is $F$-simple, adjoint, and unramified over $F$, and $b$ is basic. 
	\end{proposition}
\begin{proof} This follows from Corollary \ref{cor:red1} and Corollary \ref{cor: reduction to basic}. 
\end{proof}

	\subsection{The special case of a sum of dominant minuscule cocharacters}\label{subsec:special case} 	We assume that $\mathrm{char}(F)=0$, that $G$ is $F$-simple, adjoint, and unramified over $F$, and that $b$ is basic. Our goal in this subsection is to prove a partial result towards Theorem \ref{thm:main} when $\mu$ is a sum of minuscule dominant cocharacters. We use the idea of X.~Zhu (see \cite[\S 3.1.3]{Zhu}) that one can ``separate'' the summands of $\mu$ by constructing a convolution map from the affine Deligne--Lusztig variety of a Weil-restriction group to the original affine Deligne--Lusztig variety. This idea was originally used in \textit{loc.~cit.~}to establish the dimension formula, and it was S.~Nie who first applied this idea to the study of irreducible components (see \cite{nie} and \cite{nienew}).

	\subsubsection{}Let $F_r$ denote the unramified extension of $F$ of degree $r$ inside $\breve F$. Let $H$ be an unramified reductive group over $F_r$ and let $G':=\mathrm{Res}_{F_r/F}H$. We canonically identify $\breve F$ with $\breve F_r$. For $b \in H(\breve F)$ and $\mu$ a geometric cocharacter of $H$, we have the affine Deligne--Lusztig variety write $X^H_\mu(b)$ as in \S \ref{subsubsec:cases we consider}. In this subsection we denote  this  by $X^H_\mu(b\sigma^r)$  to emphasize that $H$ is a group over $F_r$ and the Frobenius relative to $F_r$ is $\sigma^r$. We also write $J_b^{(r)}$ for $J_b$ (defined with respect to $H$ over $F_r$), and write $B^{(r)}(H)$ for the set of $\sigma^r$-conjugacy classes in $H(\breve F)$.

	Let $\tau_0: F_r \hookrightarrow \breve F$ be the inclusion and write $\tau_i$ for $\sigma^i(\tau_0)$ for $i=1,\dotsc, r-1$. Thus $\{\tau_0,\dotsc, \tau_{r-1}\}$ is the set of $F$-algebra embeddings $F_r\rightarrow \breve F$. There is a canonical identification 
	$$G' \otimes_F \breve F\cong \prod_{ i = 0}^{r-1} H\otimes_{F_r,\tau_i}\breve F.$$  Let $T_H$ be the centralizer of a fixed maximal $ F_r$-split torus in $H$. Let $T' = \mathrm{Res}_{F_r/F} T_H$, which we view as an $F$-subgroup of $G'$. Then $T'$ is the centralizer of a maximal $F$-split torus in $G'$. A cocharacter of $T'$ is the same as a sequence $\mu'=(\mu_0,\dotsc,\mu_{r-1})$, where $\mu_i \in X_*(T_H)$. Fix a Borel subgroup of $H$ containing $T_H$ and use it to define the dominant cocharacters $X_*(T_H)^+$. This also defines a Borel subgroup of $G'$ containing $T'$ and defines $X_*(T')^+$. We fix a hyperspecial subgroup of  $ H(F_r)$ that is compatible with our  choice of the maximal $\breve F_r$-split $F_r$-rational torus of $H$. This also determines a hyperspecial subgroup of $G'(F)$. We use these hyperspecial subgroups to define affine Deligne--Lusztig varieties at hyperspecial level for $H$ and $G'$. For $b'=(b_0,\dotsc,b_r)\in G'(\breve F)$, we define  $$ \mathrm{Nm}(b'):=b_0\sigma(b_1)\cdots\sigma^{r-1}(b_{i-1})\in  H (\breve F).$$ The association $b' \mapsto \mathrm{Nm}(b')$ defines a bijection $B(G')\xrightarrow{\sim} B^{(r)}(H)$, and there is a natural isomorphism $J_{b'}(F)\cong J_{\mathrm{Nm}(b')}^{(r)}(F_r)$. 
	
	\begin{lemma}\label{lem: reduction to minuscule special} Let
$\mu'=(\mu_0,\dotsc,\mu_{r-1})\in X_*(T')^+$ and $[b']\in B(G',\mu')$. We write $|\mu'|$ for $\sum_{i=0}^{r-1}\s^i (\mu_i) \in X_*(T_H)^+$. Then there is a natural morphism $\theta: X_{\preccurlyeq\mu'}^{G'}(b')	\rightarrow X^H_{\preccurlyeq|\mu'|}(\mathrm{Nm}(b')\sigma^r)$ which is $J_{b'}(F)\cong J_{\mathrm{Nm}(b')}^{(r)}(F_r)$-equivariant. Moreover, for each $U\in \Sigma^{\topp}(X^H_{\preccurlyeq|\mu'|}(\mathrm{Nm}(b') \sigma^r))$, there exists  $Z\in \Sigma^{\topp}(X_{\preccurlyeq\mu'}^{G'}(b'))$ such that $$\Stab_{Z}(J_{b'}(F))=\Stab_{U}(J_{\mathrm{Nm}(b')}^{(r)}(F_r)).$$
	\end{lemma}
	\begin{proof}
		The morphism $\theta$ is given by the isomorphism in \cite[Lemma 3.5]{Zhu} and the left vertical map in the diagram on p.~459 of \cite{Zhu}. The $J_{b'}(F)\cong J_{\mathrm{Nm}(b')}^{(r)}(F_r)$-equivariance is clear from the construction. Let $U\in \Sigma^{\topp}(X_{\preccurlyeq|\mu'|}^H(\mathrm{Nm}(b')\sigma^r) )$.  We claim that $\mathcal J : = \Stab_{U}(J_{\mathrm{Nm}(b')}(F_r))$ acts trivially on  $\Sigma^{\topp}(\theta^{-1}(U))$. In fact, by the diagram on p.~459 of \cite{Zhu}, there exists $m \in \BN$ and an $L^m H$-torsor $U'$ over $U$ equipped with a $\mathcal J$-action such that $U' \to U$ is $\mathcal J$-equivariant and such that there exists a $\mathcal J$-equivariant $U'$-scheme isomorphism $\theta^{-1}(U) \times _U U' \isoarrow F \times_{\kk} U'$, where $\mathcal J$ acts trivially on $F$. Our claim follows. By the claim, we have $\Stab_{Z}(J_{b'}(F))=\Stab_{U}(J_{\mathrm{Nm}(b')}^{(r)}(F_r))$ for arbitrary $Z \in \Sigma^{\topp} (\theta^{-1}(U))$. By \cite[Lemma 1.8]{nie}, we have  $\Sigma^{\topp} (\theta^{-1}(U)) \subset  \Sigma^{\topp} (X^{G'}_{\preccurlyeq\mu'}(b'))$. The lemma follows. 
	\end{proof}

	\begin{proposition}\label{thm: sum of minuscule case}Assume that $\mu$ is a sum of dominant minuscule cocharacters and $[b]\in B(G,\mu)$ is basic. Then for any $Z\in \Sigma^{\topp}(X_\mu(b))$, $\Stab_Z({J_b(F)})$ is a special parahoric subgroup of $J_b(F)$.
	\end{proposition}
	
	\begin{proof}
		We first consider the case where $\mu$ is minuscule. Let $M\subset G$ be a standard Levi subgroup such that there exists $b\in[b]\cap M(\breve F)$ which is superbasic in $M$. We use the same notations as in \S \ref{subsubsec:notation for M} with respect to $M$. We choose $b \in [b] \cap M(\breve F)$ that is superbasic in $M$, and upon $\s$-conjugating $b$ in $M(\breve F)$ we may assume that $b=\dot{\tau}$ for some $\tau\in \Omega_M$.
		
		Let $Z\in \Sigma^{\mathrm{\topp}}(X_\mu(b))$ and we let $\mathcal{J}\subset J_b(F)$ denote the stabilizer of $Z$. Let $P$ be the standard parabolic subgroup of $G$ with Levi factor $M$. By \cite[Theorem 3.1.1]{ZZ}, $\mathcal{J}$ is a parahoric subgroup of $J_b(F)$.  By Theorem \ref{thm:ZZA} and the ``only if'' part of \cite[Theorem 5.12]{HV}, $J_b(F)\cap P(\breve F)$ acts transitively on each $J_b(F)$-orbit in $\Sigma^{\topp}(X_\mu(b))$. Hence we have
		\begin{align}\label{eq:Iwasawa}
		J_b(F)=(J_b(F)\cap P(\breve F))\cdot \mathcal J.\end{align} 			Note that $J_b(F)\cap P(\breve F) = Q(F)$ where $Q$ is a  minimal parabolic subgroup of $J_b$, since $b$ is superbasic in $M(\breve F)$. Recall that $\CJ$ is a parahoric subgroup of $J_b(F)$. Thus by \cite[Proposition 4.4.2]{BT1}, the equality (\ref{eq:Iwasawa}) implies that $\mathcal J$ is contained in a  special parahoric subgroup $\CJ_1$ of $J_b(F)$. Note that by (\ref{eq:Iwasawa}), there exists $j \in Q(F)$ such that $j \CJ j^{-1}$ is associated with a facet in the standard apartment $\CA$. Thus up to replacing $Z$ by $jZ$, we may assume that both $\CJ$ and $\CJ_1$ are associated with facets in $\CA$. Then from (\ref{eq:Iwasawa}) we get 
	\begin{align}\label{eq:to apply Bruhat}
\CJ_1 = (\CJ_1 \cap Q(F)) \cdot \CJ. 
	\end{align}
 Let $\overline {\CJ_1}$ denote the reductive quotient of the special fiber of $\CJ_1$. Then the images of $\CJ_1 \cap Q(F)$ and $\CJ$ in $\overline {\CJ_1}$ are ($k_F$-points of) a Borel subgroup $\mathbb B$ and a parabolic subgroup $\mathbb P$ respectively, and $\mathbb B \cap \mathbb P$ contains a maximal torus in $\overline {\CJ_1}$. By (\ref{eq:to apply Bruhat}) we have $\overline{\CJ_1} = \mathbb B \mathbb P$, and by the Bruhat decomposition this is possible only when $\mathbb P = \overline {\CJ_1}$, or equivalently $\CJ = \CJ_1$. We have thus proved that $\CJ$ is a special parahoric subgroup of $J_b(F)$.

		We now consider the case when $\mu$ is a sum of $r$ dominant minuscule cocharacters. Let $H$ be the pinned unramified reductive group over $F_r$ such that its based root datum with the $\s^r$-action is identified with the based root datum of $(G,B,T)$ with the $\s$-action. Let $T_H$ be the maximal torus in the pinning of $H$. Then we have a canonical identification $X_*(T)^+ \cong X_*(T_H)^+$, and the image of $\mu$ in $X_*(T_H)^+$, denoted by $\mu_H$, is also a sum of $r$ dominant minuscule cocharacters. We have canonical identifications $G(\breve F) \cong H(\breve F)$ and $(\breve W, \s) \cong (\breve W_H, \s^r)$. Let $b_H \in H(\breve F)$ correspond to $b\in G(\breve F)$, and let $w_{0,H}$ denote the longest element of $\breve W_H$. Then  $(G,b,w_{0}t^{\mu})$ and $(H,b_H,w_{0,H}t^{\mu_H})$ are associated as in \S\ref{subsec: reduction to char 0}. By Proposition \ref{prop: assoc parahorics}, Proposition \ref{prop: relation between Iw and Sp ADLV}, and the fact that association of parahoric subgroups preserves being very special (see the proof of Corollary \ref{cor:red1}), it suffices to prove the result for $X^H_{\mu_H}(b_H\sigma^r)$.
		
	We decompose $\mu_H$ as  $\sum_{i=0}^{r-1}\s^i (\mu_i)  $, and define $G ' = \Res_{F_r/F} H$. Let $\mu' = (\mu_0, \cdots, \mu_{r-1})$. Choose $b'\in G'(\breve F)$ such that its image under $G'(\breve F) \to B(G') \isoarrow B^{(r)}(H)$ is the class of $b_H$. By Lemma \ref{lem: reduction to minuscule special} applied to the current situation, for every $U \in \Sigma^{\topp}(X^H_{\preccurlyeq\mu_H}(b_H\sigma^r)) $, there exists $Z\in \Sigma^{\topp}(X^{G'}_{
			\preccurlyeq \mu'}(b'))$ such that $\Stab_U(J_{b_H} ^{(r)} (F_r)) = \Stab_Z (J_{b'} (F)).$ Note that $\mu'$ is minuscule, so $X^{G'}_{
			\preccurlyeq \mu'}(b')  = X^{G'}_{\mu'}(b')$, and by the previous part of the proof, we know that $\Stab_Z (J_{b'} (F))$ is a special parahoric. The desired result for $X^H_{\mu_H}(b_H\sigma^r)$ follows by noting that the natural map $X^H_{\mu_H}(b_H\sigma^r)\rightarrow X^H_{\preccurlyeq\mu_H}(b_H\sigma^r)$ induces a $J_{b_H}^{(r)}(F_r)$-equivariant bijection between the sets of top-dimensional irreducible components.
	\end{proof}

\subsection{Numerical relations}\label{subsec:numerical} Another key ingredient in our proof of Theorem \ref{thm:main} is a set of  numerical relations deduced from results in \cite{ZZ}, which we discuss here.

\subsubsection{}  \label{subsubsec:setting for numerical} We assume that $\mathrm{char}(F)=0$, that $G$ is $F$-simple, adjoint, and unramified over $F$, and that $b$ is basic. We also assume that $[b]$ is not unramified, i.e., we assume that $\mathrm{def}_G(b) \neq 0$. 

Since $b$ is basic, $J_b$ is an inner form of $G$. Thus we can transfer Haar measures on $G(F)$ to Haar measures on $J_b(F)$, as in \cite[\S 1]{kottTama}. We fix the Haar measure on $G(F)$ giving volume $1$ to hyperspecial subgroups, and transfer it to a Haar measure on $J_b(F)$. (This Haar measure on $J_b(F)$ may not give volume $1$ to Iwahori subgroups.) For each $Z\in \Sigma^{\topp}(X_\mu(b))$, the volume of the parahoric subgroup $\Stab_{Z} (J_b(F)) \subset J_b(F)$ depends on $Z$ only via the $J_b(F)$-orbit $[Z]$ of $Z$. We denote this volume by $\vol([Z])$. 

Let $S_{\mu,b}(t) \in \BQ(t)$ be the rational function in \cite[Theorem 6.1.3]{ZZ}. We have
\begin{align*} S_{\mu, b }(0) & =\mathscr{N}(\mu,b), \\ 
S_{\mu, b } (q) & = e(J_b)\sum _{ [Z]\in J_b(F)\backslash \Sigma^{\topp}(X_\mu(b))} \vol ([Z])^{-1}.
\end{align*}
Here $q$ denotes the cardinality of the residue field of $F$, and $e(J_b) \in \{ \pm 1\}$ is the Kottwitz sign of $J_b$. (Recall $\mathscr N(\mu,b)$ from Definition \ref{defn:N(mu,b)}.) We set 
\begin{align}\label{eq:Q}
Q(\mu,b) : = e(J_b) S_{\mu,b} (q) \mathscr N(\mu,b)^{-1} =  \mathscr N(\mu,b)^{-1} \sum _{ [Z]\in J_b(F)\backslash \Sigma^{\topp}(X_\mu(b))} \vol ([Z])^{-1} .
\end{align}

\begin{proposition}\label{prop:key from ZZ} Keep the assumptions on $F,$ $G,$ and $ [b]$ in \S \ref{subsubsec:setting for numerical}. Assume that none of the simple factors of $G_{\overline F}$ is of type $A$. The following statements hold.  \begin{enumerate}
	\item Assume that $G$ is not a Weil restriction of the split adjoint group of type $E_6$. Then there exists a minuscule $\mu_1 \in X_*(T)^+$ such that $\mathscr N(\mu_1, b ) =1$, and such that for all $\mu \in X_*(T) ^+$ we have
$$Q(\mu,b) = Q(\mu_1, b). $$ 
	\item Assume that $G$ is a Weil restriction of the split adjoint group of type $E_6$. (The Weil restriction is necessarily along an unramified extension of $F$ since $G$ is unramified). Then there exist $\mu_1,\mu_2 \in X_*(T) ^+$, where $\mu_1$ is minuscule and $\mu_2$ is a sum of dominant minuscule cocharacters, such that $\mathscr N(\mu_1, b) = 1$ and such that for all $\mu \in X_*(T) ^+$ we have 
	\begin{align}\label{eq:three mu's}
Q(\mu,b)= Q(\mu_1,b)  + C(\mu) (Q(\mu_2,b) -Q (\mu_1,b)),
	\end{align}
 for some $C(\mu) \in \BQ$. 
\end{enumerate}
\end{proposition}
\begin{proof}The proposition follows from the main result of \cite{ZZ} (i.e., the Chen--Zhu Conjecture), and the proof of \cite[Theorem 6.3.2]{ZZ}. More precisely, part (1) follows from the equation below \cite[(6.3.3)]{ZZ} and the main result \cite[Theorem A]{ZZ} asserting that the numbers $\mathscr M(\mu,b)$ and $\mathscr M(\mu_1,b)$ in that equation are equal to $\mathscr N(\mu,b)$ and $\mathscr N(\mu_1,b)$ respectively . Part (2) follows from the equation below \cite[(6.3.7)]{ZZ}, the equation below \cite[(6.3.8)]{ZZ}, and \cite[Theorem A]{ZZ} asserting that $\mathscr M(\mu,b)  = \mathscr N(\mu,b)$. 
\end{proof}
\begin{remark}In Proposition \ref{prop:key from ZZ}, the conclusion in case (2) is weaker than that in case (1), and this originates from the dichotomy in \cite[Proposition 6.3.2]{ZZ}. It turns out that in case (2), there is extra difficulty in trying to establish the key estimate \cite[(6.3.1)]{ZZ}, and in fact only the weaker statement \cite[Proposition 6.3.2 (2)]{ZZ} is proved. If $G$ is a Weil restriction of $\PGL_n$, there seems to be even more serious difficulty in trying to establish \cite[(6.3.1)]{ZZ}. As a result the type A case is not considered in \cite[Proposition 6.3.2]{ZZ}. After the publishing of \cite{ZZ}, the authors have realized that one can actually prove \cite[(6.3.1)]{ZZ} when $G$ is a Weil restriction of an adjoint unramified unitary group. We will not need this for the purposes of the current paper. 
\end{remark}

\subsection{Proof of Theorem \ref{thm:main}} \label{subsec:proof}
By Proposition \ref{prop: final reduction}, we may assume without loss of generality that $\mathrm{char}(F)=0$, that $G$ is $F$-simple, adjoint, and unramified over $F$, and that $b$ is basic. If $[b]$ is unramified, then Theorem \ref{thm:main} is already proved in \cite[Theorem 4.4.14 (1)]{XZ}, cf.~\cite[Theorem 6.2.2]{ZZ}. We hence assume that $[b]$ is not unramified. Thus we are in the same setting as \S \ref{subsubsec:setting for numerical}.

Let $\vol_{\max}$ be the volume of a very special parahoric subgroup of $J_b(F)$, where the Haar measure on $J_b(F)$ is as in \S \ref{subsubsec:setting for numerical}. We know that every stabilizer for the $J_b(F)$-action on $\Sigma^{\topp}(X_\mu (b))$ is a parahoric subgroup of $J_b(F)$, see Remark \ref{rem:useful} and \cite[Proposition 3.1.4]{ZZ}. As a result, the volume of such a stabilizer will be at most $\vol_{\max}$, and equality holds if and only if the stabilizer is very special. Since the quantity $Q(\mu,b)$  defined in (\ref{eq:Q}) is the average of the volumes of these stabilizers, we see that Theorem \ref{thm:main} for $(\mu,b)$ is equivalent to the relation
\begin{align}\label{eq:suffices}
Q(\mu,b) = \vol_{\max}^{-1}.
\end{align}

Since $G$ is $F$-simple, the simple factors of $G_{\overline F}$ are isomorphic to each other. If they are of type $A$, then $\mu$ is necessarily a sum of dominant minuscule cocharacters in $X_*(T)$. In this case, Theorem \ref{thm:main} follows from Proposition \ref{thm: sum of minuscule case} if we know that every special parahoric subgroup of $J_b(F)$ is automatically very special. Since $J_b$ is an inner form of $G$ and hence also of type $A$, it is indeed the case that special parahoric subgroups of $J_b(F)$ are automatically very special, by inspecting the tables in \cite[\S 4]{Tits}.

Assume that $G$ is as in Proposition \ref{prop:key from ZZ} (1), and let $\mu_1$ be as in that part of that proposition. Since $\mathscr N(\mu_1,b) =1$, it follows from Proposition \ref{prop:one maximal} that $Q(\mu_1, b ) = \vol_{\max}^{-1}$. (Here Proposition \ref{prop:one maximal} is indeed applicable since $G$ is $F$-simple and adjoint.) But $Q(\mu, b) = Q(\mu_1,b)$, so (\ref{eq:suffices}) holds for $(\mu,b)$, and this implies that Theorem \ref{thm:main} holds for $(\mu,b)$. 

We are left with the case where $G$ is a Weil restriction of the split adjoint group of type $E_6$. In this case, let $\mu_1$ and $\mu_2$ be as in Proposition \ref{prop:key from ZZ} (2). Since $J_b$ is also of type $E_6$, by inspecting the tables in \cite[\S 4]{Tits} we see that every special parahoric subgroup of $J_b(F)$ is automatically very special.  Thus by Proposition \ref{thm: sum of minuscule case} we know that Theorem \ref{thm:main} holds for $(\mu_1,b)$ and $(\mu_2,b)$. It follows that $$Q(\mu_1,b) = Q(\mu_2,b) = \vol_{\max}^{-1}. $$ Substituting this back to (\ref{eq:three mu's}), we obtain (\ref{eq:suffices}) for $(\mu,b)$, and this implies that Theorem \ref{thm:main} holds for $(\mu,b)$.

The proof of Theorem \ref{thm:main} is complete.

\section{Irreducible components of basic loci} \label{sec: Shimura}
\subsection{Shimura varieties}

\subsubsection{}We use the previous section to describe the irreducible components in the basic locus of certain Hodge type Shimura varieties constructed in \cite{KP}. Let $\mathbf{G}$ be a connected reductive group over $\mathbb{Q}$ and $X$ a conjugacy class of homomorphisms $$h:\mathbb{S}:=\mathrm{Res}_{\mathbb{C}/\mathbb{R}}\rightarrow \mathbf{G}_{\mathbb{R}}$$ such that $(\mathbf{G},X)$ is a Shimura datum. For any $\mathbb{C}$-algebra $R$ we have $R\otimes_{\mathbb{R}}\mathbb{C}\cong R\times c^*(R)$, where $c$ is the complex conjugation. For $h\in X$ we let $\mu_h$ denote the cocharacter of $\mathbf{G}_{\BC}$ given by $$R^\times\rightarrow R^\times\times c^*(R)^{\times} \xrightarrow{h}\mathbf{G}(R),$$ where $R$ is an arbitrary $\BC$-algebra and the first map is $z \mapsto (z,1)$.  The conjugacy class of $\mu_{h}^{-1}$ is defined over a number field $\bE:=\bE(\bG,X)\subset \BC$ and we write $\{\mu\}$ for the corresponding geometric conjugacy class of cocharacters over $\overline{\bE}$. 

Let $p$ be an odd prime and we write $G:=\bG_{\BQ_p}$ for the base change of $\bG$ to $\BQ_p$.  We let $\BA_f$ denote the ring of finite adeles and $\BA_f^p$ the finite adeles with trivial component at $p$. Let $\RK=\RK_p\RK^p\subset \bG(\BA_f)$ where $\RK_p\subset \bG(\BQ_p)$ and $\RK^p\subset \bG(\BA_f^p)$ are compact open subgroups. Then for $\RK^p$ sufficiently small $$\mathrm{Sh}_{\RK}(\bG,X)(\BC)=\bG(\BQ)\backslash X\times \bG(\BA_f)/\RK$$ arises as the complex points of an algebraic variety $\mathrm{Sh}_{\RK}(\bG,X)$ defined over $\bE$. 

\subsubsection{}\label{subsec:Hodge embedding}From now on, we will assume the datum $(\mathbf{G},X)$ is of Hodge type. This means that there exists an embedding of Shimura data $$\rho:(\mathbf{G},X)\to (\mathbf{GSp}(V,\psi),S^\pm)$$ where $(V,\psi)$ is a symplectic space over $\BQ$ and $(\mathbf{GSp}(V,\psi),S^\pm)$ is the standard Siegel Shimura datum. We will also make the following assumptions.

($\dagger$) The group $G: = \bG_{\BQ_p}$ is quasi-split and splits over a tamely ramified extension of $\BQ_p$. Moreover $p\nmid |\pi_1(G_{\mathrm{der}})|$, and $\RK_p$ is a connected very special parahoric subgroup of $G(\BQ_p)$.

Here we say a parahoric $\RK_p$ is connected if it is the same as the stabilizer of a facet in the building for $G$. When $G$ is unramified, every parahoric  which is contained in a hyperspecial parahoric is connected. In the sequel we let $\CG$ be the group scheme over $\BZ_p$ corresponding to the parahoric $\RK_p$.

Let $v$ be a prime of $\bE$ lying above $p$ with residue field $k_v=\BF_{q}$. We write $\CO$ for the ring of integers of $\bE$ and $\CO_{(v)}$ for the localization of $\CO$ at $v$. Under the assumptions above, Kisin--Pappas \cite{KP} have constructed an integral model $\sS_{\RK}(\bG,X)$ for $\mathrm{Sh}_{\RK}(\bG,X)$ over  $\CO_{(v)}$. We briefly recall the construction below.

By the discussion in \cite[\S2.3.15]{KP}, upon replacing $\rho$ with a different Hodge embedding, we may assume that there exists a $\BZ_p$-lattice $V_{\BZ_p}\subset V_{\BQ_p}$ such that $\rho$ induces a closed immersion $\mathcal G\rightarrow \GL(V_{\BZ_p})$. From now on we fix $\rho$ such that this condition is satisfied. We let $\RK'=\RK'_p\RK'^p\subset \mathbf{GSp}(V_{\BA_f})$ with $\RK'_p\subset \mathbf{GSp}(V_{\BQ_p})$ the stabilizer of $V_{\BZ_p}$ and $\RK'^p\subset \mathbf{GSp}(\BA_f^p)$ a sufficiently small compact open subgroup. By \cite[Lemma 2.1.2]{KisinIntMod}, up to shrinking $\RK^p$ we may choose a sufficiently small $\mathrm K'^p$ such that the Hodge embedding $\rho$ defines a closed immersion $$\mathrm{Sh}_{\RK}(\bG,X)\rightarrow \mathrm{Sh}_{\RK'}(\mathbf{GSp}(V),S^\pm)\otimes_{\BQ}\bE$$ of Shimura varieties. We let $V_{\BZ_{(p)}}=V_{\BZ_p}\cap V$ and we let $G_{\BZ_{(p)}}$ denote the Zariski closure of $\mathbf{G}$ in $\mathbf{GSp}(V_{\BZ_{(p)}})$. The choice of $V_{\BZ_{(p)}}$ gives rise to an interpretation of $\mathrm{Sh}_{\RK'}(\mathbf{GSp}(V),S^\pm)$ as a moduli space of abelian varieties and hence to an integral model $\sS_{\RK'}(\mathbf{GSp}(V),S^\pm)$ over $\BZ_{(p)}$; see \cite[\S4]{KP} and  \cite[\S6]{Zhou}. The integral model $\sS_{\RK}(\mathbf{G},X)$ is defined to be the normalization of the closure of $\mathrm{Sh}_{\RK}(\mathbf{G},X)$ in $\sS_{\RK'}(\mathbf{GSp}(V),S^\pm)\otimes_{\BZ_{(p)}}\CO_{(v)}$. We will write $\CA$ for the pullback of the universal abelian scheme on $\sS_{\RK'}(\mathbf{GSp}(V),S^\pm)\otimes_{\BZ_{(p)}}\CO_{(v)}$ to $\sS_{\RK}(\mathbf{G},X)$.

\subsection{Rapoport--Zink Uniformization}\subsubsection{}

We fix a maximal $\breve \BQ_p$-split $\BQ_p$-rational torus $S$ in $G$ (cf.~\S \ref{subsubsec:IW}) such that $\mathrm K_p$ corresponds to a $\s$-stable special point $\breve \fks$ in the apartment corresponding to $S$. We let $T$ denote the centralizer of $S$ and we fix $B$ a Borel subgroup of $G$ containing $T$ (which exists as we have assumed that $G$ is quasi-split). We let $\mu\in X_*(T)_{\Gamma_0}^+$ denote the image of a dominant representative $\widetilde{\mu}\in X_*(T)^+$ of $\{\mu\}$. (Here $\Gamma_0$ is as in \S \ref{subsubsec:IW} with respect to $F = \BQ_p$.) Then for $b\in B(G,\mu)$ we have the associated affine Deligne--Lusztig variety $X_\mu(b)$ as in \S\ref{subsec:ADLV} corresponding to the very special parahoric $\mathrm K_p$.

To ease notation we write $\Sh_{\RK}$ for the geometric special fiber of $\sS_{\RK}(\bG,X)$. By \cite[\S  8]{Zhou}, there exists a map $$\mathcal{N}:\Sh_{\RK}\to B(G,\mu)$$ which induces the \textit{Newton stratification} on $\Sh_{\RK}$. We let $[b]_{\mathrm{basic}}\in B(G,\mu)$ denote the unique basic $\sigma$-conjugacy class in $B(G,\mu)$ and we write $\Sh_{\RK,\mathrm{bas}}$ for the preimage of $[b]_{\mathrm{basic}}$ under $\mathcal N$. By \cite[Theorem 3.6]{RR} this is a closed subscheme of $\Sh_{\RK}$, which is known as the \textit{basic locus}. 

Our goal is to understand the set $\Sigma^{\topp} (\Sh_{\RK,\mathrm{bas}})$ of top-dimensional irreducible components of $\Sh_{\RK,\mathrm{bas}}$. This will follow from our study of $X_\mu(b)$ above and the following result, which is the analogue in our context of the Rapoport--Zink uniformization. 
\begin{proposition}\label{prop: RZ uniformization}
	Let $b\in [b]_{\mathrm{basic}}$.	There exists an isomorphism of perfect schemes $$I(\BQ)\backslash X_\mu(b)\times \mathbf{G}(\BA_f^p)/\RK^p\cong \Sh_{\RK,\mathrm{bas}}^{\pfn}$$
	where $I$ is a certain inner form of $\bG$ with $I\otimes_{\BQ}\BA_f^p\cong \bG\otimes_{\BQ}\BA_f^p$ and $I\otimes_{\BQ}\BQ_p\cong J_b$. Moreover this isomorphism is equivariant for prime-to-$p$ Hecke operators.
\end{proposition}
\begin{corollary}\label{thm: irred comp basic locus}
	There exists an identification $$\Sigma^{\topp}(\Sh_{\RK,\mathrm{bas}})\cong \coprod_{i=1}^{\mathscr N (\mu,b)} I(\BQ)\backslash I(\BA_f)/\RI_{p}^{i}\RI^p$$ where $\mathscr{N}(\mu,b)$ is as in Definition \ref{defn:N(mu,b)}, $\RI_p^{i}$ is a very special parahoric of $I(\BQ_p)$ and $\RI^p\cong \RK^p$ under a fixed identification $I\otimes_{\BQ}\BA_f^p\cong G\otimes_{\BQ}\BA_f^p$. Moreover the following statements hold. 
	\begin{enumerate}\item  The identification is compatible with prime-to-$p$ Hecke operators.
		\item If $G$ is unramified, we may replace the indexing set with $\BM\BV_\mu(\lambda_b)$.
	\end{enumerate}
\end{corollary}
\begin{proof}
	This follows from Proposition \ref{prop: RZ uniformization}, Corollary \ref{cor: main}, and the fact that the topology of a scheme is invariant under taking perfection.
\end{proof}

The rest of the section will be devoted to the proof of Proposition \ref{prop: RZ uniformization}. The case when $G$ is an unramified group is proved  in \cite[Corollary 7.2.6]{XZ}, a key input being the existence of a natural map $$X_\mu(b)(\overline{ \BF}_p)\to \underline{\mathrm{Sh}}_{\mathrm{K}}(\overline{ \BF}_p)$$ which was proved in \cite[Proposition 1.4.4]{KisinModp}. Our proposition follows similarly using results from \cite{Zhou}. We first recall some notations from \cite[\S6.2]{Zhou}.

\subsubsection{}\label{subsubsec:delta}
By construction, for a scheme $T$ over $\CO_{(v)}$, a point $x\in \sS_{\RK}(\mathbf  G,X)(T)$ gives rise to a triple $(\CA_x,\lambda,\epsilon^p_{\RK'})$ where $\CA_x$ is an abelian variety over $T$, $\lambda$ is a weak polarization (cf.~\cite[\S6.3]{Zhou}), and $\epsilon_{\RK'}^p$ is a global section of  the \'etale sheaf $$\underline{\mathrm{Isom}}_{\lambda,\psi}(\widehat{V}(\mathcal{A}_x),V_{\BA_f^p})/\RK'^p.$$ Here $\widehat{V}(\CA_x)=(\varprojlim_{p \nmid n} \CA_x[n])\otimes_{\BZ} \BQ$ is the adelic prime-to-$p$ Tate module of $\CA_x$, and we refer the reader to \textit{loc.~cit.~}for more details of the above \'etale sheaf.

For $R$ a ring and $M$ an $R$-module, we let $M^\otimes$  denote the direct sum of all $R$-modules obtained from $M$ by taking duals, tensor products, and symmetric and exterior products. By \cite[1.3.2]{KisinIntMod} and the assumption on $\rho$ in \S\ref{subsec:Hodge embedding}, the subgroup $G_{\BZ_{(p)}}$ of  $\mathbf{GSp}(V,\psi)$ is the stabilizer of a collection of tensors $s_\alpha\in V_{\BZ_{(p)}}^\otimes$. Let $k$ be a finite field or $\overline{\BF}_p$, and let $x\in\Sh_{\RK}(k)$. Then by the discussion in \cite[\S6]{Zhou}, the abelian variety $\CA_x$ is equipped with Frobenius-invariant tensors $s_{\alpha,\ell,x}\in T_\ell(\CA_x)^\otimes$ for primes $\ell\neq p$  and  $\varphi$-invariant tensors $s_{\alpha,0,x}\in\BD(\sG_x)^\otimes$. Here  $T_{\ell}(\CA_x)$ is the $\ell$-adic Tate module of $\CA_x$,  $\sG_x:=\CA_x[p^\infty]$ is the $p$-divisible group of $\CA_x$, and $\BD(\sG_x)$ is its contravariant Dieudonn\'e module. Upon fixing an isomorphism $$\BD(\sG_x)\cong V^\vee_{\BZ_p}\otimes_{\BZ_p} W(k)$$ taking $s_{\alpha,0,x}$ to $s_{\alpha}$, which exists by \cite[Proposition 3.3.8]{KP}, the Frobenius $\varphi$ is given by $\delta \sigma$ for an element $\delta\in G(W(k))[\frac{1}{p}])$ well defined up to $\sigma$-conjugation by $G(W(k))$. 

We write $\mathbb{M}$ for the $F$-crystal of the  $p$-divisible group associated to $\mathcal{A}$ over $\Sh_{\RK}$ and we let $\mathbb{M}[\frac{1}{p}]$ denote the associated isocrystal. By \cite{KMS}, there exists tensors $\mathbf{s}_{\alpha,0}\in\mathbb{M}[\frac{1}{p}]$ which specialize to $s_{\alpha,0,x}$ for all $x\in \Sh_{\RK}(\overline{\BF})$.

\subsubsection{}Now let $k=\BF_{p^r}$ be a finite extension of $k_v$. Fix  $x\in\Sh_{\RK}(k)$. For each prime $\ell \neq p$, upon fixing an isomorphism $$V_{\BQ_\ell}^\vee\cong T_\ell(\CA_x)^\vee\otimes_{\BZ_\ell} \BQ_\ell$$ taking $s_{\alpha}$ to $s_{\alpha,\ell,x}$, which exists by \cite[\S3.4.2]{KisinIntMod}, the $p^r$-Frobenius on the right is given by an element $\gamma_\ell\in \mathbf{G}(\BQ_\ell)$ well defined up to conjugation. We may and shall arrange that $(\gamma_{\ell})_{\ell \neq p}$ is an element of $\bG(\BA_f^p)$. We let $I_{\ell/k}$ denote the centralizer of $\gamma_\ell$. For sufficiently divisible $n$, the  centralizer of $\gamma_\ell^n$ stabilizes and we write $I_{\ell}$ for this centralizer. We also obtain $\delta \in G(W(k)[\frac{1}{p}])$ from $x$ as explained in \S \ref{subsubsec:delta}, and we define the $\BQ_p$-group $I_{p/k}$ whose points in a $\BQ_p$-algebra $R$ is given by $$I_{p/k}(R):=\{g\in G(W(k)[1/p]\otimes_{\BQ_p} R)\mid g^{-1}\delta\sigma(g)=\delta\}.$$ Then $I_{p/k}$ is a subgroup of $J_{\delta}$, and it grows if we keep $\delta$ fixed and let the finite field $k$ grow. Thus when $k'/k$ is a sufficiently large finite extension $I_{p/k'}$ stabilizes, and we denote it by $I_p$.  We write $\gamma_p$ for the norm $\delta\sigma(\delta)\dotsc \sigma^{r-1}(\delta)$.

Finally we define the $\BQ$-group whose points valued in a $\BQ$-algebra $R$ are given by $$\mathrm{Aut}(\CA_x\otimes_k\overline{ \BF}_p)(R)=(\mathrm{End}_\BQ(\CA_x\otimes_k\overline{ \BF}_p)\otimes_{\BQ} R)^\times$$
and we let $I\subset \mathrm{Aut}(\CA_x\otimes_k\overline{ \BF}_p)$ denote the subgroup which preserve the tensors $s_{\alpha,0,x}$ and $s_{\alpha,\ell,x}$ for all $\ell\neq p$.
We have the following facts about these groups for points $x$ in the basic locus.

\begin{proposition}\label{prop: I group} Let $k=\BF_{p^r}$ a finite extension of $k_v$ and $x\in\Sh_{\RK,\mathrm{bas}}(\BF_{p^r})$.
	
	\begin{enumerate}\item There exists $\gamma_0\in \bG(\BQ)$ which is elliptic in $\bG(\BR)$ such that $(\gamma_0,(\gamma_\ell)_{\ell\neq p},\delta)$ forms a Kottwitz triple of level $r$ in the sense of \cite[\S 4.3.1]{KisinModp}. In particular, $\gamma_0$ is $\mathbf{G}(\overline{ \BQ}_\ell)$-conjugate to $\gamma_\ell$ for all $\ell$ (including $\ell=p$).
		
		\item  For any prime $\ell$ (including $\ell=p$), the natural map $I\otimes_{\BQ}\BQ_\ell\rightarrow I_{\ell}$ is an isomorphism, and the group $(I/\BG_m)(\BR)$ is compact. Here $\BG_m\subset I$ arises from the image of the weight homomorphism of the Shimura datum $(\bG,X)$. 
		\item We write $I_0\subset\bG$ for the centralizer of $\gamma_0^n$ for sufficiently divisible $n$ such that the centralizers stabilize.  Then there exists an inner twisting  $I\otimes_{\BQ} \overline{ \BQ}\isoarrow I_{0}\otimes_{\BQ}\overline{ \BQ}$ which  makes $I$ an inner form of $I_0$ and  such that the diagram \[\xymatrix{ I_{0}\otimes_{\BQ}\overline{ \BQ}_\ell\ar[r]^\sim\ar@{=}[d] &  	I_{\ell}\otimes_{\BQ_l}\overline{ \BQ}_{\ell} \\
			I_{0}\otimes_{\BQ}\overline{ \BQ}_\ell\ar[r]^\sim & 	I\otimes_{\BQ}\overline{ \BQ}_\ell\ar[u]_\sim}\] commutes up to inner automorphism for any prime $\ell$.
		\end{enumerate}
\end{proposition}
\begin{proof}
	(1) and (2) follow from the discussion in \cite[\S9.5]{Zhou}. (3) follows from the same argument as \cite[Corollary 2.3.5]{KisinModp} using \cite[Theorem 9.4]{Zhou} in place of \cite[Theorem 2.2.3]{KisinModp}.
\end{proof}

\subsubsection{}
For $(\gamma_0,(\gamma_\ell)_{\ell\neq p},\delta)$ a Kottwitz triple of level $r$, $(\gamma_0^m,(\gamma_\ell^m)_{\ell\neq p},\delta)$ is a Kottwitz triple of level $rm$. We consider the smallest  equivalence relation on the set of all Kottwitz triples of all levels under which $(\gamma_0,(\gamma_\ell)_{\ell\neq p},\delta)$ is equivalent to $(\gamma_0^m,(\gamma_\ell^m)_{\ell\neq p},\delta)$ for all $m\geq 1$. An equivalence class under this relation is called a \emph{Kottwitz triple}. For $x\in \Sh_{\RK,{\mathrm{bas}}}(\overline{\BF}_p)$, we know that $x$ is defined over some $k=\BF_{p^r}$, and the associated Kottwitz triple $(\gamma_0,(\gamma_\ell)_{\ell\neq p},\delta)$ of level $r$ defines a Kottwitz triple which is independent of the choice of $\BF_{p^r}$.

Recall the following notion of isogeny classes introduced in \cite{Zhou}.

\begin{definition} Let $x,x'\in \underline{\mathrm{Sh}}_{\RK}(\overline{\BF}_p)$.  We say $x$ and $x'$ are \emph{isogenous} if there exists a quasi-isogeny $\CA_x\rightarrow \CA_{x'}$ which takes $s_{\alpha,\ell,x}$ to $s_{\alpha,\ell,x'}$ and $s_{\alpha,0,x}$ to $s_{\alpha,0,x'}$. 
	Clearly  this defines an equivalence relation on $\underline{\mathrm{Sh}}_{\RK}(\overline{\BF}_p)$, and the equivalence classes will be called \emph{isogeny classes}. 
\end{definition}

\subsubsection{}
We define an equivalence relation $\sim$ on the set of all Kottwitz triples by setting $\mathfrak{t}\sim\mathfrak{t}'$ for Kottwitz triples $\mathfrak{t},\mathfrak{t}'$  if there exist representatives $(\gamma_0,(\gamma_\ell)_{\ell\neq p},\delta)$, $(\gamma'_0,(\gamma'_\ell)_{\ell\neq p},\delta')$ of some level $r$ for $\mathfrak{t}$, $\mathfrak{t}'$ respectively such that \begin{enumerate}
	\item 
	$(\gamma_\ell)_{\ell\neq p}$ and $(\gamma'_\ell)_{\ell\neq p}$ are conjugate in $\bG(\BA_f^p)$ \item$\delta $ and $\delta'$ are $\sigma$-conjugate in $G(K_0)$, where $K_0=W(\BF_{p^r})[\frac{1}{p}]$.
\end{enumerate} It is easy to see that if $\mathfrak{t}$, $\mathfrak{t}'$ are the Kottwitz triples associated to points $x,x'\in \Sh_{\RK,{\mathrm{bas}}}(\overline{\BF}_p)$ lying in the same isogeny class, then $\mathfrak{t}\sim\mathfrak{t}'$.

\begin{proposition}\label{prop: basic same Kottwitz triple}
	Let $x,x'\in\Sh_{\RK,{\mathrm{bas}}}(\overline{\BF}_p)$  and let $\mathfrak{t}$ (resp.~$\mathfrak{t}'$) 
	denote the Kottwitz triple associated to $x$ (resp.~$x'$). Then $\mathfrak{t}\sim\mathfrak{t}'$. 
\end{proposition}
\begin{proof}
	We fix a sufficiently large finite field $k=\BF_{p^r}$ such that $x$ and $x'$ are both defined over $k$ and we  fix representatives $(\gamma_0,(\gamma_\ell)_{\ell\neq p},\delta)$ and  $(\gamma'_0,(\gamma'_\ell)_{l\neq p},\delta')$ of level $r$  for $\mathfrak{t}$ and $\mathfrak{t}'$ respectively. Write $I$ and $I'$ for the $\BQ$-groups associated to $x$ and $x'$ as above. We first claim that there exists $n\geq 1$ such that $\gamma_0^n$ and $\gamma_0'^n$ are central. Indeed this follows verbatim from the argument in \cite[Lemma 7.2.14]{XZ} which works without the assumption that $G$ is unramified. Therefore upon extending $k$ we may assume $\gamma_0$ and $\gamma_0'$ are both central.
	
	Let $Z^\circ$ denote the connected component of the center of $\mathbf{G}$. Upon enlarging $k$, we may assume $t:=\gamma^{-1}_0\gamma_0'\in Z^\circ(\BQ)$. We claim that the image of $t$ in $Z^\circ(\BA_f)$ lies in a compact open subgroup $H$. For $\ell\neq p$, we have $\gamma_0=\gamma_\ell$, hence $\gamma_0$ lies in a compact subgroup of $Z^\circ(\BQ_\ell)$ since $\gamma_\ell$ is the Frobenius automorphism of the $\ell$-adic Tate module.  The same argument works for $\gamma_0'$ and  hence $t$ lies in a compact open subgroup of $\mathbf{G}(\BA_f^p)$.  For $\ell=p$, we have that $\gamma_0$ and $\gamma_0'$ both have the same image in $\pi_1(G)_{\Gamma}$ since $\delta$ and $\delta'$ are both basic. Since the kernel of the map $X_*(Z^\circ)_{\Gamma_0}^{\sigma}\rightarrow \pi_1(G)_\Gamma$ is torsion, it follows that upon further extending $k$, we may assume that $\gamma$ and $\gamma_0'$ have the same image under the Kottwitz map $\kappa:Z^\circ(\BQ_p)\rightarrow X_*(Z^\circ)_{\Gamma_0}^{\sigma}$.  Thus $t$ lies in the kernel of $\kappa$ which is a compact open subgroup of $Z^{\circ}(\BQ_p)$.
	
	Since $\mathbf{G}$  and $I$ are inner forms (recall $\gamma_0$ is central), we may naturally consider $Z^\circ$ as a subgroup of $I$ which contains the scalars $\mathbb{G}_m$. Then the compactness of $(I/\BG_m)(\BR)$ implies $(Z^\circ/\BG_m)(\BR)$ is compact. It follows that $H\cap Z^\circ(\BQ)$ is finite. Hence there exists $m$ such that $\gamma_0^m=\gamma_0'^m$. Upon extending $k$, we may assume $\gamma_0=\gamma_0'$. This implies $\gamma_\ell=\gamma_\ell'$. 
	
	Now since $x,x'\in \Sh_{\RK,\mathrm{bas}}(k)$, there exists $g\in G(\breve{\BQ}_p)$ such that $g^{-1}\delta\sigma(g)=\delta'$.
	Taking norms, we obtain $$g^{-1}\gamma_0\sigma^r(g)=\gamma_0'=\gamma_0$$ and hence $g^{-1}\sigma^r(g)=1$ since $\gamma_0$ is central. This implies $g\in G(\BQ_{p^r})$ and hence $\delta$ and $\delta'$ are $\sigma$-conjugate in $G(\BQ_{p^r})$. It follows that $\mathfrak{t}\sim \mathfrak{t}'$.\end{proof}
  Proposition \ref{prop: I group} and Proposition \ref{prop: basic same Kottwitz triple} together with the Hasse principle for adjoint groups imply the following corollary.

\begin{corollary}\label{cor: basic same inner form}	Let $x,x'\in\Sh_{\RK,{\mathrm{bas}}}(\overline{\BF}_p)$. Then the groups $I$ and $I'$ are isomorphic as inner forms of $\mathbf{G}$. \qed
\end{corollary}

\begin{proposition}\label{prop: Basic locus one isogeny class}Let $x,x'\in \Sh_{\RK,\mathrm{bas}}(\overline{\BF}_p)$. Then $x$ and $x'$ lie in the same isogeny class.\end{proposition}
\begin{proof}
	Let $k=\BF_{p^r}$ be a sufficiently large finite field such that $x$ and $x'$ are both defined over $k$. We let $I$ and $I'$ be the groups associated to $x$ and $x'$ respectively. We let $\mathrm{Isog}(\CA_x,\CA_{x'})$ be the scheme of quasi-isogenies between $\CA_{x'}$ and $\CA_{x'}$. We define $$\mathcal{P}_{s_\alpha}(x,x')\subset \mathrm{Isog}(\CA_x,\CA_{x'})$$ to be the subscheme which takes $(s_{\alpha,\ell,x})_{l\neq p}$ (resp.~$s_{\alpha,0,x}$) to $(s_{\alpha,\ell,x'})_{\ell\neq p}$ (resp.~$s_{\alpha,0,x'}$). It suffices to show that $\mathcal{P}_{s_\alpha}(x,x')$ is a trivial $I$-torsor.
	
	We first show $\mathcal{P}_{s_\alpha}(x,x')$ is an $I$-torsor. By Corollary \ref{cor: basic same inner form}, we may fix an isomorphism $I\cong I'$. Let $\bT\subset I\cong I'$ be a maximal torus. The proof of \cite[Theorem 9.4]{Zhou} shows that upon modifying $x$ and $x'$ in its isogeny class, we may assume that $x$ and $x'$ admit lifts $\tilde{x}$ and $\tilde{x}'$ to $\mathrm{Sh}_{\RK}(\mathbf{G},X)(\overline{\BQ})$ satisfying the conditions:
	
	\begin{enumerate}	\item  $\bT\subset \mathrm{Aut}(\CA_x)$ and $\bT\subset\mathrm{Aut}(\CA_{x'})$ lift to $\bT\subset \mathrm{Aut}(\CA_{\tilde{x}})$ and $\bT\subset \mathrm{Aut}(\CA_{\tilde{x}'})$. 
		
		\item  The Hodge filtrations on $\mathrm{H}^1_{\mathrm{dR}}(\mathcal{A}_{\tilde{x}})$ and $\mathrm{H}^1_{\mathrm{dR}}(\mathcal{A}_{\tilde{x}'})$ are induced by the same $\bT$-valued cocharacter $\mu^{\bT}$.
		
		\item If $i, i':\bT\rightarrow \mathbf{G}$ are the inclusions obtained by regarding $\bT$ as a subgroup of the Mumford--Tate groups of $\CA_{\tilde{x}}$ and $\CA_{\tilde{x}'}$ (these are well-defined up to $\mathbf{G}(\BQ)$-conjugacy), then $\tilde{x}$ and $\tilde{x}'$ are in the images of the maps $$i:\mathrm{Sh}(\bT,h_{\bT})\rightarrow \mathrm{Sh}_{\RK}(\mathbf{G},X)_{\bE_{\bT}}$$ 
		$$i':\mathrm{Sh}(\bT,h_{\bT})\rightarrow \mathrm{Sh}_{\RK}(\mathbf{G},X)_{\bE_{\bT}}$$
		respectively.  Here $\mathrm{Sh}(\bT,h_{\bT})$ is the Shimura variety for $(\bT, h_{\bT})$ and  $\bE_{\bT}$ is its reflex field.
	\end{enumerate}
	We let $\tilde{\mathcal{P}}\subset \mathrm{Isog}(\CA_{\tilde{x}},\CA_{\tilde{x}'})$ be the scheme of isogenies which respect the Hodge cycles and the action of $\bT$. We claim that $\tilde{\mathcal{P}}$ is a $\bT$-torsor; for this it suffices to show that $\tilde{\mathcal{P}}$ is non-empty.
	
	By Proposition \ref{prop: I group}, the map $$\bT\xrightarrow{i}\bG\otimes_{\BQ}\overline{ \BQ}\cong I\otimes_{\BQ}\overline{ \BQ}$$
	is conjugate to the natural inclusion, and a similar statement holds for the map $$\bT\xrightarrow{i'}\bG\otimes_{\BQ}\overline{ \BQ}\cong I'\otimes_{\BQ}\overline{ \BQ}.$$
	It follows that there exists $g\in \mathbf{G}(\overline{ \BQ})$ such that $gig^{-1}=i'$. Since $i(\bT)$ is  its own centralizer in $\mathbf{G}$, we have $c_\tau=g^{-1}\tau(g)\in i(\bT)(\overline{ \BQ})$ for any $\tau\in\mathrm{Gal}(\overline{ \BQ}/\BQ)$. Let $\RK_\infty$ denote the centralizer of $i\circ h_T$. Then by the same argument as in \cite[Proposition 4.4.13]{KisinModp}, the image of $c$  in $ \mathrm{H}^1(\BR,\RK_{\infty})$ is trivial.
	
	This defines a $\bT$-torsor $\tilde{\mathcal{P}}'$ which is  isomorphic to $\tilde{\mathcal{P}}$ by \cite[Proposition 4.2.6]{KisinModp}. Indeed the proposition in loc. cit. shows that $\mathcal{A}_{\tilde{x}'}$ is isomorphic to the twisted abelian variety $\mathcal{A}_{\tilde{x}}^{\tilde{\mathcal{P}}'}$ as in \cite[\S4.1]{KisinModp} equipped with its collection of Hodge cycles and action of $\bT$ induced from $\mathcal{A}_{\tilde{x}}$. It then follows by the construction of $\mathcal{A}_{\tilde{x}}^{\tilde{\mathcal{P}}'}$ that $\tilde{\mathcal{P}}\cong\tilde{\mathcal{P}}'$.
	It follows that $\mathcal{P}_{s_{\alpha}}$ is the $I$-torsor obtained by pushout from the $\bT$-torsor $\tilde{\mathcal{P}}$.  
	
	By \cite[Lemma 4.4.3]{KisinModp}, there is an isomorphism $$\ker(\RH^1(\BQ,I)\rightarrow \RH^1(\BR,I))\cong \ker(\RH^1(\BQ,\mathbf G)\rightarrow \RH^1(\BR,\mathbf G)).$$ 
	By \cite[Lemma 4.4.5]{KisinModp} applied to the inclusion $\bT_{\BR}\rightarrow K_\infty$, the image of $c$ in $\RH^1(\BR,\bT)$ is trivial, and hence the image of $c$ in  $\RH^1(\BQ,I)$ lies in  $\ker(\RH^1(\BQ,I)\rightarrow \RH^1(\BR,I))$.  Since the image of $c$  in $\RH^1(\BQ,\mathbf G)$ is trivial, we have that $c$ is trivial in $\RH^1(\BQ,I)$.   It follows that the $I$-torsor $\mathcal{P}_{s_\alpha}(x,x')$ is trivial.
\end{proof}

\begin{proof}[Proof of Proposition \ref{prop: RZ uniformization}] Let $x\in\Sh_{\RK,\mathrm{bas}} (\overline{ \BF}_p)$. We first define a natural map $ X_\mu(\delta)\rightarrow \Sh_{\RK,\mathrm{bas}}^{\mathrm{pfn}}.$ The key input for this is the existence of such a map on $\overline{ \BF}_p$-points which was constructed in  \cite[Proposition 7.7]{Zhou}.  We may then argue as in \cite[Lemma 7.2.12]{XZ}; we sketch the argument  emphasizing the points which do not directly carry over to the ramified case. 
	
	As in \cite[7.2.6]{XZ}, we may construct an abelian variety $\CA$ over $X_\mu(b)$ equipped with a $p$-power quasi isogeny $\CA\rightarrow\CA_x\times X_\mu(b) $. Moreover this quasi-isogeny equips $\CA$ with tensors $\mathbf{s}'_{\alpha,0}\in \BD(\CA[p^\infty])^\otimes$, as well as a weak polarization and a prime-to-$p$ level structure. Hence we obtain a map $$\iota':X_\mu(b)\to \sA_{g,K'}.$$ We claim $\iota'$ lifts to a unique map $$\iota:X_\mu(b)\to \Sh_{\RK}^{\mathrm{pfn}}$$ such that for each closed point $y\in X_\mu(b)$, we have $s_{\alpha,0,y}=\mathbf{s}'_{\alpha,0,y}$.  The uniqueness follows from \cite[Corollary 6.3]{Zhou} and the fact that two maps between perfect schemes coincide if and only if they coincide on the level of closed points. Thus it suffices to prove the lifting locally. 
	
	Let $y$ be a closed point of $X_\mu(b)$ and $U\subset X_\mu(b)$ an affine open neighborhood containing $y$ which is perfectly of finite presentation. We may assume $U$ is the perfection  of a reduced affine scheme $U_0\cong\Spec R$ and that the quasi-isogeny $\CA|_U\rightarrow \CA_x\times U$ comes from pullback from a quasi-isogeny $\CA_0\rightarrow\CA_x\times U_0$ over $U_0$. We thus obtain a map $$\iota_0':U_0\rightarrow\sS_{\RK'}(\mathbf{GSp}(V),S^\pm)\otimes_{\BZ_{(p)}}\overline{\BF}_p$$ and it suffices to show $\iota'_0$ can be lifted to $\iota:U_0\rightarrow \Sh_{\RK}$. 
	
	We form	the pullback diagram
	\[\xymatrix{Y \ar[r]\ar[d]& \Sh_{\RK} \ar[d]\\
		\Spec R\ar[r]& \sS_{\RK'}(\mathbf{GSp}(V),S^\pm)\otimes_{\BZ_{(p)}}\overline{\BF}_p	}.\]
	
	Then $Y$ is equipped with a polarized abelian variety $(\mathcal{A}_Y,\lambda_Y)$ and tensors $$\mathbf{s}'_{\alpha,0,Y}\in\mathbb{D}(\mathcal{A}[p^\infty])[\frac{1}{p}]^\otimes, \ \ \ \mathbf{s}_{\alpha,0,Y}\in\mathbb{D}(\mathcal{A}[p^\infty])[\frac{1}{p}]^\otimes,$$ where the $\mathbf{s}'_{\alpha,0,Y}$ are obtained from pullback of $\mathbf{s}_{\alpha,0}'$ along $Y\rightarrow \Spec R$, and the $\mathbf{s}_{\alpha,0,Y}$ are obtained from pullback of $\mathbf{s}_{\alpha,0}$ along $Y\rightarrow \Sh_{\RK}$. We let $Y^\circ$ denote the union of connected components which contain an $\overline{ \BF}_p$-point $y$ such that $\mathbf{s}_{\alpha,0,y}=\mathbf{s}'_{\alpha,0,y}$. By \cite[Lemma 5.10]{MP}, $\mathbf{s}_{\alpha,0,Y^\circ}=\mathbf{s}'_{\alpha,0,Y^\circ}$. By \cite[Proposition 6.5 (i)]{Zhou}, the map $Y^\circ\rightarrow \Spec R$ is bijective on $\overline{ \BF}_p$-points and by \cite[Proposition 4.2.2]{KP}, the map $Y^\circ\rightarrow \Spec R$ is finite and is a closed immersion when completed at every point of the domain. In addition $R$ is reduced; it follows that $Y^\circ \rightarrow\Spec R$ is an isomorphism.
	
	The map $\iota$ induces  a finite map $$\iota_{\mathrm{isog}}:I(\BQ)\backslash X_\mu(b)\times \bG(\BA_f^p)/\RK^p\to \Sh_{\RK,b}^{\mathrm{pfn}}$$ which is bijective on closed points by \cite[Proposition 9.1]{Zhou} and Proposition \ref{prop: Basic locus one isogeny class}, and is a closed immersion when completed at every closed point of the domain. It follows that $\iota_{\mathrm{isog}}$ is an isomorphism.

\end{proof}

\section*{Erratum for \cite{ZZ}} In \cite[Definition 5.2.7]{ZZ}, the two appearances of $\mathbb C[Y^*]$ should be replaced by $\mathbb C$. In \cite[Proposition 5.5.1]{ZZ}, the identity is between two elements of $\mathbb C[\mathbf q^{-1}]$.

\bibliographystyle{hep}
\bibliography{myref}

\begin{thebibliography}{GHKR06}

\bibitem[BP89]{Borel-Prasad}
A.~Borel and G.~Prasad, \textsl{ Finiteness theorems for discrete subgroups of
  bounded covolume in semi-simple groups},
\newblock Inst. Hautes \'{E}tudes Sci. Publ. Math. (69), 119--171 (1989).

\bibitem[BS17]{BS}
B.~Bhatt and P.~Scholze, \textsl{ Projectivity of the {W}itt vector affine
  {G}rassmannian},
\newblock Invent. Math. \textbf{ 209}(2), 329--423 (2017).

\bibitem[BT72]{BT1}
F.~Bruhat and J.~Tits, \textsl{ Groupes r\'{e}ductifs sur un corps local},
\newblock Inst. Hautes \'{E}tudes Sci. Publ. Math. (41), 5--251 (1972).

\bibitem[BT84]{BT2}
F.~Bruhat and J.~Tits, \textsl{ Groupes r\'{e}ductifs sur un corps local. {II}.
  {S}ch\'{e}mas en groupes. {E}xistence d'une donn\'{e}e radicielle
  valu\'{e}e},
\newblock Inst. Hautes \'{E}tudes Sci. Publ. Math. (60), 197--376 (1984).

\bibitem[CKV15]{CKV}
M.~Chen, M.~Kisin and E.~Viehmann, \textsl{ Connected components of affine
  {D}eligne-{L}usztig varieties in mixed characteristic},
\newblock Compos. Math. \textbf{ 151}(9), 1697--1762 (2015).

\bibitem[Gas10]{Ga}
Q.~R. Gashi, \textsl{ On a conjecture of {K}ottwitz and {R}apoport},
\newblock Ann. Sci. \'{E}c. Norm. Sup\'{e}r. (4) \textbf{ 43}(6), 1017--1038
  (2010).

\bibitem[GH10]{GoHe}
U.~G\"{o}rtz and X.~He, \textsl{ Dimensions of affine {D}eligne-{L}usztig
  varieties in affine flag varieties},
\newblock Doc. Math. \textbf{ 15}, 1009--1028 (2010).

\bibitem[GHKR06]{GHKR}
U.~G\"ortz, T.~J. Haines, R.~E. Kottwitz and D.~C. Reuman, \textsl{ Dimensions
  of some affine {D}eligne-{L}usztig varieties},
\newblock Ann. Sci. \'Ecole Norm. Sup. (4) \textbf{ 39}(3), 467--511 (2006).

\bibitem[Gre63]{Greenberg}
M.~J. Greenberg, \textsl{ Schemata over local rings. {II}},
\newblock Ann. of Math. (2) \textbf{ 78}, 256--266 (1963).

\bibitem[Ham15]{Ham}
P.~Hamacher, \textsl{ The dimension of affine {D}eligne-{L}usztig varieties in
  the affine {G}rassmannian},
\newblock Int. Math. Res. Not. IMRN (23), 12804--12839 (2015).

\bibitem[He14]{He14}
X.~He, \textsl{ Geometric and homological properties of affine
  {D}eligne-{L}usztig varieties},
\newblock Ann. of Math. (2) \textbf{ 179}(1), 367--404 (2014).

\bibitem[He15]{He-zero}
X.~He,
\newblock Centers and cocenters of 0-{H}ecke algebras,
\newblock in \textsl{ Representations of reductive groups}, volume 312 of
  \textsl{ Progr. Math.}, pages 227--240, Birkh\"{a}user/Springer, Cham, 2015.

\bibitem[He16]{He-CDM}
X.~He,
\newblock Hecke algebras and {$p$}-adic groups,
\newblock in \textsl{ Current developments in mathematics 2015}, pages 73--135,
  Int. Press, Somerville, MA, 2016.

\bibitem[HN14]{HN}
X.~He and S.~Nie, \textsl{ Minimal length elements of extended affine {W}eyl
  groups},
\newblock Compos. Math. \textbf{ 150}(11), 1903--1927 (2014).

\bibitem[HP17]{HP}
B.~Howard and G.~Pappas, \textsl{ Rapoport-{Z}ink spaces for spinor groups},
\newblock Compos. Math. \textbf{ 153}(5), 1050--1118 (2017).

\bibitem[HR08]{HainesRapoport}
T.~Haines and M.~Rapoport, \textsl{ On parahoric subgroups},
\newblock Adv. Math. \textbf{ 219}(1), 118--198 (2008),
\newblock Appendix to: G. Pappas and M. Rapoport, Twisted loop groups and their
  affine flag varieties.

\bibitem[HR17]{HR}
X.~He and M.~Rapoport, \textsl{ Stratifications in the reduction of {S}himura
  varieties},
\newblock Manuscripta Math. \textbf{ 152}(3-4), 317--343 (2017).

\bibitem[HV18]{HV}
P.~Hamacher and E.~Viehmann, \textsl{ Irreducible components of minuscule
  affine {D}eligne-{L}usztig varieties},
\newblock Algebra Number Theory \textbf{ 12}(7), 1611--1634 (2018).

\bibitem[HV20]{HVfinite}
P.~Hamacher and E.~Viehmann, \textsl{ Finiteness properties of affine
  {D}eligne-{L}usztig varieties},
\newblock Doc. Math. \textbf{ 25}, 899--910 (2020).

\bibitem[HY12]{HY}
X.~He and Z.~Yang, \textsl{ Elements with finite {C}oxeter part in an affine
  {W}eyl group},
\newblock J. Algebra \textbf{ 372}, 204--210 (2012).

\bibitem[Kis10]{KisinIntMod}
M.~Kisin, \textsl{ Integral models for {S}himura varieties of abelian type},
\newblock J. Amer. Math. Soc. \textbf{ 23}(4), 967--1012 (2010).

\bibitem[Kis17]{KisinModp}
M.~Kisin, \textsl{ Mod $p$ points on {S}himura varieties of abelian type},
\newblock J. Amer. Math. Soc. \textbf{ 30}(2), 819--914 (2017).

\bibitem[KMPS15]{KMS}
M.~Kisin, K.~Madapusi~Pera and S.~W. Shin, \textsl{ Honda--{T}ate theory for
  {S}himura varieties},
\newblock preprint  (2015).

\bibitem[Kot88]{kottTama}
R.~E. Kottwitz, \textsl{ Tamagawa numbers},
\newblock Ann. of Math. (2) \textbf{ 127}(3), 629--646 (1988).

\bibitem[Kot97]{kottwitzisocrystal2}
R.~E. Kottwitz, \textsl{ Isocrystals with additional structure. {II}},
\newblock Compositio Math. \textbf{ 109}(3), 255--339 (1997).

\bibitem[Kot06]{Ko06}
R.~E. Kottwitz, \textsl{ Dimensions of {N}ewton strata in the adjoint quotient
  of reductive groups},
\newblock Pure Appl. Math. Q. \textbf{ 2}(3, Special Issue: In honor of Robert
  D. MacPherson. Part 1), 817--836 (2006).

\bibitem[KP18]{KP}
M.~Kisin and G.~Pappas, \textsl{ Integral models of {S}himura varieties with
  parahoric level structure},
\newblock Publ. Math. Inst. Hautes {\'E}tudes Sci. \textbf{ 128}(1), 121--218
  (2018).

\bibitem[LMPT19]{LMPT}
W.~Li, E.~Mantovan, R.~Pries and Y.~Tang, \textsl{ Newton polygons arising from
  special families of cyclic covers of the projective line},
\newblock Research in Number Theory \textbf{ 5}(12) (2019).

\bibitem[LT20]{LT}
Y.~Liu and Y.~Tian, \textsl{ Supersingular locus of {H}ilbert modular
  varieties, arithmetic level raising and {S}elmer groups},
\newblock Algebra Number Theory \textbf{ 14}(8), 2059--2119 (2020).

\bibitem[Lus03]{Lusztig03}
G.~Lusztig,
\newblock \textsl{ Hecke algebras with unequal parameters}, volume~18 of
  \textsl{ CRM Monograph Series},
\newblock American Mathematical Society, Providence, RI, 2003.

\bibitem[MP16]{MP}
K.~Madapusi~Pera, \textsl{ Integral canonical models for spin {S}himura
  varieties},
\newblock Compos. Math. \textbf{ 152}(4), 769--824 (2016).

\bibitem[{Nie}18a]{nienew}
S.~{Nie}, \textsl{ {Irreducible component of affine Deligne-Lusztig
  varieties}},
\newblock ArXiv e-prints  (September 2018), {1809.03683}.

\bibitem[{Nie}18b]{nie}
S.~{Nie}, \textsl{ {Semi-modules and irreducible components of affine
  Deligne-Lusztig varieties}},
\newblock ArXiv e-prints  (February 2018), {1802.04579}.

\bibitem[PR08]{PR}
G.~Pappas and M.~Rapoport, \textsl{ Twisted loop groups and their affine flag
  varieties},
\newblock Adv. Math. \textbf{ 219}(1), 118--198 (2008).

\bibitem[Rap05]{RapGuide}
M.~Rapoport, \textsl{ A guide to the reduction modulo {$p$} of {S}himura
  varieties},
\newblock Ast\'{e}risque (298), 271--318 (2005),
\newblock Automorphic forms. I.

\bibitem[Ric16]{Richarz}
T.~Richarz, \textsl{ On the {I}wahori {W}eyl group},
\newblock Bull. Soc. Math. France \textbf{ 144}(1), 117--124 (2016).

\bibitem[Rou77]{rouss}
G.~Rousseau,
\newblock \textsl{ Immeubles des groupes r\'{e}ducitifs sur les corps locaux},
\newblock U.E.R. Math\'{e}matique, Universit\'{e} Paris XI, Orsay, 1977,
\newblock Th\`ese de doctorat, Publications Math\'{e}matiques d'Orsay, No.
  221-77.68.

\bibitem[RR96]{RR}
M.~Rapoport and M.~Richartz, \textsl{ On the classification and specialization
  of {$F$}-isocrystals with additional structure},
\newblock Compositio Math. \textbf{ 103}(2), 153--181 (1996).

\bibitem[Tit79]{Tits}
J.~Tits,
\newblock Reductive groups over local fields,
\newblock in \textsl{ Automorphic forms, representations and {$L$}-functions
  ({P}roc. {S}ympos. {P}ure {M}ath., {O}regon {S}tate {U}niv., {C}orvallis,
  {O}re., 1977), {P}art 1}, Proc. Sympos. Pure Math., XXXIII, pages 29--69,
  Amer. Math. Soc., Providence, R.I., 1979.

\bibitem[Vie06]{viehmanndim}
E.~Viehmann, \textsl{ The dimension of some affine {D}eligne-{L}usztig
  varieties},
\newblock Ann. Sci. \'Ecole Norm. Sup. (4) \textbf{ 39}(3), 513--526 (2006).

\bibitem[VW11]{VW}
I.~Vollaard and T.~Wedhorn, \textsl{ The supersingular locus of the Shimura
  variety of $GU(1,n−1)$ II},
\newblock Invent. Math. \textbf{ 184}, 591--627 (2011).

\bibitem[XZ17]{XZ}
L.~{Xiao} and X.~{Zhu}, \textsl{ {Cycles on Shimura varieties via geometric
  Satake}},
\newblock arXiv e-prints  (July 2017), {1707.05700}.

\bibitem[Zho20]{Zhou}
R.~Zhou, \textsl{ Mod p isogeny classes on {S}himura varieties with parahoric
  level},
\newblock Duke Math. J. \textbf{ 169}(15), 2937--3031 (2020).

\bibitem[Zhu15]{Zhuram}
X.~Zhu, \textsl{ The geometric {S}atake correspondence for ramified groups},
\newblock Ann. Sci. \'{E}c. Norm. Sup\'{e}r. (4) \textbf{ 48}(2), 409--451
  (2015).

\bibitem[Zhu17]{Zhu}
X.~Zhu, \textsl{ Affine {G}rassmannians and the geometric {S}atake in mixed
  characteristic},
\newblock Ann. of Math. (2) \textbf{ 185}(2), 403--492 (2017).

\bibitem[ZZ20]{ZZ}
R.~Zhou and Y.~Zhu, \textsl{ Twisted orbital integrals and irreducible
  components of affine {D}eligne-{L}usztig varieties},
\newblock Camb. J. Math. \textbf{ 8}(1), 149--241 (2020).

\end{thebibliography}
\end{document}